%% file: main.tex
\documentclass[a4paper]{amsart}
\usepackage[utf8]{inputenc}
\usepackage{preamble}
\usepackage{epigraph}

\title{Open $2$D TFTs admit initial open-closed extensions}

\author{Shaul Barkan}
\address{Hebrew University of Jerusalem, Israel}
\email{shaul.barkan@mail.huji.ac.il}

\author{Jan Steinebrunner}
\address{Gonville \& Caius College, Cambridge, UK}
\email{js2675@cam.ac.uk}

\author{Adela YiYu Zhang}
\address{Department of Mathematical Sciences, University of Copenhagen, Denmark}
\email{yz@math.ku.dk}

\usepackage[style=alphabetic, 
maxbibnames=5, backend=biber, 
doi=false, isbn=false]{biblatex}
\begin{document}
\date{}
\maketitle
\begin{abstract}
    We show that any open $2$-dimensional topological field theory valued in a symmetric monoidal \category{} (with suitable colimits)
    extends canonically to an open-closed field theory whose value at the circle is the Hochschild homology object of its value at the disk. 
    As a corollary, we obtain an action of the moduli spaces of surfaces on the Hochschild homology object of $\Erm_1$-Calabi--Yau algebras. 
    This provides a space level refinement of previous work of Costello over $\mathbb{Q}$, and Wahl--Westerland and Wahl over $\mathbb{Z}$, and serves as a crucial ingredient to Lurie's ``non-compact cobordism hypothesis'' in dimension $2$.
    As part of the proof we also give a description of slice categories of the $d$-dimensional bordism category with boundary, which may be of independent interest.
\end{abstract}


\setcounter{tocdepth}{1}
\tableofcontents

\input{sections/introduction}

\input{sections/slice}
\input{sections/prelim}
\input{sections/calculus}

\input{sections/cyclic}

\input{sections/arcs}
\input{sections/main-theorem}

\input{sections/examples}

\appendix
\input{sections/appendix}

\printbibliography

\end{document}

%% file: sections/introduction.tex
\section{Introduction}
Let $(\Vcal,\otimes, \mathbf{1})$ be a symmetric monoidal \category{} with geometric realizations, such as the derived category of some commutative ring, or more generally any presentably symmetric monoidal \category{}.
For $A$ an $\Erm_1$-algebra in $\Vcal$, we can form the \hldef{Hochschild homology object $\mrm{HH}(A)$} by taking the geometric realization of the cyclic bar construction, which is equivalent to computing the factorization homology $\int_{S^1}A$. For example, when $\Vcal = \mrm{Sp}$ is the category of spectra, this gives the topological Hochschild homology of a ring spectrum.

In this paper, we determine natural operations on $\int_{S^1}A$ when $A$ has the structure of a higher-categorical analogue of symmetric Frobenius algebras. 
More precisely, we say that an \hldef{$\Erm_1$-Calabi--Yau algebra} in $\Vcal$ is a pair $(A, \tau)$ of an $\Erm_1$-algebra $A$ and an $\SO(2)$-invariant map $\tau\colon \int_{S^1}A\to \mathbf{1}$, called the \hldef{cyclic trace}, such that the pairing
\[
    A\otimes A\xto{\simeq} \int_{D^1\sqcup D^1}A\to \int_{S^1}A\xto{\tau}\mathbf{1}
\]
is nondegenerate and thus exhibits a self-duality of $A$.\footnote{This definition appears as a variant of the notion of a Frobenius algebra (in a symmetric monoidal \category{}) in the sense of Lurie \cite[Remark 4.6.5.9]{HA}.}

We will construct an action of the moduli spaces of surfaces on $\int_{S^1} A$.
For $n,g,k \ge 0$, let $\Sigma_{g,k}^n$ be the genus $g$ surfaces with $k$ boundaries and $n$ punctures,
and let $\Diff_\partial(\Sigma_{g,k}^n)$ be the topological group of those diffeomorphisms that fix the boundary pointwise.

\begin{thm}\label{thm:operation}
    Suppose that the symmetric monoidal product in $\Vcal$ preserves geometric realization in each variable. Let $A$ be an $\Erm_1$-Calabi--Yau algebra in $\Vcal$. 
    There are maps of spaces
    \[
        B\Diff_\partial(\Sigma^n_{g,i+j}) \to \Map_{\Vcal}\left(\Big(\int_{S^1}A\Big)^{\otimes i},\Big(\int_{S^1}A\Big)^{\otimes j}\right)
    \]
    for any $i>0$ and $j\geq 0$, which assemble into a symmetric monoidal $\infty$-functor
    \[
        \Bord_2^{\partial_+} \too \Vcal
    \]
    from the positive-boundary surface bordism category that sends $S^1$ to $\int_{S^1} A$.
    In particular, $\int_{S^1} A$ is a non-unital $\Erm_2^{\rm fr}$-algebra and an $\Erm_2^{\rm fr}$-algebra in $\Vcal$, and it is a non-unital commutative Frobenius algebra in the homotopy category $h(\Vcal)$.
\end{thm}

This formulation of the theorem relies on work in preparation of the first two authors \cite{cyclic}, but our main theorem, \cref{thm:main}, does not.
We will in fact show in \cref{cor:formal-operations} that the moduli spaces appearing in \cref{thm:operation} parametrize universal natural operations, i.e., the space of ``formal operations'' of the form $(\int_{S^1} A)^{\otimes i} \to (\int_{S^1} A)^{\otimes j}$ 
is equivalent to $\coprod_{n,g\ge 0} B\Diff_\partial(\Sigma^n_{g,i+j})$ for $i>0$.

\begin{rem}
    \cref{thm:operation} was proved by Costello when $\Vcal=D(k)\simeq\mathrm{Mod}_k$ is the derived category of chain complexes over a field $k$ of characteristic 0 \cite{costello2007topological} and by Wahl--Westerland for $\Vcal=D(\mathbb{Z})$ \cite{Wahl2016}. 
    Wahl also computed the chain complex of formal operations in the $\mathbb{Z}$-linear case \cite{wahl2016universal}.
    Our result can be viewed as space-level refinement of their results, albeit through a more categorical and conceptual route.
\end{rem}

\cref{thm:operation} follows from a qualitative understanding of the relation between certain variants of the 2-dimensional bordism category as well as symmetric monoidal functors out of them. We will give precise statements in \cref{thm:main} and \cref{thm:O-dense-in-OC-intro}. For now, we first record some examples of $\Erm_1$-Calabi--Yau algebras and explicitly identify their Hochschild homology object when possible. We will provide more detailed explanations in \cref{sec:app}.
\begin{example}
\begin{enumerate}[(a)]
    \item For $\Vcal=\mathrm{Mod}_R$, the $R$-valued cochains $C^*(M;R)$ on $M$ has the structure of an $\mrm{E}_1$-Calabi--Yau algebra, where $R$ is an even-periodic ring spectrum and $M$ an $R$-oriented, even-dimensional, closed manifold. The cyclic trace is provided by a lift of the Poincar\'{e} duality pairing. 
    \item  As a variant of (a), suppose that $R$ is the Eilenberg--MacLane spectrum of a commutative ring and $M$ is an $R$-oriented even-dimensional, simply connected, closed manifold.
   We adjoin an invertible generator $t$ in degree 2 to $R$ and obtain an $\Erm_\infty$-ring $R[t^{\pm1}]$. 
   Then $C^*(M;R)[t^{\pm1}] \simeq C^*(M;R[t^{\pm1}])$  has the structure of an $\Erm_1$-Calabi--Yau algebra in $\mathrm{Mod}_{R[t^{\pm1}]}$ (and in fact an $\Erm_\infty$-Calabi--Yau algebra) and
    \[
      F'(S^1)=\int_{S^1}C^*(M;R)[t^{\pm1}]=C^*(\mathcal{L}M;R)[t^{\pm1}]
     \]
     is the cochain algebra on the free loop space on $M$, see \cite[Proposition 5.3]{ayala2015factorization} and \cite{jones1987cyclic,ungheretti2017free}. 
     We expect that in this case the operations from \cref{thm:operation} recover the classical string operations as in \cite{Wahl2016}.
    \item Suppose that $A\to B$ is a $G$-Galois extension of $\Erm_\infty$-ring spectra for a finite group $G$. Then $B$ is an $\mrm{E}_1$-Calabi--Yau algebra in $\Vcal=\mathrm{Mod}_A$, with the cyclic trace given by a trace pairing that exhibits $B$ as its $A$-linear self dual {\cite{rognes2008galois}}.
    Here $\int_{S^1} B \simeq B$.
    \item Consider $\Vcal=\mathrm{Lex}^f$, the 2-category of finite $k$-linear 1-categories over an algebraically closed field $k$ with 1-morphisms left exact functors and the 2-morphisms linear natural transformations. It was shown in \cite{muller2024categorified} that the groupoid of $\Erm_1$-Calabi--Yau algebras in $\mathrm{Lex}^f$ is equivalent to the $2$-groupoid of pivotal Grothendieck--Verdier categories (cf. \cite{Boyarchenko2013}, which 
    expands work of Barr \cite{Barr1979} on $*$-autonomous categories).
    In the case when the $\Erm_1$-Calabi--Yau algebra is a pivotal finite tensor category $\Pcal$ \cite{etingof2004finite}, there is a canonical identification of $\int_{S^1}\Pcal$ with the Drinfeld center $Z(\Pcal)$ of $\Pcal$, see e.g.~\cite[Theorem 5.9]{muller2023lyubashenko}.
\end{enumerate}
\end{example}

\subsubsection{Bordism categories and field theories}

In order to state the main theorem of this paper, we
recall that \hldef{$\Bord_2^\partial$} is the symmetric monodial $(\infty,1)$-category whose objects are compact oriented 1-manifolds with boundary. 
A morphism from $M$ to $N$ is a 2-bordism $W$ with corners. In particular, the boundary $\partial W$ is the union of the incoming boundary $M$, the outgoing boundary $N$, and possibly nonempty free boundary \hldef{$\pf W$}$=\overline{\partial W-M\sqcup N}$. Furthermore, the corners of $W$ are precisely the intersection $\pf W\cap(M\sqcup N)$. 
The higher morphisms in this $\infty$-category encode diffeomorphisms, isotopies between them, and so on.
The symmetric monoidal structure is given by the disjoint union of manifolds and bordisms. 
A precise definition of $\Bord_2^\partial$ can be found in \cref{sec:bord}.

In particular, the oriented 2-bordism category $\Bord^{\mrm{or}}_2$ is a non-full subcategory of $\Bord_2^\partial$ on disjoint union of circles.  Below is a morphism $W$ from $D^1$ to $S^1$. The free boundary $\pf W$ is the interval colored in red, and the corners of $W$ are the green points that lie in the intersection of the incoming $D^1$ boundary and $\pf W$.
 \begin{figure}[ht]
\centering
\includegraphics[width=4cm]{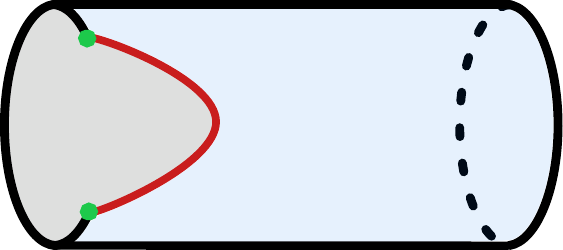}
\caption{The ``whistle'' bordism $D^1 \to S^1$ in $\Bord_2^\partial$.}
\label{fig:bordex}
\end{figure}

\begin{defn}\label{defn:OC}
    We let $\hldef{\OC} \subset \Bord_2^\partial$ denote the subcategory, called the \hldef{open-closed bordism category} that has all objects but only those bordisms $W\colon M \to N$ for which the subspace $M \cup \partial_{\rm free} W \subset W$ intersects all connected components of $W$.
    We let $\hldef{\Ocal} \subset \OC$ denote the full subcategory, called the \hldef{open bordism category}, on objects of the form $\sqcup_k D^1$ for $k\ge 0$.
\end{defn}
Geometrically, the boundary condition on the bordisms in $\OC$ is equivalent to the requirement that a bordism has handle dimension at most $1$ relative to its outgoing boundary.

In particular, there is a (faithful) symmetric monoidal functor $\Disk_1\hookrightarrow\Ocal$ that sends $D^1$ to $D^1$ and embeddings of disks to flat pairs of pants, therefore equipping $D^1\in\Ocal$ with the structure of an $\Erm_1$-algebra.

In forthcoming work, the first and second-named authors establish the following classification of symmetric monoidal functors $\Ocal \to \Vcal$, also know as open topological field theories valued in $\Vcal$, by describing them as algebras over the free modular \operad{} on the cyclic \operad{} $\Erm_1$, generalizing Costello's result for $\Vcal=D(\mathbb{Q})$ {\cite[Theorem A.1]{costello2007topological}}.
The analogue of this result in the case where $\Vcal$ is a symmetric monoidal bicategory already follows from the work of M\"uller--Woike \cite[Proof of Theorem 2.2]{muller2024categorified}.

\begin{thm}[\cite{cyclic}]\label{thm:O->V}
    There is an equivalence between the $\infty$-groupoids of symmetric monoidal functors $\Ocal \to \Vcal$ and $\Erm_1$-Calabi--Yau algebras in $\Vcal$. The equivalence is implemented by evaluating at $D^1\in\Ocal$.
\end{thm}
It is then natural to ask when and how a symmetric monoidal functor $\Ocal \to \Vcal$ can be extended to a symmetric monoidal functor $\OC\to \Vcal$, also known as an open-closed topological field theory or a topological conformal field theory. The main theorem of our paper gives an affirmative answer, which generalizes Costello's result for $\Vcal=D(\mathbb{Q})$ \cite[Theorem A.II]{costello2007topological} and Wahl-Westerland for $\Vcal=D(\mathbb{Z})$ \cite[Theorem 6.2]{Wahl2016}.
In the case of $\Vcal = \mathrm{Lex}^f$, such extensions to $\OC$ (and in fact to $\Bord_2^\partial$) have been constructed in \cite{muller2024categorified} under the additional assumption that the value at the disk is a pivotal finite tensor category in the sense of \cite{etingof2004finite} (which in particular implies rigidity), although they did not show that such extension is initial and unique.

\begin{thm}[Main theorem]\label{thm:main}
    Let $\Vcal$ be a symmetric monoidal \category{}  that admits geometric realizations and where the tensor product preserves geometric realizations separately in each variable. Then there is an adjunction
    \[
        i_!\colon \Fun^\otimes(\Ocal, \Vcal) \adj \Fun^\otimes(\OC, \Vcal) \cocolon i^\ast,
    \]
    where $i_!$ is fully faithful, i.e., every symmetric monoidal functor $F\colon \Ocal \to \Vcal$ can be extended to a symmetric monoidal functor $\OC \to \Vcal$ and $i_!F$ is initial among such extensions. Furthermore, $i_!F$ can be characterized as the \emph{unique} symmetric monoidal extension  $F'\colon \OC \to \Vcal$ such that the canonical map
    \[
        \int_{S^1} F(D^1) \too F'(S^1)
    \]
    is an equivalence.
\end{thm}

The assumption that $\Vcal$ has all geometric realizations can be weakened to having the specific colimit over the diagram indexing the cyclic bar construction. We provide in \cref{thm:extending-from-O-to-OC} a formulation of the main theorem in its full generality.

The left Kan extension of a symmetric monoidal functor (if it exists) always inherits a natural lax symmetric monoidal structure. 
To show that it is in fact (strong) symmetric monoidal requires inputs that are non-formal. In this case, we will prove the strong symmetric monoidality geometrically, with inputs the convergence of embedding calculus in dimension one \cite{krannich2024embedding} and the contractibility of certain arc complexes \cite{Wahl2016,hatcherwahl}.

\begin{figure}[ht]
\centering
\def\svgwidth{.8\linewidth}
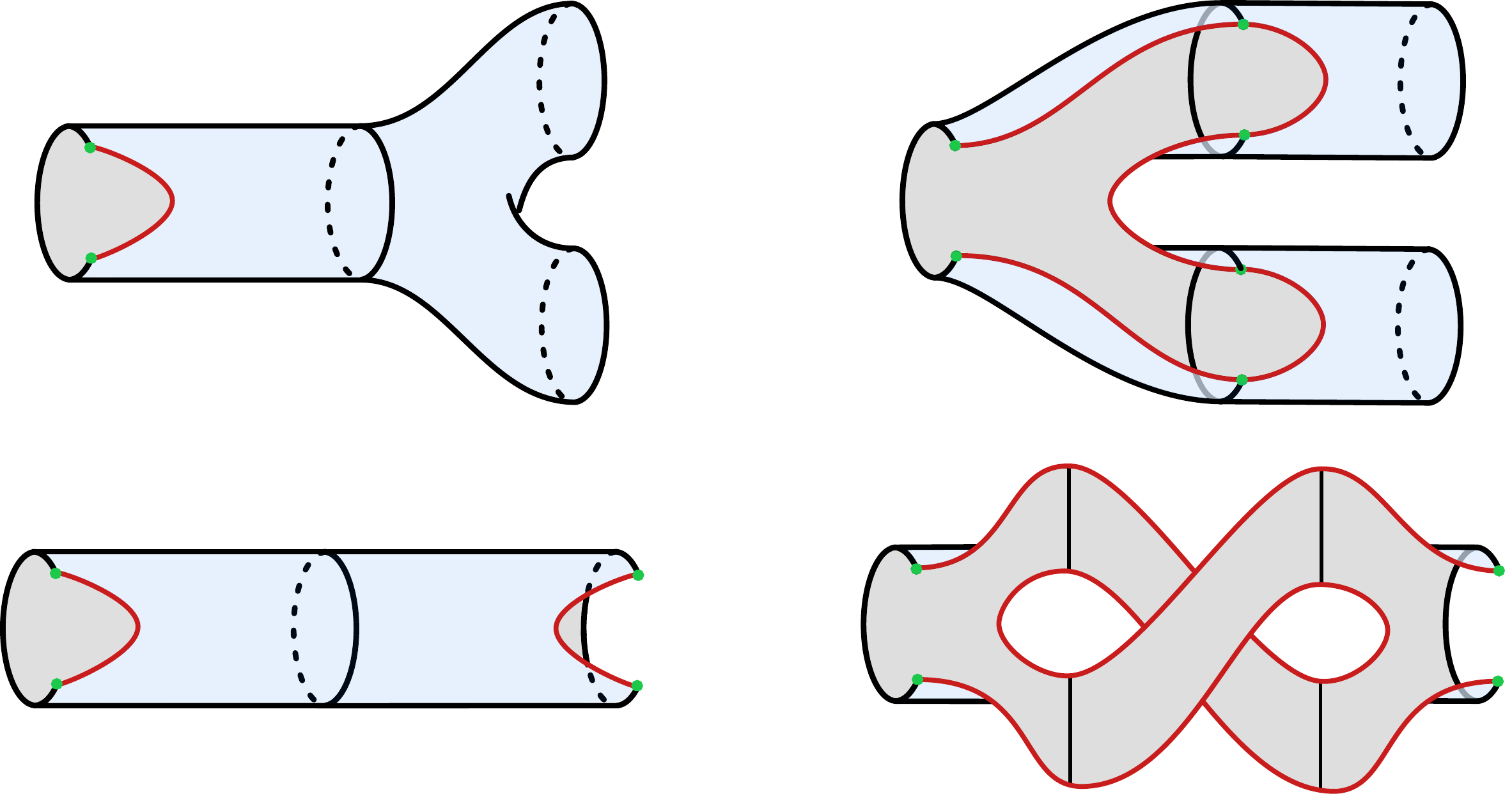
\caption{The first isomorphism shows that the ``whistle'' bordism $W\colon D^1 \to S^1$ from \cref{fig:bordex} is a coalgebra map, and the second isomorphism is the ``Cardy condition'' expressing $W^\vee \circ W\colon D^1 \to D^1$ purely in terms of the multiplication and comultiplication on $D^1$.}
\label{fig:cardy}
\end{figure}

Combining \cref{thm:main} with \cref{thm:O->V} we obtain \cref{thm:operation}.
In fact, this proves the stronger statement that in the setting of \cref{thm:operation} we obtain a symmetric monoidal functor $F\colon \OC \to \Vcal$ that sends $D^1$ to $A$ and $S^1$ to $\int_{S^1} A$.
In particular, the ``whistle'' morphism from \cref{fig:bordex} is sent to a map $F(W)\colon A \to \int_{S^1} A$ relating the algebraic structures on $A$ and $\int_{S^1} A$ in various ways, for example as indicated in \cref{fig:cardy}.

\begin{rem}
    One might ask if a similar result holds for extensions along the (non-full) inclusion $\OC\hookrightarrow\Bord_2^\partial$. 
    The answer is no in general, since $S^1$ is a dualisable object in $\Bord_2^\partial$ (but not in $\OC$), while Hochschild homology objects are usually not dualisable.
    As we will see in \cref{thm:O-dense-in-OC-intro}, $\OC$ is in some sense the ''maximal'' subcategory of $\Bord_2^\partial$ for \cref{thm:main} to hold.
\end{rem}

\subsubsection{An embedding calculus perspective}
To put \cref{thm:main} in context, we take inspiration from Goodwillie--Weiss embedding calculus where one studies $d$-manifolds by studying embeddings of disks into them.
Let $\Mfd_d^\partial$ denote the \category{} of (compact, oriented) $d$-manifolds and (orientation preserving) embeddings between them and let $\Disk_d \subset \Mfd_d^\partial$ denote the full subcategory on the objects of the form $\sqcup_k D^d$.
(Here the embeddings do not need to send boundary to boundary.)
Then embedding calculus is concerned with the restricted Yoneda embedding
\[
    \Yo_{\Disk}\colon \Mfd_d \xtoo{\Yo} \PSh(\Mfd_d) \xtoo{\text{restrict}} \PSh(\Disk_d),
\]
which sends a $d$-manifold $M$ to the presheaf of $\Disk_d$ given by $\Emb(-, M)$. 
On mapping spaces this functor induces the comparison map
\[
    \Yo_\Disk\colon \Emb(M, N) \too \Map_{\PSh(\Disk_d)}(E_M, E_N) =: T_\infty\Emb(M, N)
\]
between the space of embeddings and the limit of the embedding tower.\footnote{To study the tower, one also needs to restrict to the full subcategories $\Disk_d^{\le k} \subset \Disk_d$ where one limits the numbers of disks.}
Embedding calculus is said to \emph{converge for all manifolds in dimension $d$} if and only if the restricted Yoneda embedding $\Yo_\Disk$ is fully faithful.
This is indeed the case when $d \le 2$ by Krannich-Kupers \cite[Theorem A]{krannich2024embedding}. 
(For $d\geq3$ this is generally false and one needs to restrict to specific types of manifolds to guarantee convergence.)

In analogy, the major ingredient to our main theorem (\cref{thm:main}) can be stated as:  
\begin{thm}\label{thm:O-dense-in-OC-intro}
    The open surface \category{} $\Ocal$ is \hldef{dense}\footnote{
        Originally introduced by Isbel as ``adequate subcategories'' \cite{Isbell1960-adequate}.
    } in the open-closed surface \category{} $\OC$, i.e.~the restricted Yoneda embedding
    \[
        \Yo_\Ocal\colon \OC \hookrightarrow \Psh(\Ocal) 
    \]
    is fully faithful.
\end{thm}
Roughly speaking, this means that every object in $\OC$, in particular $S^1$, can be obtained as formally adjoining a certain canonical colimit to $\Ocal$. To prove \cref{thm:main}, we need to further show that the canonical colimit diagram for $S^1$ admits a final subdiagram that is precisely the diagram indexing the cyclic bar construction. We will provide an outline of the proof in \cref{subsec:outline} after some preliminary remarks and recollections about the bordism category. In particular, we make use of the following geometric description of the slice category of the bordism category, the proof of which can be found in \cref{subsec:Bord-pullback-proof}.
\begin{thm}
    For $d \ge 0$ and $M \in \Bord_d^\partial$ a compact $(d-1)$-manifold with boundary there is an equivalence
    \[
        \boxMfd_{d,M} \simeq (\Bord_d^\partial)_{M/}
    \]
    where $\boxMfd_{d,M}$ is a topologically enriched category whose objects are bordisms $W\colon M \to N$ starting at $M$, 
    and where the mapping space from $W$ to $W'$ is the space of embeddings $i\colon W \hookrightarrow W'$ that restrict to the identity on $M$ and that satisfy $\pf W = i^{-1}(\pf W')$.
\end{thm}

\subsubsection{Related work} We close this section by briefly mentioning some related results in the literature and directions of further work.
\begin{enumerate}[a.]
    \item \emph{$D$-branes and Calabi--Yau categories.}
    In \cite{costello2007topological}, Costello studied symmetric monoidal functors out of the bordism categories $\Ocal_{\Lambda}$ and $\OC_{\Lambda}$, where the free boundaries of the 2-bordisms in $\Ocal$ and $\OC$ are further labeled by elements of a set $\Lambda$ (which are called ''$D$-branes'').
    Then by \cite[Theorem A]{costello2007topological} symmetric monoidal functors from $\Ocal_\Lambda$ to the derived category of $\Qbb$ are classified by (proper) $A_\infty$-Calabi--Yau categories with $\Lambda$ the set of objects, and extensions to $\OC_{\Lambda}$ are canonical in the same way as \cref{thm:main}.  
    \cref{thm:main} (and \cref{thm:O->V}) concern the case where $\Lambda$ contains a single element. 
    We believe that both results and our proof strategies can be generalized to the case where $\Lambda$ is any set, but we will not spell out the details here.
    The resulting action of the surface operad on the Hochschild homology of the Calabi--Yau-category should then generalise the $\Erm_2^{\rm fr}$-action obtained by Brav--Rozenblyum \cite{BravRozenblyum} (though they work in the more general ``relative Calabi--Yau'' case).
    
    \item \emph{Non-compact cobordism hypothesis.}
    In \cite[Section 4.2]{lurie2008classification}, Lurie proposed a ``non-compact'' variant of the oriented cobordism hypothesis in dimension 2, which classifies symmetric monoidal functors from a wide but non-full subcategory of the $(\infty,2)$-category of fully-extended oriented $2$-dimensional bordisms to an $(\infty,2)$-category in terms of ``Calabi--Yau objects'' therein. 
    We explain how \cref{thm:main} is related to this hypothesis in \cref{subsec:noncompact}. 
\end{enumerate}

\subsubsection{Acknowledgments}
We would like to thank Søren Galatius, Manuel Krannich, Maxime Ramzi, Oscar Randal-Williams, and Lukas Woike for helpful conversations;
and Robert Burklund and Nathalie Wahl for helpful conversations and feedback on an earlier draft.

JS was supported by the Independent Research Fund Denmark (grant no.~10.46540/3103-00099B).
AZ was supported by the European Union via the Marie Skłodowska--Curie postdoctoral fellowship (project 101150469) and the DNRF through the Copenhagen Centre for Geometry and Topology (``GeoTop'' -- DRNF151).
SB and JS would like to thank the GeoTop Centre for its hospitality during the early stages of this project.
AZ would like to thank the Isaac Newton Institute for Mathematical Sciences, Cambridge for hospitality during the programme \emph{Equivariant homotopy theory in context} (EPSRC grant EP/Z000580/1), where some of the writing of this paper was undertaken.

%% file: figures/cardy.pdf_tex
\begingroup%
  \makeatletter%
  \providecommand\color[2][]{%
    \errmessage{(Inkscape) Color is used for the text in Inkscape, but the package 'color.sty' is not loaded}%
    \renewcommand\color[2][]{}%
  }%
  \providecommand\transparent[1]{%
    \errmessage{(Inkscape) Transparency is used (non-zero) for the text in Inkscape, but the package 'transparent.sty' is not loaded}%
    \renewcommand\transparent[1]{}%
  }%
  \providecommand\rotatebox[2]{#2}%
  \newcommand*\fsize{\dimexpr\f@size pt\relax}%
  \newcommand*\lineheight[1]{\fontsize{\fsize}{#1\fsize}\selectfont}%
  \ifx\svgwidth\undefined%
    \setlength{\unitlength}{1134.08162755bp}%
    \ifx\svgscale\undefined%
      \relax%
    \else%
      \setlength{\unitlength}{\unitlength * \real{\svgscale}}%
    \fi%
  \else%
    \setlength{\unitlength}{\svgwidth}%
  \fi%
  \global\let\svgwidth\undefined%
  \global\let\svgscale\undefined%
  \makeatother%
  \begin{picture}(1,0.52501552)%
    \lineheight{1}%
    \setlength\tabcolsep{0pt}%
    \put(0.4761104,0.38761236){\color[rgb]{0,0,0}\makebox(0,0)[lt]{\lineheight{1.25}\smash{\begin{tabular}[t]{l}$\cong$\end{tabular}}}}%
    \put(0.47206143,0.10598907){\color[rgb]{0,0,0}\makebox(0,0)[lt]{\lineheight{1.25}\smash{\begin{tabular}[t]{l}$\cong$\end{tabular}}}}%
    \put(0,0){\includegraphics[width=\unitlength,page=1]{cardy.pdf}}%
  \end{picture}%
\endgroup%

%% file: sections/slice.tex
\section{The bordism category and its slice}

In this section, we give a geometric definition of the \category{} of bordisms $\Bord_d^\partial$, where objects are compact oriented $(d-1)$-manifolds and the mapping spaces are moduli-spaces of compact oriented $d$-bordisms with corners.
All our manifolds are oriented and we will leave this implicit.
Then we state a theorem (proven in \cref{subsec:Bord-pullback-proof}) that describes the slice \categories{} $(\Bord_d^\partial)_{M/}$ in terms of certain topologically enriched categories $\boxMfd_{d,M}$ of $d$-manifolds and embeddings.

\subsection{Constructing the bordism category}\label{sec:bord}
When working with the bordism category as an \category{}, it is often convenient to construct it as a Segal space \cite{lurie2008classification, CS22, KrannichKupers-Disc}.
We will briefly recall this story here, and our approach will be close to \cite{KrannichKupers-Disc} in two ways: 
we use quasi-unital Segal spaces, and we will use topologically enriched groupoids of bordisms rather than topological spaces of (embedded) bordisms.

There is an adjunction
\[
    \ac\colon \Fun(\Dop, \Scal) \adj \Cat \cocolon \xN^r
\]
where $\xN^r$ is the Rezk nerve defined by $\xN^r_n(\Ccal) \coloneq \Fun([n], \Ccal)^\simeq$ and $\ac$ is the associated \category{} functor.
The Rezk nerve functor $\xN^r$ is fully faithful, and its essential image are the complete Segal spaces, see e.g.~\cite{Rezk-nerve-25}.
Thus, we can construct \categories{} by constructing complete Segal spaces.
In practice, it turns out to be more convenient to only construct Segal spaces and to then formally complete them (this is for example done in \cite[\S5]{CS22}), or equivalently to directly apply $\ac(-)$.

In constructing these Segal spaces, it can often be challenging to specify the degeneracy maps ``on the nose'', as for example in the case of $\Bord_d^\partial$ they result in length $0$ bordisms.
Luckily, a theorem of Haugseng \cite{Haugseng-semicategories} tells us that it suffices to specify face maps and to then check the condition that degeneracies exist up to homotopy.
More precisely, he proves that for $\Seg(\Dop; \Scal) \subset \Fun(\Dop, \Scal)$ the full subcategory of Segal spaces, the restriction functor
\[
    \Seg(\Dop; \Scal) \too \Fun(\Dop, \Scal) \too \Fun(\Dop_\inj, \Scal)
\]
to semi-simplicial spaces is faithful and replete (i.e.~fully faithful on maximal subgroupoids), and it is an equivalence onto the subcategory of those semi-simplicial spaces that satisfy the Segal condition $X_n \simeq X_1 \times_{X_0} \dots \times_{X_0} X_1$ and an additional ``quasi-unital'' condition.
(Note that while this is not a full subcategory, as not all morphisms preserve quasi-units, it is full on maximal subgroupoids, so to prove that two Segal spaces are equivalent it suffices to exhibit an equivalence of semi-simplicial spaces.)

Finally, the ``spaces'' (i.e.~\groupoids) that appear as part of the (semi-)Segal space $\Bord_d^\partial$ are moduli spaces of manifolds.
This means that, broadly speaking, we have two choices for how to model them:
either as topological spaces of (unparametrized) submanifolds of $\Rbb^\infty$ (as is common when using scanning methods \cite{GMTW09}),
or as topologically enriched groupoids, where objects are manifolds and mapping spaces are the spaces of diffeomorphisms with their usual Whitney $\mcal{C}^\infty$-topology (as in \cite{KrannichKupers-Disc}).
The latter approach seems to be more convenient here.
This means that we will construct a semi-simplicial object in the $1$-category of topologically enriched groupoids $\mrm{TopGpd}$ and then use the functor
\[
    \mrm{TopGpd} \xtoo{|-|} \Top \too \Scal
\]
that sends a topologically enriched groupoid to the geometric realization of its (topological) nerve.
When applied to the topologically enriched groupoid of $d$-manifolds and diffeomorphisms, this exactly results in the moduli space of $d$-manifolds.

In the following we will consider submanifolds $W \subset \Rbb \times \Rbb^\infty$ and when $A \subset \Rbb$ we let $W_A \coloneq W \cap (A \times \Rbb^\infty)$.
(When $A$ is a single point $t\in \Rbb$ we think of this as a submanifold of $\Rbb^\infty$, rather than $\{t\} \times \Rbb$.)
\begin{defn}
    For all $[n] \in \Dop_\inj$, an \hldef{$[n]$-walled $d$-bordism} is a pair of a strictly monotone map $\mu\colon [n] \hookrightarrow \Rbb$ and an oriented submanifold with boundary $W \subset \Rbb \times \Rbb^\infty$ such that
    \begin{enumerate}
        \item for all $i \in [n]$, $\mu(i)$ is a regular value of the first coordinate projection $\pr_1\colon W \to \Rbb$, and 
        \item $W_{[\mu(0), \mu(n)]}$ is compact.
    \end{enumerate}
    A diffeomorphism between two $[n]$-walled $d$-bordisms $(W, \mu)$ and $(W', \mu')$ is 
    an orientation-preserving diffeomorphism $\varphi\colon W_{[\mu(0),\mu(n)]} \to W'_{[\mu'(0), \mu'(n)]}$ that satisfies $\varphi(W_{[\mu(i),\mu(i+1)]}) = W'_{[\mu'(i), \mu'(i+1)]}$ for all $0\le i <n$.
   We let $\Bord_d^\partial[n]$ denote the topologically enriched groupoid whose objects are the $[n]$-walled $d$-bordisms and whose mapping spaces are the spaces of diffeomorphisms between $[n]$-walled bordisms, with the usual Whitney $\Ccal^\infty$-topology.
   We can think of this as a topological groupoid by equipping the set of objects with the discrete topology.
\end{defn}

\begin{rem}
    Unlike the definition of \cite{GMTW09} our space of bordisms is not topologised: the objects of $\Bord_d^\partial[n]$ are a discrete set, and only the morphisms (which are diffeomorphisms) are topologised.
    In particular, the manifolds we use are only embedded into $\Rbb \times \Rbb^\infty$ for convenience and there is no topology that allows us to isotope an embedded bordism.
    Therefore, we could equivalently work with submanifolds of $\Rbb \times S^{2d}$ (since by the Whitney embedding theorem every bordism admits at least one embedding into this), or we could work with abstract manifolds $W$ equipped with a proper smooth function $W \to \Rbb$ that replaces $\mrm{pr}_1$.
    (Though we would need to ensure that there is a set of such manifolds.)
\end{rem}

The following lemma was proved for a slightly more complicated version of $\Bord_d^\partial$ in \cite{KrannichKupers-Disc}, and for this specific definition it will appear in \cite{modular}.
\begin{lem}
    $\Bord_d^\partial[\bullet]$ defines a quasi-unital semi-Segal space.
\end{lem}

It thus follows from \cite{Haugseng-semicategories} that $\Bord_d^\partial[\bullet]$ uniquely extends to a simplicial space that is a Segal space.
By abuse of notation, we will also denote the simplicial space by $\Bord_d^\partial[\bullet]$.
We then define the bordism category as
\[
    \Bord_d^\partial \coloneq \ac(\Bord_d^\partial[\bullet]).
\]

We will also need the symmetric monoidal structure on $\Bord_d^\partial$ that is given by disjoint union. 
A symmetric monoidal \category{} is a commutative monoid in the \category{} $\Cat$, i.e.~a functor $F\colon \Fin_* \to \Cat$ satisfying the Segal condition $F(A_+) \simeq \prod_{a \in A} F(\{a\}_+)$, where $A_+ = A \sqcup \{\infty\}$ is a finite set with added base-point.
Our construction of the bordism category can easily be improved to be symmetric monoidal, by defining a functor
\[
    \Bord_d^\partial \colon \Dop_\inj \times \Fin_* \too \mrm{TopGpd}
\]
that sends $([n], A_+)$ to a version of $\Bord_d^\partial[n]$ where the $[n]$-walled bordisms $(W,\mu)$ additionally come with a continuous map $l\colon W_{[\mu(0),\mu(n)]} \to A$ that decomposes the bordism into an $A$-indexed disjoint union.
The functoriality in $\alpha\colon A_+ \to B_+$ is by composing $l$ with $\alpha$ and then discarding any components that are mapped to the base point $\infty \in B_+$.
We refer the reader to \cite{modular} (in preparation) for more details.

\subsection{Slices of the bordism category}
We now want to describe the slice categories $(\Bord_d^\partial)_{M/}$.
An object of this \category{} is a bordism $W\colon M \to N$ and a morphism $W_1 \to W_2$ is a bordism $V\colon N_1 \to N_2$ together with an identification $\varphi\colon W_1 \cup V \cong W_2$. 
This identification in particular gives an embedding $i = \varphi_{|W_1} \colon W_1 \hookrightarrow W_2$, and we can in fact recover the bordism $V$ as the complement $W_2 \setminus i(W_1)$ of the embedding.
The goal of this section is to make this precise and to show that the \category{} $(\Bord_d^\partial)_{M/}$ is equivalently modelled by a topologically enriched category $\boxMfd_{d,M}$ whose objects are bordisms starting at $M$ and whose morphisms are embeddings that fix $M$.
As bordisms in $\Bord_d^\partial$ are manifolds with corners, we will have to allow them in $\boxMfd_{d,M}$.
In fact, it will be more convenient to first consider a version where the incoming boundary is only fixed set-wise.

\begin{defn}
    We define the non-unital topologically enriched category \hldef{$\boxMfd_d$} of marked $d$-manifolds.
    Objects are $d$-manifolds $W$ with corners and a decomposition of their boundary as $\partial W = \partial_+ W \cup \pf W \cup \partial_- W$ such that each of $\partial_+ W$, $\partial_- W$, $\pf W$ is a compact $(d-1)$-manifold with boundary, $\partial_+ W \cap \partial_- W = \emptyset$, and the corners of $W$ are precisely $(\partial_+ W \sqcup \partial_- W) \cap \pf W$.
    (In particular this means that the corners are $\partial(\pf W) = \partial( \partial_+ W \sqcup \partial_- W)$.)
    A morphism is a (smooth) embedding $i\colon W \hookrightarrow V$ satisfying
    \begin{enumerate}[(1)]
        \item $i(\partial_+ W) = \partial_+ V$,
        \item $i(\pf W) \subseteq \pf V$, and 
        \item $i_{|\partial_- W}\colon \partial_- W \to V$ lands in $V \setminus \partial_- V$ and is transverse to $\pf V$.
    \end{enumerate}
    We topologise the mapping spaces as subspaces
    \[
        \Emb^\square(W, V) \subset \Emb(W, V).
    \]
\end{defn}

This is only a \emph{non-unital} topologically enriched category as the identity maps do not satisfy (3).
However, it does have quasi-units  that can be obtained by taking the identity embedding $\id_W\colon W \to W$ and pushing it away from $\partial_- W$ using a collar.
(Equivalently, these are ``weak units'' in the sense of \cite[Definition 3.12]{ERW-semi-simplicial}.)
We also have a variant where we fix the $\partial_+$-boundary.

 \begin{figure}[h]
\centering
\def\svgwidth{.8\linewidth}
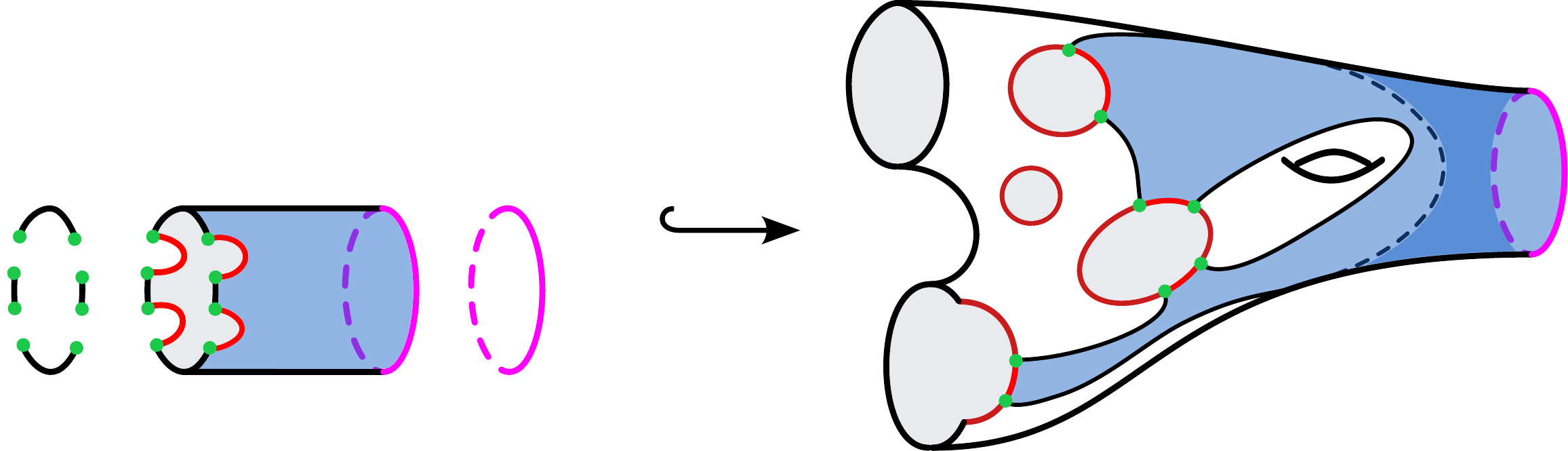
\caption{A morphism in $\Emb_{S^1}^\square(W, V)$.}
\end{figure}

\begin{defn}
    For a fixed compact oriented $(d-1)$-manifold $M$ we let \hldef{$\boxMfd_{d,M}$} denote the (non-replete) non-unital subcategory of $\boxMfd_d$ where we require the objects to satisfy $\partial_+ W = M$ and the morphisms $i\colon W \to V$ to satisfy $i_{|\partial_+ W} = \id_M$.
    We denote these embedding spaces by 
    \[
        \Emb_M^\square(W, V) = \{\id_M\} \times_{\Diff(M)} \Emb^\square(W, V) \subset \Emb(W, V).
    \]
\end{defn}

To these quasi-unital topologically enriched categories we can assign \categories{} by first taking their topologically enriched nerve to get a semi-simplicial Segal space, which can then uniquely be promoted to a simplicial Segal space (by \cite{Haugseng-semicategories} as before), and then taking $\ac(-)$.
This chain of operations preserves the mapping spaces.
By abuse of notation we will denote by $\boxMfd_d$ and $\boxMfd_{d,M}$ both the quasi-unital topologically enriched categories and the \categories{} obtained from them.
Note that both of these are in fact symmetric monoidal \categories{} under the disjoint union operation $\sqcup$ in the same way that $\Bord_d^\partial$ is.

We can now state the key result that relates the slices of the bordism category to embeddings.
This will be proven as \cref{thm:Phi-and-Psi} in \cref{subsec:Bord-pullback-proof}, and in fact we will prove a slightly stronger result (\cref{prop:ArBord-pullback}).

\begin{cor}\label{cor:Phi-and-Psi} 
    For every compact oriented $(d-1)$-manifold $M$ there are equivalences
    \[
        \hldef{\widetilde{\Psi}_S} \colon \boxMfd_{d,S} \xtoo{\simeq} (\Bord_d^\partial)_{S/}
        \qquad \text{and} \qquad
        \hldef{\widetilde{\Phi}_S} \colon (\boxMfd_{d,S})^\op \xtoo{\simeq} (\Bord_d^\partial)_{/S}
    \] 
    and for $S = \emptyset$ these equivalences are symmetric monoidal.
\end{cor}

In the rest of the paper, we will be mostly interested in computing mapping spaces in slice categories of $\OC$ instead of $\Bord_2^\partial$. 
Recall from \cref{defn:OC} that $\OC \subset \Bord_2^\partial$ is the subcategory that has all objects but only those bordisms $W\colon M \to N$ for which the subspace $M \cup \partial_{\rm free} W \subset W$ intersects all connected components of $W$. For example, consider two morphisms $\emptyset\to S^1$ in $\Bord_2^\partial$ as in \cref{fig:Oc}. The left one is not a morphism in $\OC$, while the right one is.

 \begin{figure}[h]
\centering
\def\svgwidth{.3\linewidth}
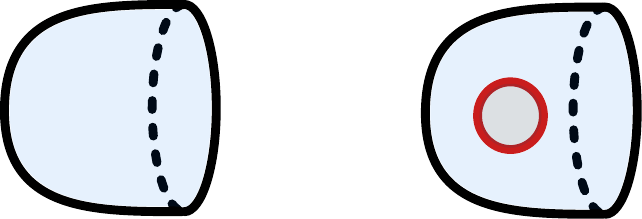
\caption{Non-example and example of morphism in $\OC$.}
\label{fig:Oc}
\end{figure}

It turns out that slicing under an object behaves well with respect to the inclusion of subcategory $\OC$ of $\Bord_2^{\partial}$.
\begin{prop}\label{prop: sliceunderff}
    For any $M\in \OC$, the inclusion $\OC_{M/}\rightarrow(\Bord_2^\partial)_{M/}$ is fully faithful.
\end{prop}
\begin{proof}
    Because $\OC$ is defined as a subcategory of $\Bord_2^\partial$, the functor is automatically faithful (i.e.~a monomorphism on mapping spaces).
    We need to show that for all $(U\colon M \to N), (V\colon M \to N') \in \OC_{M/}$ and all $W\colon N \to N'$ in $\Bord_2^\partial$ with $W \circ U \cong V$, the morphism $W$ is automatically in $\OC$.
    We can think of $W$ as a codimension $0$ submanifold of $V$ (not intersecting $M$) and $U = \overline{V \setminus W}$ as the closure of its complement. Then $N = U \cap W$.
    Suppose that $W$ contains a path component $W'\subset W$ that does not lie in $\OC$. 
    Then there exists a path component $U'$ of $U$ such that $U'\cap W'$ is non-empty, or $W'$ would be a path component of $V\in\OC$, a contradiction. But $U'\cap W'\subseteq U\cap W = N$ is a submanifold of the outgoing boundary of $U$, which implies that the incoming boundary of $W'$ is non-empty, a contradiction. This concludes the proof.
\end{proof}

In contrast, this is in general not true if we take the slice over an object in $\OC$.
\begin{rem}\label{rem: OC_/S^1 not ff}
    The inclusion $\OC_{/S^1}\rightarrow (\Bord_2^\partial)_{/S^1}$ is not fully faithful. 
    As an example, consider the two objects of $\OC_{/S^1}$ given by the bordism $V\colon S^1 \to S^1$ that is a cylinder with an open disk removed (thus $\pf V$ is a circle) and $W\colon \emptyset \to S^1$ that is a cylinder with one of its boundaries being free.
    The bordism $U\colon \emptyset\rightarrow S^1$ that is a $2$-disk whose boundary is outgoing boundary is not in $\OC$, but $U \cup V \cong W$, so $U$ defines a morphism $W \to V$ in $(\Bord_2^\partial)_{/S^1}$ that is not in $\OC_{/S^1}$.
\end{rem}

%% file: figures/boxemb.pdf_tex
\begingroup%
  \makeatletter%
  \providecommand\color[2][]{%
    \errmessage{(Inkscape) Color is used for the text in Inkscape, but the package 'color.sty' is not loaded}%
    \renewcommand\color[2][]{}%
  }%
  \providecommand\transparent[1]{%
    \errmessage{(Inkscape) Transparency is used (non-zero) for the text in Inkscape, but the package 'transparent.sty' is not loaded}%
    \renewcommand\transparent[1]{}%
  }%
  \providecommand\rotatebox[2]{#2}%
  \newcommand*\fsize{\dimexpr\f@size pt\relax}%
  \newcommand*\lineheight[1]{\fontsize{\fsize}{#1\fsize}\selectfont}%
  \ifx\svgwidth\undefined%
    \setlength{\unitlength}{1109.94103139bp}%
    \ifx\svgscale\undefined%
      \relax%
    \else%
      \setlength{\unitlength}{\unitlength * \real{\svgscale}}%
    \fi%
  \else%
    \setlength{\unitlength}{\svgwidth}%
  \fi%
  \global\let\svgwidth\undefined%
  \global\let\svgscale\undefined%
  \makeatother%
  \begin{picture}(1,0.28653224)%
    \lineheight{1}%
    \setlength\tabcolsep{0pt}%
    \put(0,0){\includegraphics[width=\unitlength,page=1]{boxemb.pdf}}%
    \put(0.30136112,0.00630655){\color[rgb]{1,0,1}\makebox(0,0)[lt]{\lineheight{1.25}\smash{\begin{tabular}[t]{l}$\partial_+ W$\end{tabular}}}}%
    \put(-0.00067031,0.00630655){\color[rgb]{0,0,0}\makebox(0,0)[lt]{\lineheight{1.25}\smash{\begin{tabular}[t]{l}$\partial_- W$\end{tabular}}}}%
    \put(0.17275085,0.00630655){\color[rgb]{0,0,0}\makebox(0,0)[lt]{\lineheight{1.25}\smash{\begin{tabular}[t]{l}$W$\end{tabular}}}}%
    \put(0,0){\includegraphics[width=\unitlength,page=2]{boxemb.pdf}}%
    \put(0.21311086,0.223874){\color[rgb]{1,0,0}\makebox(0,0)[lt]{\lineheight{1.25}\smash{\begin{tabular}[t]{l}$\pf W$\end{tabular}}}}%
  \end{picture}%
\endgroup%

%% file: figures/OC.pdf_tex
\begingroup%
  \makeatletter%
  \providecommand\color[2][]{%
    \errmessage{(Inkscape) Color is used for the text in Inkscape, but the package 'color.sty' is not loaded}%
    \renewcommand\color[2][]{}%
  }%
  \providecommand\transparent[1]{%
    \errmessage{(Inkscape) Transparency is used (non-zero) for the text in Inkscape, but the package 'transparent.sty' is not loaded}%
    \renewcommand\transparent[1]{}%
  }%
  \providecommand\rotatebox[2]{#2}%
  \newcommand*\fsize{\dimexpr\f@size pt\relax}%
  \newcommand*\lineheight[1]{\fontsize{\fsize}{#1\fsize}\selectfont}%
  \ifx\svgwidth\undefined%
    \setlength{\unitlength}{307.16555858bp}%
    \ifx\svgscale\undefined%
      \relax%
    \else%
      \setlength{\unitlength}{\unitlength * \real{\svgscale}}%
    \fi%
  \else%
    \setlength{\unitlength}{\svgwidth}%
  \fi%
  \global\let\svgwidth\undefined%
  \global\let\svgscale\undefined%
  \makeatother%
  \begin{picture}(1,0.34090633)%
    \lineheight{1}%
    \setlength\tabcolsep{0pt}%
    \put(0,0){\includegraphics[width=\unitlength,page=1]{OC.pdf}}%
  \end{picture}%
\endgroup%

%% file: sections/prelim.tex
\section{Preliminaries and outline of the proof}
In this section, we will provide an outline of the proof for \cref{thm:main} and  \cref{thm:O-dense-in-OC-intro}. We will also formalize how \cref{thm:operation} generalizes the determination of formal operations of $\Ocal$-algebras in the sense \cite{wahl2016universal}.
\subsection{Factorization homology of $\mrm{E}_1$-algebras}
We start with some recollections on $\mrm{E}_1$-algebras in a symmetric monoidal \category{} $\Ccal$ and their factorization homology.

Recall that the $1$-category $\Assoc$ is defined to have as objects finite sets and as morphisms maps $f\colon A \to B$ together with a total order on $f^{-1}(a)$ for all $a \in A$.
Compositing of morphisms is defined by composing the maps and inducing the lexicographic order on each $f^{-1}(g^{-1}(c))$.
This is a symmetric monoidal category under disjoint union.
A more geometric model of this category is given as follows:
\begin{lem}
    Let $\Disk_1 \subset \Mfd_1^\partial$ denote the full subcategory where objects are $1$-manifolds of the form $\sqcup_k D^1$ for $k\ge0$.
    Then the functor
    \[
        \pi_0 \colon \Disk_1 \too \Assoc
    \]
    is an equivalence of symmetric monoidal categories.
\end{lem}

In fact, $\Assoc$ (or equivalently $\Disk_1 $) is the symmetric monoidal envelope of the operad $\Erm_1$.
Therefore, the category of $\Erm_1$-algebras in a symmetric monoidal \category{} $\Ccal$ is given by the category of symmetric monoidal functors from $\Assoc$ to $\Ccal$.
\[
    \Alg_{\Erm_1}(\Ccal) = \Fun^{\otimes}(\Assoc, \Ccal)
\]

A definition of \operads{} and a proof of this fact can be found in \cite[\S 2.1.1, \S 4.1.1, Proposition 2.2.4.9]{HA},
but for the purpose of this paper we might as well take this as our definition of $\Erm_1$-algebras and thus avoid \operads{} and the envelope construction all together.
(Under this definition the value of the symmetric monoidal functor at the one element set $*$ is the underlying object of the $\Erm_1$-algebra.)
\begin{defn}
    For a $1$-manifold $M \in \Mfd_1^\partial$ we let 
    \[ 
        \Disk_{/M} \coloneq \Disk_1 \times_{\Mfd_1^\partial} (\Mfd_1^\partial)_{/M}
    \]
    be the \category{} of ``disks in $M$''.
    Given an $\Erm_1$-algebra $A \in \Alg_{\Erm_1}(\Ccal)$, which we can write as a symmetric monoidal functor $A\colon \Disk_1 \to \Ccal$, its \hldef{factorization homology over $M$}, if exists, is defined as the colimit
    \[
        \hldef{\int_M A} \coloneq \colim\left( \Disk_{/M} \too \Disk_1 \xtoo{A} \Ccal \right).
    \]
\end{defn}
When $M=S^1$, the factorization homology $\int_{S^1}A$ is also called the Hochschild homology object of $A$. It is well-known that $\int_{S^1}A$ is equivalent to the geometric realization of the cyclic bar construction $  \mrm{Bar}^{\rm cyc}_\bullet(A) \colon \Dop \to \Ccal 
$, see for instance \cite[Example 1.3.9]{afmgr}. We provide an independent proof of this equivalence as a corollary of \cref{lem: comparediagram}.

Because $\Disk_1$ (and more generally $\Mfd_1^\partial$) is the symmetric monoidal envelope of an \operad{}, the slice categories satisfy the following relation, which we will need later. (See \cite[Example 2.3.17]{properads}.)

\begin{lem}\label{cor:Disk1tensordisjunct}
    $\Mfd_1^\partial$ is $\otimes$-disjunctive, i.e.~the functor
    \[
    \sqcup: (\Mfd_1)_{/M}\times (\Mfd_1)_{/N}\to (\Mfd_1)_{/M\sqcup N}
    \]
    is an equivalence of \categories{}, and similarly for $\Disk_{/M}$.
\end{lem}

\subsection{Outline of the proof}\label{subsec:outline}

Now we summarize the strategy and the main ingredients that go into the proofs of \cref{thm:main} and \cref{thm:O-dense-in-OC-intro}. First we make a few remarks about the notion of denseness that appears in the statement of \cref{thm:O-dense-in-OC-intro}.
\begin{defn}
    A full subcategory $\Dcal \subset \Ccal$ of an \category{} $\Ccal$ is called \hldef{dense}
   if the restricted Yoneda embedding
    \[
       \Yo_\Dcal\colon \Ccal \xtoo{\Yo} \PSh(\Ccal) \xtoo{\text{restrict}} \PSh(\Dcal) 
    \]
    is fully faithful.
\end{defn}
The presheaf category $\PSh(\Dcal)$ can be described as the free cocompletion of $\Dcal$, meaning that it is obtained by freely adjoining all (small) colimits to the \category{} $\Dcal$.
This universal property can be very useful when working with the \category{} $\PSh(\Dcal)$, for example it means that every functor $\Dcal \to \Vcal$ into a \category{} $\Vcal$ with colimits uniquely extends to a colimit preserving functor $\PSh(\Dcal) \to \Vcal$.
Similarly, dense subcategories $\Dcal \subset \Ccal$ can be thought of as exactly those fully faithful functors where $\Ccal$ is obtained from $\Dcal$ by formally adding some, but maybe not all, colimits.
To make this precise, we use an  alternative characterization of dense subcategories:

\begin{lem}[{\cite[03VG]{Kerodon}}]\label{lem:dense}
    A full subcategory $\Dcal \subset \Ccal$ is \emph{dense} if and only if for every object $c \in \Ccal$ the diagram
    \[
        (\Dcal \times_\Ccal \Ccal_{/c})^\rhd \too \Ccal_{/c} \too \Ccal 
    \]
    is a colimit diagram, i.e., if and only if every object in $\Ccal$ is the colimit of all the objects of $\Dcal$ mapping to it.   
\end{lem}

Therefore, \cref{thm:O-dense-in-OC-intro}, which states that the open bordism category $\Ocal$ is dense in the open-closed bordism category $\OC$, in particular says that $S^1 \in \OC$ is (canonically) the colimit of objects $\sqcup_k D^1 \in \Ocal$.
The canonical colimit diagram $\Ocal \times_{\OC} \OC_{/S^1}$ is rather complicated, which is why, in order to prove \cref{thm:O-dense-in-OC-intro}, we will instead express $S^1$ as a simpler colimit given by the cyclic bar construction (\cref{proposition A}), and show that the restricted Yoneda embedding preserves this colimit (\cref{proposition B}).
This then also proves a stronger version of \cref{thm:O-dense-in-OC-intro} and allows us to deduce \cref{thm:main}, which describes the value of the left Kan extension along $\Ocal \to \OC$ in terms of the factorization homology over $S^1$.

More precisely, we will reduce the theorems to the following two computations of colimits of shapes $\Disk_{/M}$ in $\OC$ and $\Psh(\Ocal)$ respectively:
\begin{propA}\label{proposition A}
    For every $1$-manifold $M$ the factorization homology over $M$ of the $\Erm_1$-algebra $D^1$ in $\OC$ is
    \[
        \int_{M} D^1:= \colim\left( \Disk_{/M} \too \Disk_1 \xto{\col'} \OC \right) \simeq  M \; \in \OC.
    \]
\end{propA}

The symmetric monoidal functor $\col'\colon \Disk_1\to\OC$ is constructed in \cref{subsec:E1onD1} and in particular equips $D^1\in\OC$ with the structure of an $\mrm{E}_1$-algebra. We prove \cref{proposition A} in \cref{subsec: proofA} using convergence of embedding calculus in dimension 1, see for instance \cite[Theorem A]{krannich2024embedding}. While the proof is straightforward, \cref{proposition A} is somewhat surprising, given that not many colimits are known to exist in $\OC$, or in bordism categories more generally.
(For example, $\OC$ has no coproducts.)

\begin{propA}\label{proposition B}
    The restricted Yoneda embedding functor $\Yo_\Ocal\colon \OC \to \PSh(\Ocal)$ preserves the colimit in \cref{proposition A}. That is, for any $M\in\OC$, the factorization homology over $M$ of the $\Erm_1$-algebra $\Yo_{\Ocal}(D^1) \in \Alg_{\Erm_1}(\PSh(\Ocal))$ is
    \[
        \int_{M} \Yo_\Ocal(D^1)  \simeq  \Yo_\Ocal(M) \; \in \PSh(\Ocal).
    \]
\end{propA}

In fact, \cref{proposition B} implies that the functor 
\[
    \mrm{Bar}^{\rm cyc}_\bullet(D^1) \colon \Dop \too \Ocal \times_{\OC} \OC_{/S^1} 
\]
is final (\cref{cor:Dopfinalinslice}), and therefore the value at $S^1$ (and in general at any $M$) of the left Kan extension along $\Ocal \hookrightarrow \OC$ may be computed as the colimit of a simplicial object.
This will also allow us to show that the left Kan extension, which a priori is only a lax symmetric monoidal functor, is in fact a symmetric monoidal functor.
As a consequence we also show that the inclusion $\Ocal \hookrightarrow \OC$ is initial and hence induces an equivalence between classifying spaces $|\Ocal| \simeq |\OC|$ (\cref{cor:O-OC-initial}).

To prove \cref{proposition B}, we first unwind the definitions and rewrite the relevant mapping spaces in the under slice category $\OC_{/M}$ in terms of mapping spaces in $\boxMfd_{2,M}$ via \cref{cor:Phi-and-Psi}. Then we show that the latter are equivalent to systems of arcs in the sense of \cite{wahl08}. 
This allows us to further reduce \cref{proposition B} to proving that certain arc complexes are contractible in \cref{sec: proofB}. The contractibility of those arc complexes then follows from \cite{wahl08, hatcherwahl}.

\subsection{Formal operations on Hochschild homology}
Before proceeding with the proofs of the two propositions, we explain how  \cref{thm:O-dense-in-OC-intro} provides a determination of ``formal operations'' on the Hochschild homology of $E_1$-Calabi--Yau algebras, analogous to the formal operations studied in \cite{wahl2016universal}.

We have a symmetric monoidal functor $\col'\colon \Disk_1 \to \Ocal$ that describes the $\Erm_1$-algebra structure on $D^1 \in \OC$.
For every $1$-manifold $M$ consider the functor
\[
    I_M\colon \Fun(\Ocal, \Scal) \xtoo{\col'} \Fun(\Disk, \Scal) \too \Scal
\]
where the second functor takes the colimit over $\Disk_{/M}$.
(Note that here $\Fun(\Ocal,\Scal)$ denotes non-monoidal functors.)
For every $\Erm_1$-Calabi--Yau algebra $A$, evaluating $I_M$ on the corresponding functor $F_A\colon \Ocal \to \Scal$ yields the factorization homology $I_M(F_A) = \int_M A$.
Therefore, natural transformations $I_M \to I_N$ yield operations $\int_M A \to \int_N A$ for $\Erm_1$-Calabi--Yau algebras $A$.
We let $\mrm{Nat}$ denote the \category{} whose objects are $1$-manifolds $M$ and whose mapping spaces from $M$ to $N$ are the natural transformations $I_M \to I_N$.
\cref{thm:O-dense-in-OC-intro} allows us to identify the \category{} of formal operations with the open-closed bordism category, generalizing \cite[Theorem B]{wahl2016universal}.\footnote{Note that we can also replace the \category{} $\Scal$ by any other presentable \category{} $\Wcal$, and we would still get a description of $\Wcal$-linear natural transformations.}

\begin{cor}\label{cor:formal-operations}
    The \category{} $\mrm{Nat}$ of formal operations is equivalent to the open-closed bordism category $\OC$.
\end{cor}
\begin{proof}
    If $M= \sqcup_{k} D^1$ is a disjoint union of disks, then $I_{\sqcup_k D^1}$ is simply the evaluation at $\sqcup_k D^1$, which preserves colimits.
    In general, we can write $I_{M} = \colim_{D \in \Disk_{/M}} I_{D}$, which therefore also preserves colimits.
    Hence, $\mrm{Nat}$ is a full subcategory of the \category{} of colimit preserving functors
    \[
        \mrm{Nat} \subset \Fun^L(\Fun(\Ocal, \Scal), \Scal) \simeq \Fun(\Ocal^\op, \Scal) = \PSh(\Ocal)
    \]
    where the equivalence uses that $\Fun(\Ocal, \Scal) = \PSh(\Ocal^\op)$ is the free colimit completion of $\Ocal^\op$.
    We thus have a fully faithful functor $\mrm{Nat} \hookrightarrow \PSh(\Ocal)$ that sends $I_{\sqcup_k D^1}$ to $\Yo_\Ocal(D^1)$ and $I_M$ to $\colim_{D \in \Disk_{/M}} \Yo_\Ocal(D^1)$, 
    which by \cref{proposition B} below is equivalent to $\Yo_\Ocal(M)$.
    Therefore, the above argument and \cref{thm:O-dense-in-OC-intro} give two fully faithful functors
    \[\mrm{Nat} \hookrightarrow \PSh(\Ocal) \hookleftarrow \OC \cocolon \Yo_\Ocal\]
    and as both have essential image $\{\Yo_\Ocal(M)\}_{M \in \Mfd_1}$, the two full subcategories are equivalent.
\end{proof}

%% file: sections/calculus.tex
\section{Proposition A: embedding calculus in dimension 1}

The goal of this section is to prove \cref{proposition A}, which says that for every $M \in \OC$
\[
    \int_{M} D^1 \simeq M \in \OC.
\]
We will prove this using that embedding calculus converges in dimension $1$ \cite[Theorem A]{krannich2024embedding}.

\subsection{The \texorpdfstring{$\Erm_1$}{ E1}-algebra structure on \texorpdfstring{$D^1$}{ D1}}\label{subsec:E1onD1}
We start by equipping $D^1 \in \Bord_2^\partial$ with the structure of an $\Erm_1$-algebra by constructing a symmetric monoidal functor $\Assoc \to \Bord_2^\partial$ that sends $A \in \Assoc$ to $A \times D^1 \in \Bord_2^\partial$. 
This $\Erm_1$-algebra structure will be such that the multiplication is given by the bordism $D^1 \sqcup D^1 \to D^1$ that is homeomorphic to the $2$-disk.
 \begin{figure}[h]
\centering
\def\svgwidth{.3\linewidth}
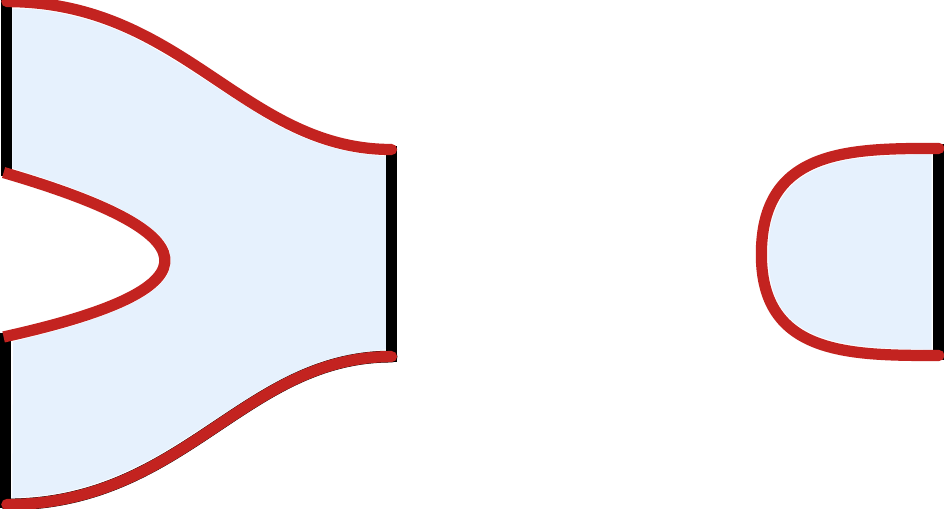
\caption{Algebra multiplication and unit map of $D^1\in \Bord_2^\partial$.}
\end{figure}

There are multiple ways of implementing this. 
One would be to characterize the image of the functor as a (symmetric monoidal) subcategory of $\Bord_2^\partial$, observe that it is in fact a $1$-category, and show it is equivalent to $\Assoc$ as a symmetric monoidal $1$-category.
(In fact, this show that the space of $\Erm_1$-algebra structures on $D^1 \in \OC$ is equivalent to $S^0$ with one point being the structure we define below and the other its opposite.)
For our purposes it will be more convenient to use a slightly more elaborate construction, via the equivalence $\widetilde{\Psi}_\emptyset\colon \boxMfd_{2,\emptyset} \simeq (\Bord_2^\partial)_{\emptyset/}$ from \Cref{cor:Phi-and-Psi}.

For this construction we would like to use a functor of the form
\[
    -\times [0,1] \colon \Mfd_1^\partial \too \boxMfd_{2,\emptyset}
\]
that sends $M \in \Mfd_1^\partial$ to the collar $\col(M) = M \times [0,1]$ where $\pf \col(M) = M \times \{0\}$ and the remaining boundary is $\partial_-$.
While this would be well-defined as a topological manifold, the corner structure is such that this does not define a valid smooth manifold in $\boxMfd_{2,\emptyset}$.
We let $\col(M)$ be the result of smoothing those corners of $M \times [0,1]$ that are in $M \times \{1\}$.
This is not strictly speaking functorial as a functor of topologically enriched categories, but luckily as a functor of \categories{} it has the universal property of a left adjoint.

\begin{lem}\label{lem:col-empty-left-adjoint}
    The free boundary functor $\pf \colon\ \boxMfd_{2,\emptyset} \too \Mfd_1^\partial$ admits a left adjoint
    \[
        L \colon \Mfd_1^\partial \too \boxMfd_{2,\emptyset}
    \]
    that is fully faithful and symmetric monoidal.
    Its value on $M \in \Mfd_1^\partial$ is obtained from $M \times [0,1]$ by smoothing the corners in $M \times \{1\}$.
    Then $\pf L(M) = M \times \{0\}$ and $\partial_- L(M)$ is the smoothing of $(\partial M \times [0,1]) \cup_{\partial M \times \{1\}} M \times \{1\}$.
\end{lem}

 \begin{figure}[h]
\centering
\def\svgwidth{.3\linewidth}
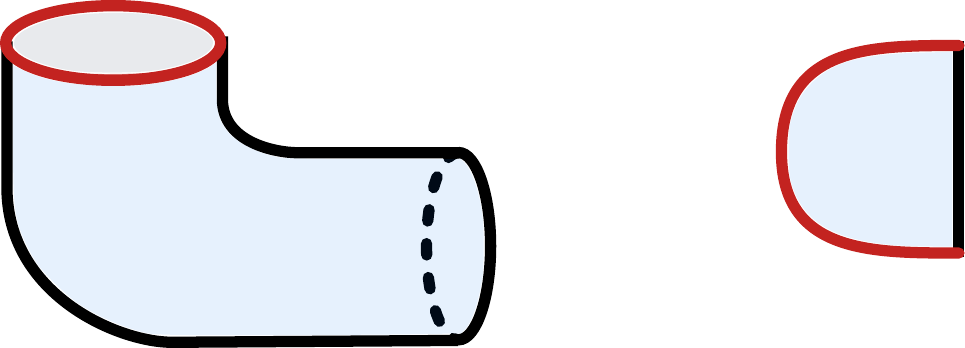
\caption{The images of $S^1$ and $D^1$ under $L$.}
\end{figure}

\begin{proof}
We prove the existence of the adjoint locally.
For a fixed $N\in \Mfd_1^\partial$ let $L(N) \in \boxMfd_2$ be the manifold obtained by smoothing $N \times [0,1]$ as indicated.
We have a canonical isomorphism $\eta_N \colon N \cong \pf L(N) = N \times \{1\}$, which we use as the unit.
To check that this (locally) defines an adjoint, we need to show that for all $W \in \boxMfd_{2,\emptyset}$ the map
    \[
        \Emb^\square(L(N), W) 
        \xtoo{\pf} \Emb(\pf L(N), \pf W)
        \xrightarrow[\cong]{-\circ \eta_N} \Emb(N, \pf W)
    \]
is an equivalence. This follows from the contractibility of collars \cite[ 5.2.1, Corollaire 1]{cerf1961topologie}.
Therefore, the local adjoints assemble into a left adjoint functor $L$.
By construction the unit transformation is invertible, so $L$ is fully faithful.

Since $\pf$ is symmetric monoidal, its left adjoint $L(-)$ is oplax symmetric monoidal.  
Furthermore, the comparison $L(M\sqcup N)\rightarrow L(M)\sqcup L(N)$ is an isomorphism by inspection.
\end{proof}

\begin{defn}\label{def:col}
    We let 
    \[
        \hldef{\col(-)}\colon \Mfd_{1} \xtoo{L} \boxMfd_{2,\emptyset} \xrightarrow[\simeq]{\Psi_\emptyset} (\Bord_2^\partial)_{\emptyset/}
    \]
    denote the composite of the left adjoint from \cref{lem:col-empty-left-adjoint} with the equivalence from \cref{cor:Phi-and-Psi}.
    This functor lands in $\OC_{\emptyset/}$, since every object in the essential image has nonempty free boundary and $\OC_{\emptyset/}\subset(\Bord_2^\partial)_{\emptyset/}$ is full (\Cref{prop: sliceunderff}). We let
    \[
        \hldef{\col(-)'}\colon \Mfd_1^\partial \xtoo{\Psi_\emptyset \circ L} \OC_{\emptyset/} \too \OC
    \]
    denote the functor obtained by forgetting from the slice. This is symmetric monoidal since it is a composite of symmetric monoidal functors.
\end{defn}

 \begin{figure}[h]
\centering
\def\svgwidth{.5\linewidth}
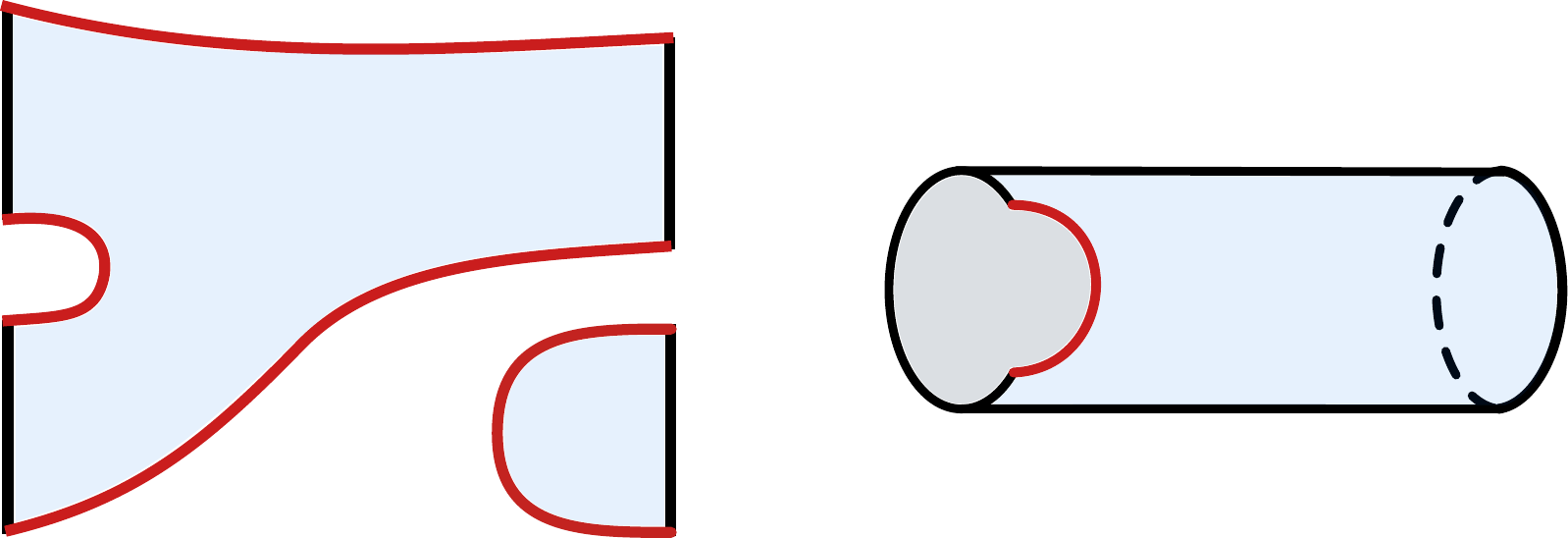
\caption{
On objects $\col'$ sends $D^1$ to $D^1$ and $S^1$ to $S^1$, and embeddings of disks are sent to flat pairs of pants. For instance, the above are the images of the morphism $D^1\sqcup D^1\to D^1\sqcup D^1$ that embeds both disks into the first copy of disk and the embedding $D^1\to S^1$.
}
\end{figure}

\begin{defn}
    The $\Erm_1$-algebra structure on $D^1 \in \OC$ is the one obtained via the symmetric monoidal functor
    \[
        \Disk_1 \subset \Mfd_1^\partial \xtoo{\col(-)} \OC_{\emptyset/} \too \OC.
    \]
\end{defn}

\subsection{Proof of \cref{proposition A}}\label{subsec: proofA}

The convergence of embedding calculus in dimension $1$ implies, via \cref{lem:dense}, the following.

\begin{lem}\label{lem:caluclus-in-dimension-1}
    The diagram 
    \[
        ((\Disk_1)_{/M})^\rhd \too (\Mfd_1^\partial)_{/M} \too \Mfd_1^\partial
    \]
    is a colimit diagram and thus $\int_{M} D^1 = M$ in $\Mfd_1^\partial$.
\end{lem}

\cref{proposition A} says that this is also a colimit diagram in $\OC$. 
To prove this we establish the following lemma about how weakly contractible colimits in slice categories remain colimits in the underlying category.
Note that the subtlety of this lemma is that we do not a priori assume that $\Ccal$ itself has any colimits.

\begin{lem}\label{lem:contractible-colimit-slice}
    Let $\Ccal \in \Cat$ be an \category{} and $x \in \Ccal$.
    Then the projection $\pi\colon \Ccal_{x/} \to \Ccal$ preserves all weakly contractible colimits that exist in $\Ccal_{x/}$.
\end{lem}
\begin{proof}
    Suppose $f\colon I^\rhd \to \Ccal_{x/}$ is a colimit diagram where $I$ is some weakly contractible \category{}.
    To show that $\pi \circ f$ is a colimit diagram, we need to show that for all $y \in \Ccal$ the map
    \[
        \Map(f(\infty), y) \too \lim_{i \in I} \Map(f(i), y)
    \]
    is an equivalence.
    First, consider the case where $\Map(x,y) = \emptyset$.
    In this case we must also have $\Map(f(\infty), y) = \emptyset = \Map(f(i), y)$ for all $i \in I$, as otherwise we could construct a map $x \to y$ by composing with $x \to f(i)$.

    Now we may assume that $\Map(x, y) \neq \emptyset$.
    Every $\alpha\colon x \to y$ defines a lift of $y$ to $\Ccal_{x/}$ and because $f$ is a diagram in $\Ccal_{x/}$ we get a map of fiber sequences.
\[\begin{tikzcd}
	{\Map_{\Ccal_{x/}}(f(\infty), \alpha)} & {\Map_\Ccal((\pi\circ f)(\infty), y)} & {\Map_\Ccal(x, y)} \\
	{\lim_{i \in I} \Map_\Ccal(f(i), \alpha)} & {\lim_{i \in I} \Map_\Ccal((\pi\circ f)(i), y)} & {\lim_{i \in I} \Map_\Ccal(x, y)}
	\arrow[from=1-1, to=1-2]
	\arrow["\simeq", from=1-1, to=2-1]
	\arrow[from=1-2, to=1-3]
	\arrow[from=1-2, to=2-2]
	\arrow["\simeq"', from=1-3, to=2-3]
	\arrow[from=2-1, to=2-2]
	\arrow[from=2-2, to=2-3]
\end{tikzcd}\]
    The top fiber sequence comes from the definition of the slice category $\Ccal_{x/}$.
    The bottom fiber sequence is obtained by taking a limit over $I$ of similar fiber sequences.
    (Being a limit of fiber sequence, it is still a fiber sequence.)
    The left vertical map is an equivalence as $f$ was a colimit diagram in $\Ccal_{x/}$.
    The right vertical map is an equivalence as the limit of a constant diagram over a weakly contractible \category{} is equivalent to the value of the diagram at any point.
    As we can show this for any choice of base-point $\alpha \in \Map(x,y)$, we can conclude that the middle map is an equivalence.
    This shows that $\pi \circ f$ is a colimit diagram, as claimed.
\end{proof}

\begin{proof}[Proof of \cref{proposition A}]
    By \cref{lem:col-empty-left-adjoint} (and \cref{cor:Phi-and-Psi}) the functor 
    \[
    \Mfd_1^\partial \to \boxMfd_{2} \simeq (\Bord_2^\partial)_{\emptyset/}
    \]
    is a left-adjoint.
    This remains true if we restrict its codomain to the subcategory $\OC_{\emptyset/}$, because this subcategory is full by \cref{prop: sliceunderff}.
    Now consider the composite functor
    \[
       \col(-)'\colon \Mfd_1^\partial \xtoo{\Psi_\emptyset \circ L} \OC_{\emptyset/} \too \OC.
    \] 
    The first functor preserves colimits because we just showed it is a left adjoint and the second functor preserves weakly contractible colimits by \cref{lem:contractible-colimit-slice}.
    The category $\Disk_{/M}$ is weakly contractible, as we shall see below in \cref{cor:DiskM-weakly-contractible}.
    Hence, $\col(-)'$ preserves the colimit from \cref{lem:caluclus-in-dimension-1}, which yields the desired colimit in $\OC$.
\end{proof}

\begin{rem}
    The above proof also show that colimit $\int_M D^1 = M$ still holds in $\Bord_2^\partial$.
    (Simply omit the step of the proof where we restricted to $\OC$.)
    In contrast, this is not true in $\OC^\op$, i.e.~the map 
    \[
        \Map_{\OC^\op}(S^1, M) \too
        \lim_{\sqcup_n D^1 \to S^1}\Map_{\OC^\op}(\sqcup_n D^1, M)
    \]
    is not an equivalence for all $M\in \OC$. 
    For instance, take $M=\emptyset$ and fix a bordism $W\colon S^1\rightarrow \emptyset$ that is connected and has empty free boundary.
    Under the equivalence $ \Map_{\Bord_2^\partial}(S^1, \emptyset)\xto{\simeq}\lim\Map_{\Bord_2^\partial}(\sqcup_n D^1, \emptyset)$, the collection of bordisms $W\circ C_n\colon \sqcup_n D^1\to\emptyset$ given by composing with the cylinder bordism $\sqcup_n D^1\to S^1$  gives rise to a bordism equivalent to $W\colon S^1\to\emptyset$.
    Note that $W\circ C_n$ for all $n\geq1$ are morphisms in $\OC^\op$, but $W$ is not.
\end{rem}

%% file: figures/multunit.pdf_tex
\begingroup%
  \makeatletter%
  \providecommand\color[2][]{%
    \errmessage{(Inkscape) Color is used for the text in Inkscape, but the package 'color.sty' is not loaded}%
    \renewcommand\color[2][]{}%
  }%
  \providecommand\transparent[1]{%
    \errmessage{(Inkscape) Transparency is used (non-zero) for the text in Inkscape, but the package 'transparent.sty' is not loaded}%
    \renewcommand\transparent[1]{}%
  }%
  \providecommand\rotatebox[2]{#2}%
  \newcommand*\fsize{\dimexpr\f@size pt\relax}%
  \newcommand*\lineheight[1]{\fontsize{\fsize}{#1\fsize}\selectfont}%
  \ifx\svgwidth\undefined%
    \setlength{\unitlength}{452.03887219bp}%
    \ifx\svgscale\undefined%
      \relax%
    \else%
      \setlength{\unitlength}{\unitlength * \real{\svgscale}}%
    \fi%
  \else%
    \setlength{\unitlength}{\svgwidth}%
  \fi%
  \global\let\svgwidth\undefined%
  \global\let\svgscale\undefined%
  \makeatother%
  \begin{picture}(1,0.53877163)%
    \lineheight{1}%
    \setlength\tabcolsep{0pt}%
    \put(0,0){\includegraphics[width=\unitlength,page=1]{multunit.pdf}}%
  \end{picture}%
\endgroup%

%% file: figures/L.pdf_tex
\begingroup%
  \makeatletter%
  \providecommand\color[2][]{%
    \errmessage{(Inkscape) Color is used for the text in Inkscape, but the package 'color.sty' is not loaded}%
    \renewcommand\color[2][]{}%
  }%
  \providecommand\transparent[1]{%
    \errmessage{(Inkscape) Transparency is used (non-zero) for the text in Inkscape, but the package 'transparent.sty' is not loaded}%
    \renewcommand\transparent[1]{}%
  }%
  \providecommand\rotatebox[2]{#2}%
  \newcommand*\fsize{\dimexpr\f@size pt\relax}%
  \newcommand*\lineheight[1]{\fontsize{\fsize}{#1\fsize}\selectfont}%
  \ifx\svgwidth\undefined%
    \setlength{\unitlength}{462.07237424bp}%
    \ifx\svgscale\undefined%
      \relax%
    \else%
      \setlength{\unitlength}{\unitlength * \real{\svgscale}}%
    \fi%
  \else%
    \setlength{\unitlength}{\svgwidth}%
  \fi%
  \global\let\svgwidth\undefined%
  \global\let\svgscale\undefined%
  \makeatother%
  \begin{picture}(1,0.35995702)%
    \lineheight{1}%
    \setlength\tabcolsep{0pt}%
    \put(0,0){\includegraphics[width=\unitlength,page=1]{L.pdf}}%
  \end{picture}%
\endgroup%

%% file: figures/colemb.pdf_tex
\begingroup%
  \makeatletter%
  \providecommand\color[2][]{%
    \errmessage{(Inkscape) Color is used for the text in Inkscape, but the package 'color.sty' is not loaded}%
    \renewcommand\color[2][]{}%
  }%
  \providecommand\transparent[1]{%
    \errmessage{(Inkscape) Transparency is used (non-zero) for the text in Inkscape, but the package 'transparent.sty' is not loaded}%
    \renewcommand\transparent[1]{}%
  }%
  \providecommand\rotatebox[2]{#2}%
  \newcommand*\fsize{\dimexpr\f@size pt\relax}%
  \newcommand*\lineheight[1]{\fontsize{\fsize}{#1\fsize}\selectfont}%
  \ifx\svgwidth\undefined%
    \setlength{\unitlength}{766.71966312bp}%
    \ifx\svgscale\undefined%
      \relax%
    \else%
      \setlength{\unitlength}{\unitlength * \real{\svgscale}}%
    \fi%
  \else%
    \setlength{\unitlength}{\svgwidth}%
  \fi%
  \global\let\svgwidth\undefined%
  \global\let\svgscale\undefined%
  \makeatother%
  \begin{picture}(1,0.34225153)%
    \lineheight{1}%
    \setlength\tabcolsep{0pt}%
    \put(0,0){\includegraphics[width=\unitlength,page=1]{colemb.pdf}}%
  \end{picture}%
\endgroup%

%% file: sections/cyclic.tex
\section{The cyclic bar construction in \texorpdfstring{$\OC$}{OC}}
In this section, we identify a model of the paracyclic category $\Lambda_\infty$ with a co-cone point as a full subcategory of $\OC$. 
This is the colimit diagram computing the factorization homology over $S^1$ in $\Psh(\Ocal)$, which we will use in proving \cref{proposition B}.
We further identify this colimit diagram with the colimit diagram in \cref{proposition A} that indexes the cyclic bar construction on the $\Erm_1$-algebra $D^1\in\OC$.

\subsection{Comparison with systems of arcs}
\begin{notation}\label{notn: Cn}
 For $1\leq n<\infty$, let \hldef{$C_n$}$\in\boxMfd_{2,S^1}$ be the manifold with corner obtained from the cylinder $S^1\times[0,1]$ by deleting a disjoint union of $n$ half $2$-disks, which are the intersections of  $B_{
\frac{1}{4n}}(\frac{2k-1}{2n})$ and $S^1\times [0,1]$ for $ 1\leq k\leq n$, and then taking closure. 
Here we let the circle have circumference equal to 1 and $B_r(a)$ denotes the $2$-disk of radius $r$ centred at $(a,0)$, so $\partial_-C_n$ is the disjoint union of the intervals $[(\frac{2k-1}{2n}+
\frac{1}{4n},0),(\frac{2k
+1}{2n}-\frac{1}{4n},0)]\subset S^1\times\{0\}$.  Denote by $D_k$ the free boundary arcs from $(\frac{2k-1}{2n}-
\frac{1}{4n},0)$ to $(\frac{2k-1}{2n}+\frac{1}{4n},0)$, so $\pf C_n=\sqcup_{k=1}^n D_k$. Let $x_k=(\frac{2k-1}{2n},
\frac{1}{4n})$ be the centre point of free boundary interval $D_k$, which is the intersection $D_k\cap S^1\times \{\frac{1}{4n}\}$. See \cref{fig:C_3} for an illustration of the case $n=3$. 
 
 \begin{figure}[h]
\centering
\def\svgwidth{.5\linewidth}
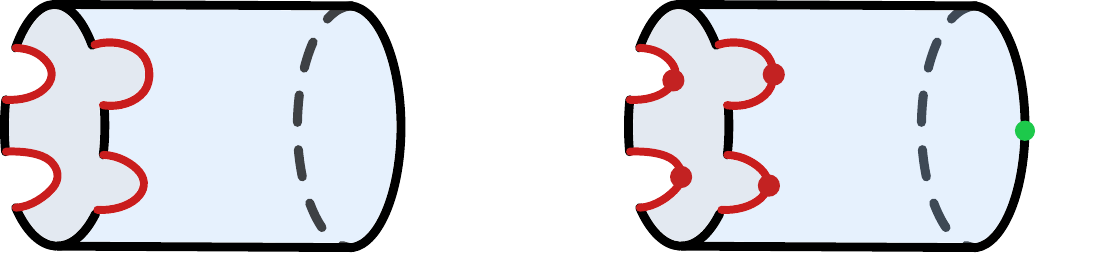
\caption{$C_4$ and its labelling.}
\label{fig:C_3}
\end{figure}
 The $n$ boundary arcs come with a canonical cyclic ordering extending the linear ordering $x_1< \ldots< x_n$, which we take to be in the counterclockwise direction. We further fix a base point $s_0=(\frac{1}{2},1)\in S^1$.
\end{notation}

The embedding spaces of the $C_n$'s have a classical interpretation as systems of arcs, which we recall below.

\begin{defn}\label{defn:arc}
    Let $W$ be an oriented 2-manifold, $s_0 \in \partial W$ a base-point in the boundary, and $P \subset W\setminus \{s_0\}$ some submanifold of dimension $0$ or $1$.
    An \hldef{arc in $W$ rel.~$P$} a smooth embedding $\gamma\colon [0,1]\hookrightarrow W$ such that $\gamma(0)=s_0$ and $\gamma(1)\in P$.
    For any $k\geq 1$, a \hldef{system of $n$ arcs in $W$ rel.~$P$} is an $n$-tuple $(\gamma_1,\ldots,\gamma_n)$ of arcs in $W$ rel.~$P$ such that:
    \begin{enumerate}[(1)]
        \item the $\gamma_i$'s do not intersect other than the start point and possibly at the end points, and \item the $\gamma_i$'s appear in the order $\gamma_1 \le \dots \le \gamma_n$ at the start point $s_0$.
    \end{enumerate}
    We denote by \hldef{$\Arc_n(W; P)$} the set of systems of $n$ arcs in $W$ rel.~$P$ up to \emph{simultaneous} isotopy.
    For $V \in \boxMfd_{2,S^1}$ we choose the basepoint $s_0 \in S^1 \subset V$ and write $\Arc_n(V; \pf) \coloneq \Arc_n(V; \pf V)$.
\end{defn}

\begin{obs}\label{rem: orderedarcs}
    It is well known that the set of isotopy classes of $n$ disjointly embedded arcs $\alpha_1,\ldots,\alpha_n$ in $W$ is isomorphic to the set of unordered $n$-tuples $([\alpha_1],\ldots, [\alpha_n])$ of isotopy classes of arc in $W$ such that there exists representatives $\alpha_1,\ldots, \alpha_n$ that are pairwise disjoint other than possibly at the endpoints. 
    Moreover, this $n$-tuple of representatives is unique up to \emph{simultaneous} isotopy.
    Furthermore, there is a canonical ordering associated to each isotopy class of $n$ embedded arcs, which is given by the cyclic ordering of the arcs in a small collar of the outgoing circle boundary and independent of the choice of representatives. 
    A proof of those two facts follows from a straightforward adaptation of \cite[p.552-553]{wahl2010survey}.
\end{obs}

From this we in particular deduce the following.
\begin{cor}\label{cor:arc-into-tuple}
    The map
    \begin{align*}
        \Arc_n(W; P) &\too \Arc_1(W;P) \times \dots \times \Arc_1(W;P)  \\
        [\gamma_1,\dots,\gamma_n] & \longmapsto ([\gamma_1], \dots, [\gamma_n])
    \end{align*}
    is injective.
\end{cor}

  \begin{notation}\label{not: standardcollection}
      We fix a ``standard collection'' of arcs $[\beta_1,\dots, \beta_n] \in \Arc_n(C_n; \pf)$ for all $n$,  where each $\beta_i$ is the minimal geodesic from $s_0$ to $x_i$ in the flat metric. For example, when $n=4$ this is illustrated in \cref{fig:standardarc}.
       \begin{figure}[h]
\centering
\def\svgwidth{.4\linewidth}
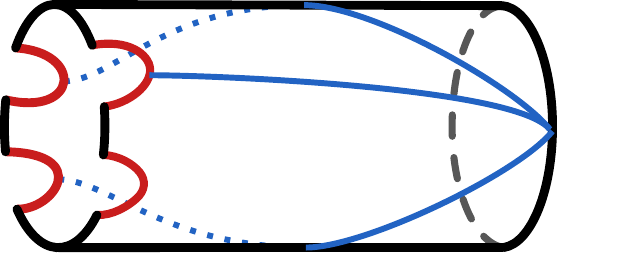
\caption{A standard collection of arcs in $C_4$.}
\label{fig:standardarc}
\end{figure}
  \end{notation}  

\begin{lem}\label{lem: map=arc systems}
    For all $M\in\boxMfd_{2,S^1}$ and $n\geq 1$ the map
    \begin{align*}
        \Emb^\square_{S^1}(C_n,M) &\too \Arc_n(M; \pf)\\
        (i\colon C_n \hookrightarrow M)& \longmapsto [i\circ \beta_1, \dots, i \circ \beta_n]
    \end{align*}
    is an equivalence.
    In particular, every path component of $\Emb^\square_{S^1}(C_n,M)$ is contractible.
\end{lem}
As an aid to the reading of the proof of this lemma, we provide illustration of the major steps in \cref{fig:toplem}.
\begin{proof}
    Let $A_n \subset C_n$ denote the disjoint union of the $n$ straight lines $\{\tfrac{k}{n}\}\times [\tfrac{1}{4}n,1]$ and let $B_n \subset C_n$ be a tubular neighbourhood of $A$.
    We write $\Emb_{S^1}^\square(A_n, M)$ for the space of smooth embeddings of these arcs such they are the identity on $A_n \cap (S^1 \times \{1\})$ and such that they send the other endpoints of the arcs to the free boundary, and similarly for $B_n$.
    (This is a slight abuse of notation as $A_n$ and $B_n$ do not actually contain all of $S^1 \times \{1\}$.)
    The restriction map $\boxEmb_{S^1}(B_n, M) \to \boxEmb_{S^1}(A_n, M)$ is a fibration and its fiber at some $i\colon A_n \hookrightarrow M$ is the space of (parametrized) tubular neighbourhoods of $i(A_n)$ in $M$ that have a fixed parametrization at $i(A_n \cap S^1 \times \{1\})$.
    Because we fix the parametrization at exactly one endpoint for each arc, this space of tubular neighbourhoods is contractible,
    and therefore the restriction map is an equivalence.
    
    Fix $0<\varepsilon<\tfrac{1}{2}$ and consider the commutative square
    \[
    \begin{tikzcd}
        {\boxEmb_{S^1\times[1-\varepsilon,1]}(C_n, M)} \dar \rar[hook] &
        {\boxEmb_{S^1}(C_n, M)} \dar \\
        {\boxEmb_{S^1\times[1-\varepsilon,1]}(B_n, M)} \rar[hook] \ar[ur, dashed, "H"] &
        {\boxEmb_{S^1}(B_n, M)}
    \end{tikzcd}
    \]
    where the vertical maps are restrictions and the horizontal maps are the inclusion of the subspace of those embeddings $C_n \hookrightarrow M$ that are the identity on $S^1 \times [1-\varepsilon, 1]$, and those embeddings $B_n \hookrightarrow M$ that are the identity on $B_n \cap (S^1 \times [1-\varepsilon, 1])$.
    The horizontal inclusions are equivalences by Cerf's contractibility of the space of collars.
    To construct the dashed map, let $\varphi\colon C_n \hookrightarrow C_n$ be in $\Emb_{S^1}^\square(C_n, C_n)$ such that it is isotopic to the identity, that $\varphi(C_n) \subset B_n \cup (S^1 \times [1-\varepsilon, 1])$, and that $\varphi_{|B_n}$ is isotopic to the identity in $\boxEmb_{S^1}(B_n, B_n)$.
    We can then define the dashed map to be 
    \[
        H\colon (i\colon B_n \hookrightarrow M) \longmapsto (C_n \xhookrightarrow{\varphi} B_n \cup (S^1 \times [1-\varepsilon,1]) \xrightarrow{i \cup \id} C_n)
    \]
    If the original map $i$ was restricted from some $j \in \boxEmb_{S^1 \times [0,\varepsilon]}(C_n, M)$, then $H(i) = j \circ \varphi$.
    Since $\varphi$ is isotopic to the identity, $H(i)$ is isotopic to $j$ and this defines a homotopy for the top triangle in the diagram.
    Similarly, for all $i$ we have an isotopy between $H(i)_{|B_n}$ and $i$ because $\varphi_{|B_n}$ is isotopic to $\id_{B_n}$.
    Therefore, the entire square commutes up to homotopy, and it follows from $2$-out-of-$6$ that all the maps are weak equivalences.

    \begin{figure}
    \centering
    \includegraphics[width=.9\linewidth]{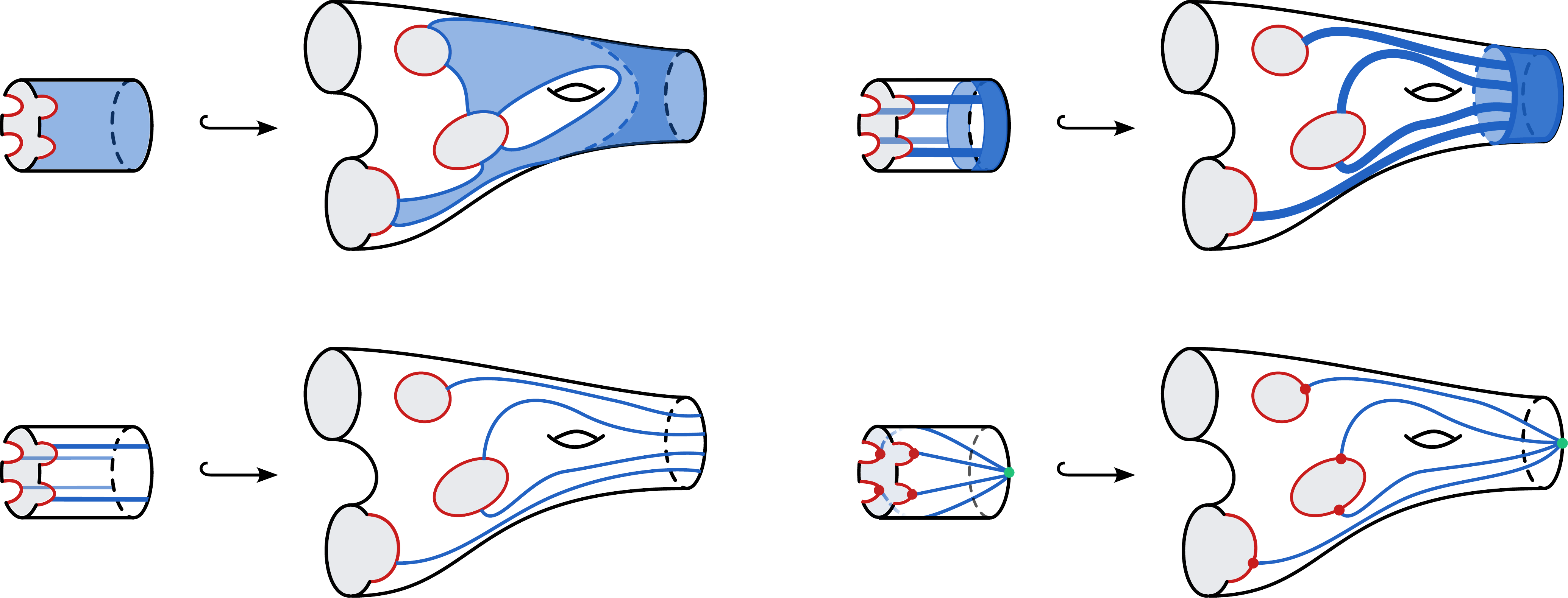}
    \caption{Proof of \cref{lem: map=arc systems} by picture.}
    \label{fig:toplem}
    \end{figure}

    It remains to compare $\boxEmb_{S^1}(A_n, M)$ with $\Arc_n(M; \pf)$.
    By \cite{gramain} the space $\boxEmb_{S^1}(A_1,M)$ has contractible path components and an inductive argument shows that the same is true for $\boxEmb_{S^1}(A_n,M)$.
    The path components are isotopy classes of $n$-tuples of disjoint arcs with prescribed starting points on the fixed boundary $S^1$ and whose endpoints lie on the free boundary.
    As discussed in \cref{rem: orderedarcs} we can equivalently consider $n$-tuples of isotopy classes of arc that have the property that they can be made disjoint.
    Finally, we can use our set of ``standard arcs'' in $C_n$ (\cref{not: standardcollection}) to move all the arcs so that they have the same starting point, resulting in the bijection $\pi_0\boxEmb_{S^1}(A_n, M) \cong \Arc_n(M; \pf)$. 
    (Crucially, this step uses that the arcs appear in the standard order at the base-point.)
\end{proof}

\begin{notation}
    Denote by \hldef{$\Cyc$} the full subcategory of $\boxMfd_{2,S^1}$ generated by $C_n, 1\leq n<\infty$.
\end{notation}

\begin{cor} \label{cor:Cyc-1cat}
    $\Cyc$ is equivalent to a $1$-category.
\end{cor}
\begin{proof}
    It follows from \Cref{lem: map=arc systems} that all the mapping spaces in $\Cyc$ are equivalent to discrete sets
    \(
        \boxEmb_{S^1}(C_n, C_m) \simeq \Arc_n(C_m; \pf)
    \).
\end{proof}


\subsection{A model for the paracyclic category}
In this subsection, we show that $\Cyc$ is equivalent to the paracyclic category $\Lambda_\infty$, whose definition we recall now.

\begin{defn}[{\cite[Appendix B]{NS2018}}]
    Let $\mathbb{Z}$PoSet be the $1$-category of posets with an action by $\mathbb{Z}$ and non-decreasing maps compatible with the $\mathbb{Z}$-actions.
    Define the \hldef{paracyclic category $\Lambda_\infty$} to be the full subcategory of $\mathbb{Z}$PoSet consisting of objects that are isomorphic to $(\tfrac{1}{n}\mathbb{Z}, \le, +1)$ for some integer $n\geq 1$. 
\end{defn}

\begin{const}\label{const:q}
    We define a functor \hldef{$q\colon \Cyc \to \Lambda_\infty$}.
    Let $p\colon \widetilde{C_n} \to C_n$ denote the universal cover and let
    $\hldef{q(C_n)} \coloneq p^{-1}(\{x_1,\dots,x_n\})$ denote the preimage of the midpoints of the free boundaries.
    If we think of the universal cover of $C_n$ as a subset $\widetilde{C_n} \subset \Rbb \times [0,1]$, then $q(C_n) = \{(\tfrac{2i+1}{2n}, \tfrac{1}{4n})\;|\; i \in \Zbb\}$.
    By recording the first coordinate we can identify this with a subset of $\Rbb$, which induces a total order, and we define the $\Zbb$-action to be addition in the first coordinate.
    (Equivalently, the $\Zbb$-action is by Deck transformations of the universal cover.)
    The $\Zbb$-poset $q(C_n)$ is isomorphic to $(\tfrac{1}{n}\mathbb{Z}, \le, +1)$ and thus an object of $\Lambda_\infty \subset \Zbb\text{PoSet}$.
    A morphism $i\colon C_n \hookrightarrow C_m$ induces an embedding on universal covers (and this is unambiguous as $i$ fixes $S^1 \times \{1\}$), which induces the map $q(i)\colon q(C_n) \to q(C_m)$.
\end{const}

We can define a map of sets \hldef{$\epsilon\colon \Arc_1(C_n;\{x_1,\ldots, x_n\}) \to q(C_n)$} by sending an arc $[\gamma]$ to the endpoint of the lift of $\gamma$ to the universal over $\widetilde{C_n}$.

\begin{lem}\label{lem: q(Cn)=Arc1}
    The map $\epsilon$ is a bijection.
\end{lem}
\begin{proof}
    We can construct an inverse $\delta\colon q(C_n) \to \Arc_1(C_n; \{x_1,\dots,x_n\})$ as follows.
    For a point $y \in q(C_n) = p^{-1}(\{x_1,\dots,x_n\})$ we let $\sigma_y\colon [0,1] \to \widetilde{C_n} \subset \Rbb \times [0,1]$ be the straight line from the base point $s_0 = (\tfrac{1}{2},1)$ to $y$.
    Then we set $\delta(y) \coloneq [p \circ \sigma_y]$ to be the isotopy class of the path obtained by projecting $\sigma_y$ back to $C_n$.
    This is by construction a path from $s_0$ to $p(y) \in \{x_1,\dots, x_n\}$, and it does not self-intersect.
    The composite $\epsilon \circ \delta$ is the identity by construction.
    To see $\delta \circ \epsilon = \id$, consider some $[\gamma] \in \Arc_1(C_n; \{x_1,\dots,x_n\})$.
    Without loss of generality, we can assume that $\gamma$ is a geodesic.
    Then the lift of $\gamma$ to the universal cover must be exactly the straight line $\sigma_y$ with $y = \epsilon(\gamma)$, so after composing with $p$ we get $[\gamma] = \delta(y)$.
\end{proof}

This endows $\Arc_1(C_n;\{x_1,\ldots, x_n\})$ with the structure of a $\mathbb{Z}$Poset, with $\mathbb{Z}$-action given by the Dehn twists 
and total ordering the unique one extending $[\beta_1]<\ldots<[\beta_n]<[\beta_1]+1$.

\begin{lem}\label{prop: orderedarc}
    Consider the map
    \[
        (\epsilon_1,\dots,\epsilon_k)\colon \Arc_k(C_n; \{x_1,\dots,x_n\}) \too q(C_n)^{\times k}
    \]
    that records the end-points of the $k$ arcs after lifting them to the universal cover.
    This map is injective and a $k$-tuple $(y_1,\dots,y_k)$ is in the image if%
    \footnote{In fact, this is an ``if and only if'', but we will not need the other implication here.}
    it satisfies
    \( y_1 \le y_2 \le \dots \le y_k \le y_1 + 1 \).
\end{lem}
\begin{proof}
    By \cref{cor:arc-into-tuple} an isotopy class of arc systems is uniquely determined by its tuple of isotopy classes of arcs, and 
    by \cref{lem: q(Cn)=Arc1} these isotopy classes of arcs are in turn uniquely determined by their endpoints $\epsilon(\gamma_i)$ in the universal cover, so the map is indeed injective.
    Now suppose $(y_1,\dots, y_n) \in q(C_n)^{\times k}$ is a tuple satisfying the inequality.
    Take the geodesics (doubling if necessary) from $(1/2,1)$ to each $y_i$ and wrapping the universal cover $\widetilde{C_n}$ around $C_n$. This yields an embedding of $k$ arcs in $C_n$, and the condition $y_k< y_1+1$ ensures that they are pairwise non-intersecting except possibly at end points. This defines a system of $k$ arcs in $C_n$ as desired.
\end{proof}

\begin{lem} \label{lem: Cyc=Lambda infty}
    The functor $q\colon \Cyc \to \Lambda_\infty$ is an equivalence of $\infty$-categories.
    (In fact, both \categories{} are equivalent to $1$-categories.)
\end{lem}
\begin{proof}
    The functor $q$ is essentially surjective because $q(C_n) \cong (\tfrac{1}{n}\Zbb,\le,+1)$.
    It remains to check fully faithfulness, and since $\Lambda_\infty \subset \Zbb\text{PoSet}$ is full by definition, we need to check that 
    \[
    \Emb^\partial_{S^1}(C_n,C_m)\xtoo{q}\Map_{\Lambda_\infty}(q(C_n),q(C_m)))
    \cong\Map_{\mathbb{Z}\mathrm{Poset}}(q(C_n),q(C_m)))
    \]
    is an equivalence for all $m,n$. 
    By \Cref{lem: map=arc systems}, the source is equivalent to $\Arc_n(C_m)$, where we suppress the endpoints $\{x_1,\ldots,x_m\}$ for ease of notation.     

    Let $\epsilon(\beta_1),\dots, \epsilon(\beta_n) \in q(C_n)$ be the endpoints (after lifting to the universal cover) of the collection of standard arcs $\beta_1,\dots, \beta_n$ in $C_n$.
    Consider the commutative square 
    \[
    \begin{tikzcd}
        \Emb^\partial_{S^1}(C_n, C_m)\ar[r,"q"]\ar[d,"{(\ev_{\beta_1},\dots, \ev_{\beta_n})}"',"\simeq"]&\Map_{\mathbb{Z}\mathrm{Poset}}(q(C_n), q(C_m))\ar[dl,dashed] \dar[hook, "{(\ev_{\epsilon(\beta_1)},\dots,\ev_{\epsilon(\beta_n)})}"]\\
        {\Arc_n(C_m, \{x_1,\dots,x_m\})} \rar[hook, "{(\epsilon_1,\dots,\epsilon_n)}"] & 
        q(C_m)^{\times n}
    \end{tikzcd} 
    \]
    where the top map applies the functor $q$, the bottom map takes the (lifted) endpoints of the arcs,  the right map applies a morphism $f\colon q(C_n) \to q(C_m)$ to the elements $\epsilon(\beta_1),\dots,\epsilon(\beta_n) \in q(C_m)$,
    and the left map applies the embedding to the standard arcs.
    (Here we implicitly use the canonical identification $\Arc_n(C_m, \{x_1,\dots,x_m\}) \cong \Arc_n(C_m; \pf)$, i.e.~we isotope the arcs so that they end at the midpoint of their respective disk, see \cref{prop:arc-cpx-comparison}.)
    For every $f$, these elements satisfy the inequalities
    \[
    f(\epsilon(\beta_1))\le \dots \le f(\epsilon(\beta_n)) \le f(\epsilon(\beta_1)) + 1
    \]
    because they satisfy this before applying $f$ and $f$ preserves both the partial order and the $\Zbb$-action.
    By \cref{prop: orderedarc}, this implies that the right map factors through the bottom map (which is injective) thus giving the dashed map.

    Note that the right map in the square is injective because a map of $\Zbb$-posets $q(C_n) \to q(C_m)$ is uniquely determined by where it sends the $n$ points that generate the $\Zbb$-orbits.
    Hence, the dashed map must be injective.
    The left map in the square is an equivalence by \cref{lem: q(Cn)=Arc1}, so the dashed map is also surjective, and in fact the top map must be an equivalence as claimed.
\end{proof}

\subsection{Cone and co-cone on \texorpdfstring{$\Cyc$}{ Cyc}}
In this section, we add a cone point and a co-cone point to $\Cyc$ as a full subcategory of $\boxMfd_{2,S^1}$. This will serve as a model for $\Lambda_\infty^{\lhd\rhd}$. 
\begin{notation}\label{notation:cols}
    Let \hldef{$\cyl_{\rhd}$}$\in\boxMfd_{2,S^1}$ be the cylinder $S^1\times [0,1]$ with $\pf=S^1\times\{0\}$, $\partial_+ = S^1 \times \{1\}$, and $\partial_-=\emptyset$. 
    Let \hldef{$\cyl_{\lhd}$}$\in\boxMfd_{2,S^1}$ be the cylinder $S^1\times [0,1]$ with $\pf=\emptyset$, $\partial_+ = S^1 \times \{1\}$, and $\partial_-=S^1\times\{0\}$, see \cref{fig:cone-and-cocone}.
    We let $\hldef{\Cyc^{\lhd\rhd}} \subset \boxMfd_{2,S^1}$ denote the full subcategory generated by $\cyl_\lhd$, $\cyl_\rhd$ and $C_n$ for $1\leq n<\infty$.
\end{notation}

\begin{rem}
    As bordisms, $\Phi_{S^1}(\cyl_\rhd)$ has empty ingoing boundary, outgoing boundary $S^1\times \{1\}$ and free boundary $S^1\times \{1\}$, whereas $\Phi_{S^1}(\cyl_\lhd)$ is the identity bordism with ingoing and outgoing boundaries both given by $S^1$.
\end{rem}

We would like to show that $\cyl_{\rhd}$ is terminal with respect to objects in $\Cyc$, and $\cyl_\lhd$ is initial with respect to objects in $\Cyc$. We do so by computing the following mapping spaces.

\begin{figure}[ht]
    \centering
    \def\svgwidth{.8\linewidth}
    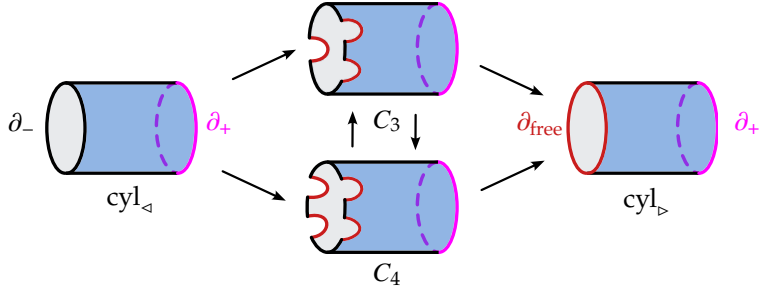
    \caption{A sketch of (some of) the objects in $\Cyc^{\lhd\rhd}$. Morphisms only exist in the indicated directions because the embeddings must restrict to an embedding on the free boundary.}
    \label{fig:cone-and-cocone}
\end{figure}

\begin{lem}\label{lem:mappingspaceinCyc++}
\begin{enumerate}[(1)]
    \item For all $1\leq n<\infty$, $\Map_{\boxMfd_{2,S^1}}(C_n,\cyl_\rhd)\simeq *$.
    \item $\Map_{\boxMfd_{2,S^1}}(\cyl_\rhd,\cyl_\rhd)\simeq *$.
    \item $\Map_{\boxMfd_{2,S^1}}(\cyl_\lhd,W)\simeq *$ for any $W\in\boxMfd_{2,S^1}$.
\end{enumerate}
\end{lem}
\begin{proof}
    For the first claim, by \Cref{lem: map=arc systems}, it suffices to show that $\Arc_n(\cyl_\rhd)$ contains only one element. 
    By \cref{cor:arc-into-tuple} this set injects to $\Arc_1(\cyl_\rhd)^{\times n}$ and $\Arc_1(\cyl_\rhd)$ only has one element by \cite[Theorem 6]{gramain}: all arcs from $s_0 \in S^1 \times \{1\}$ to the free boundary $S^1 \times \{0\}$ are isotopic.
    (Here the isotopies are allowed to move the end-point within the free boundary.)

    For the second claim, we want to show that $\Emb^{\square}_{S^1}(\cyl_\rhd, \cyl_\rhd)\simeq *$. 
    Because half of the boundary is $\partial_+$ and the other half is $\pf$, this is the space of those diffeomorphisms of $S^1 \times [0,1]$ that fix $S^1 \times \{1\}$ pointwise.
    In other words, this is the space of pseudo-isotopies of $S^1$, which is contractible as a consequence of \cite[57]{gramain}.
    (It is the fiber of the fibration $\Diff(S^1 \times [0,1], S^1 \times \{1\}) \rightarrow \Diff(S^1 \times \{1\})$ and this map is an equivalence where both sides are equivalent to $S^1$.)

    For the third claim, note that $\Emb^{\square}_{S^1}(\cyl_\lhd, M)$ is precisely the space of collars of the boundary component $S^1$, so it is contractible by \cite[Section 5.2.1]{cerf1961topologie}. 
    Alternatively, we know that the identity bordism $\Phi_{S^1}(\cyl_\lhd)=S^1\times [0,1]$ with ingoing and outgoing boundary $S^1$ is the terminal object in $(\Bord_2^\partial)_{/S^1}$. Hence, it is the initial object in $\boxMfd_{2,S^1}\simeq \big((\Bord_2^\partial)_{/S^1}\big)^\op$ via the equivalence $\Phi_{S^1}$.
\end{proof}

Combining \Cref{cor:Cyc-1cat} and \Cref{lem:mappingspaceinCyc++}, we get the following description of $\Cyc^{\lhd\rhd}$.

\begin{cor} \label{cor: conedcyc=lambda_infty}
    The $\infty$-category $\Cyc^{\lhd\rhd}$ is obtained by adding a cone point and a co-cone point to $\Cyc$.
    Therefore, it is equivalent to the $1$-category $\Lambda_\infty^{\lhd\rhd}$ obtained by adding a cone point and co-cone point to $\Lambda_\infty$.
\end{cor}

\subsection{Comparing colimit diagrams}

When proving \cref{proposition B}, we will be using the equivalence
\[
    \widetilde{\Phi}_{S^1}\colon (\boxMfd_{2,S^1})^\op \simeq (\Bord_2)_{/S^1}
\]
from \cref{thm:Phi-and-Psi}
in order to construct the cone diagram $(\Lambda_\infty^\op)^\rhd \too \OC$, for which we can show that the composite with $\Yo_\Ocal \colon \OC \too \PSh(\Ocal)$ is a colimit diagram.
The paracyclic category is self-dual, meaning that there is an equivalence $\Lambda_\infty^\op \simeq \Lambda_\infty$, so we can also think of this as a colimit of a paracyclic diagram in $\PSh(\Ocal)$.
However, to prove \cref{proposition B} we need to compute the colimit of $\Yo_\Ocal(-)$ applied to the diagram in \cref{proposition A}, and in fact we also need to make sure that the comparison map between the colimit and $\Yo_\Ocal(S^1)$ is $\Yo_\Ocal$ applied to the map from \cref{proposition A}.
This means that we need to compare the two cone-diagrams involved.
This is exactly the content of the following lemma.

Recall from \Cref{def:col} that the functor $\col(-):\Mfd_1 \too \OC_{\emptyset/}$ that sends $D^1$ to $D^1\times[0,1]$ with the corners at $\{0\}$ rounded, which defines a bordism from $\emptyset$ to $D^1$ with free boundary $D^1$; while $S^1$ is sent to the cylinder $S^1\times[0,1]$ considered as a bordism from $\emptyset$ to $S^1$ with free boundary $S^1$. 
The next lemma says that, after post-composing with the projection to $\OC$, the induced diagram $(\Disk_1)_{/S^1} \xtoo{\col(-)'} \OC$ is the colimit diagram indexing the cyclic bar construction on the $\mrm{E}_1$-algebra $D^1\in\OC$.

\begin{lem}\label{lem: comparediagram}
    The diagrams
    \[
        (\Disk_1)_{/S^1}^\rhd \xtoo{\col(-)'} \OC
        \qquad\text{and}\qquad
        (\Cyc^{\lhd\rhd})^\op \hookrightarrow (\boxMfd_{2,S^1})^\op \xtoo{\Phi_{S^1}} \OC
    \]
    are equivalent as objects of $(\Cat)_{/\OC}$.
\end{lem}
\begin{proof}
    Let $\Dcal \subset (\Bord_2^\partial)_{\emptyset/}$ denote the full subcategory on the essential image of the functor $\col(-)\colon \Mfd_1^\partial \to (\Bord_2^\partial)_{\emptyset/}$.
    Let $\Ccal \subset (\Bord_2^\partial)_{/S^1}$ denote the essential image of $(\Cyc^{\lhd\rhd})^\op$ under the functor $\widetilde{\Phi}_{S^1}\colon (\boxMfd_{S^1})^\op \to (\Bord_2^\partial)_{/S^1}$.
    Then we have a commutative diagram:
%
\[\begin{tikzcd}
	&& {\Ccal_{\widetilde{\Phi}_{S^1}(\cyl_\rhd)/}} & \Ccal & {(\Cyc^{\lhd\rhd})^\op} \\
	{(\Mfd_1)_{/S^1}} & {\Dcal_{/\col(S^1)}} & \begin{array}{c} \substack{((\Bord_2^\partial)_{/S^1})_{\col(S^1)/} \\ \simeq ((\Bord_2^\partial)_{\emptyset/})_{/\col(S^1)}} \end{array} & {(\Bord_2^\partial)_{/S^1}} \\
	{\Mfd_1} & \Dcal & {(\Bord_2^\partial)_{\emptyset/}} & {\Bord_2^\partial}
	\arrow["\simeq", from=1-3, to=1-4] \arrow["\cap"{description}, draw=none, from=1-3, to=2-3] \arrow["\cap"{description}, draw=none, from=1-4, to=2-4] \arrow["{\widetilde{\Phi}_{S^1}}"', "\simeq", from=1-5, to=1-4] \arrow["\simeq", dotted, no head, from=2-1, to=2-2] \arrow[from=2-1, to=3-1] \arrow[equals, from=2-2, to=1-3] \arrow["\subset"{description}, draw=none, from=2-2, to=2-3] \arrow[from=2-2, to=3-2] \arrow[from=2-3, to=2-4] \arrow[from=2-3, to=3-3] \arrow[from=2-4, to=3-4] \arrow["{\col(-)}", "\simeq"', from=3-1, to=3-2] \arrow["\subset"{description}, draw=none, from=3-2, to=3-3] \arrow[from=3-3, to=3-4]
\end{tikzcd}\]
    
    The equivalence $(\Cyc^{\lhd\rhd})^\op \xrightarrow{\simeq}\Ccal$ is a restriction of the equivalence $\widetilde{\Phi}_{S^1}\colon (\boxMfd_{2,S^1})^\op \simeq (\Bord_2)_{/S^1}$. 
    Since $\widetilde{\Phi}_{S^1}(\cyl_\rhd)$ is the initial object in $\Ccal$ by \cref{cor: conedcyc=lambda_infty}, the projection $\Ccal_{\widetilde{\Phi}_{S^1}(\cyl_\rhd)/}\rightarrow \Ccal$ is an equivalence. 
    We now claim that $\Dcal_{/\col(S^1)}$ and $\Ccal_{\widetilde{\Phi}_{S^1}(\cyl_\rhd)/}$ are equal as full subcategories of the double-slice category $((\Bord_2^\partial)_{\emptyset/})_{/\col(S^1)}\simeq((\Bord_2^\partial)_{/S^1})_{\col(S^1)/}$,
    where objects in this double-slice are factorizations of the morphism $\col(S^1) \colon \emptyset \to S^1$.
    The claim follows by inspecting the following table. (Note that in the right column the functor $\widetilde{\Phi}_{S^1}$ is applied to morphisms in the \emph{opposite} of $\Cyc^{\lhd\rhd}$.)
    \begin{center}
    {\renewcommand{\arraystretch}{1.3}
    \begin{tabular}{ c | c | c } 
      $\Dcal_{/\col(S^1)}$ & factorization & $\Ccal_{\widetilde{\Phi}_{S^1}(\cyl_\rhd)/}$ 
      \\
      \hline
      $\col(\emptyset) \to \col(S^1)$ & $\emptyset = \emptyset \to S^1$ & $\widetilde{\Phi}_{S^1}(\cyl_\rhd \to \cyl_\rhd)$ \\
      $\col(\sqcup_n D^1) \to \col(S^1)$ & $\emptyset \to \sqcup_n D^1 \to S^1$ & $\widetilde{\Phi}_{S^1}(\cyl_\rhd \to C_n)$ \\
      $\col(S^1) \to \col(S^1)$ & $\emptyset \to S^1 = S^1$ & $\widetilde{\Phi}_{S^1}(\cyl_\rhd \to \cyl_\lhd)$ 
    \end{tabular}}
    \end{center}
    Combining the equivalences in the diagram  completes the proof.
\end{proof}

%% file: figures/C4.pdf_tex
\begingroup%
  \makeatletter%
  \providecommand\color[2][]{%
    \errmessage{(Inkscape) Color is used for the text in Inkscape, but the package 'color.sty' is not loaded}%
    \renewcommand\color[2][]{}%
  }%
  \providecommand\transparent[1]{%
    \errmessage{(Inkscape) Transparency is used (non-zero) for the text in Inkscape, but the package 'transparent.sty' is not loaded}%
    \renewcommand\transparent[1]{}%
  }%
  \providecommand\rotatebox[2]{#2}%
  \newcommand*\fsize{\dimexpr\f@size pt\relax}%
  \newcommand*\lineheight[1]{\fontsize{\fsize}{#1\fsize}\selectfont}%
  \ifx\svgwidth\undefined%
    \setlength{\unitlength}{533.95258866bp}%
    \ifx\svgscale\undefined%
      \relax%
    \else%
      \setlength{\unitlength}{\unitlength * \real{\svgscale}}%
    \fi%
  \else%
    \setlength{\unitlength}{\svgwidth}%
  \fi%
  \global\let\svgwidth\undefined%
  \global\let\svgscale\undefined%
  \makeatother%
  \begin{picture}(1,0.22536881)%
    \lineheight{1}%
    \setlength\tabcolsep{0pt}%
    \put(0,0){\includegraphics[width=\unitlength,page=1]{C4.pdf}}%
    \put(0.71677008,0.16413794){\color[rgb]{0,0,0}\makebox(0,0)[lt]{\lineheight{1.25}\smash{\begin{tabular}[t]{l}$x_3$\end{tabular}}}}%
    \put(0.52398618,0.154569){\color[rgb]{0,0,0}\makebox(0,0)[lt]{\lineheight{1.25}\smash{\begin{tabular}[t]{l}$x_4$\end{tabular}}}}%
    \put(0.71677008,0.05796199){\color[rgb]{0,0,0}\makebox(0,0)[lt]{\lineheight{1.25}\smash{\begin{tabular}[t]{l}$x_2$\end{tabular}}}}%
    \put(0.52609315,0.05796199){\color[rgb]{0,0,0}\makebox(0,0)[lt]{\lineheight{1.25}\smash{\begin{tabular}[t]{l}$x_1$\end{tabular}}}}%
    \put(0.95162503,0.10280138){\color[rgb]{0,0,0}\makebox(0,0)[lt]{\lineheight{1.25}\smash{\begin{tabular}[t]{l}$s_0$\end{tabular}}}}%
  \end{picture}%
\endgroup%

%% file: figures/beta.pdf_tex
\begingroup%
  \makeatletter%
  \providecommand\color[2][]{%
    \errmessage{(Inkscape) Color is used for the text in Inkscape, but the package 'color.sty' is not loaded}%
    \renewcommand\color[2][]{}%
  }%
  \providecommand\transparent[1]{%
    \errmessage{(Inkscape) Transparency is used (non-zero) for the text in Inkscape, but the package 'transparent.sty' is not loaded}%
    \renewcommand\transparent[1]{}%
  }%
  \providecommand\rotatebox[2]{#2}%
  \newcommand*\fsize{\dimexpr\f@size pt\relax}%
  \newcommand*\lineheight[1]{\fontsize{\fsize}{#1\fsize}\selectfont}%
  \ifx\svgwidth\undefined%
    \setlength{\unitlength}{308.78333865bp}%
    \ifx\svgscale\undefined%
      \relax%
    \else%
      \setlength{\unitlength}{\unitlength * \real{\svgscale}}%
    \fi%
  \else%
    \setlength{\unitlength}{\svgwidth}%
  \fi%
  \global\let\svgwidth\undefined%
  \global\let\svgscale\undefined%
  \makeatother%
  \begin{picture}(1,0.3895604)%
    \lineheight{1}%
    \setlength\tabcolsep{0pt}%
    \put(0,0){\includegraphics[width=\unitlength,page=1]{beta.pdf}}%
    \put(0.75252143,0.29358927){\color[rgb]{0.03921569,0,0}\makebox(0,0)[lt]{\lineheight{1.25}\smash{\begin{tabular}[t]{l}$\beta_4$\end{tabular}}}}%
    \put(0.90836652,0.1744316){\color[rgb]{0.03921569,0,0}\makebox(0,0)[lt]{\lineheight{1.25}\smash{\begin{tabular}[t]{l}$s_0$\end{tabular}}}}%
    \put(0,0){\includegraphics[width=\unitlength,page=2]{beta.pdf}}%
    \put(0.42915542,0.2955523){\color[rgb]{0.03921569,0,0}\makebox(0,0)[lt]{\lineheight{1.25}\smash{\begin{tabular}[t]{l}$\beta_3$\end{tabular}}}}%
    \put(0.44754757,0.03157187){\color[rgb]{0.03921569,0,0}\makebox(0,0)[lt]{\lineheight{1.25}\smash{\begin{tabular}[t]{l}$\beta_1$\end{tabular}}}}%
    \put(0.39309315,0.1305804){\color[rgb]{0.03921569,0,0}\makebox(0,0)[lt]{\lineheight{1.25}\smash{\begin{tabular}[t]{l}$\beta_2$\end{tabular}}}}%
  \end{picture}%
\endgroup%

%% file: figures/cone-and-cocone.pdf_tex
\begingroup%
  \makeatletter%
  \providecommand\color[2][]{%
    \errmessage{(Inkscape) Color is used for the text in Inkscape, but the package 'color.sty' is not loaded}%
    \renewcommand\color[2][]{}%
  }%
  \providecommand\transparent[1]{%
    \errmessage{(Inkscape) Transparency is used (non-zero) for the text in Inkscape, but the package 'transparent.sty' is not loaded}%
    \renewcommand\transparent[1]{}%
  }%
  \providecommand\rotatebox[2]{#2}%
  \newcommand*\fsize{\dimexpr\f@size pt\relax}%
  \newcommand*\lineheight[1]{\fontsize{\fsize}{#1\fsize}\selectfont}%
  \ifx\svgwidth\undefined%
    \setlength{\unitlength}{976.41069067bp}%
    \ifx\svgscale\undefined%
      \relax%
    \else%
      \setlength{\unitlength}{\unitlength * \real{\svgscale}}%
    \fi%
  \else%
    \setlength{\unitlength}{\svgwidth}%
  \fi%
  \global\let\svgwidth\undefined%
  \global\let\svgscale\undefined%
  \makeatother%
  \begin{picture}(1,0.36723769)%
    \lineheight{1}%
    \setlength\tabcolsep{0pt}%
    \put(0,0){\includegraphics[width=\unitlength,page=1]{cone-and-cocone.pdf}}%
    \put(-0.00063496,0.19884985){\color[rgb]{0,0,0}\makebox(0,0)[lt]{\lineheight{1.25}\smash{\begin{tabular}[t]{l}$\partial_-$\end{tabular}}}}%
    \put(0.25669318,0.19955366){\color[rgb]{1,0,1}\makebox(0,0)[lt]{\lineheight{1.25}\smash{\begin{tabular}[t]{l}$\partial_+$\end{tabular}}}}%
    \put(0.94364059,0.19955366){\color[rgb]{1,0,1}\makebox(0,0)[lt]{\lineheight{1.25}\smash{\begin{tabular}[t]{l}$\partial_+$\end{tabular}}}}%
    \put(0.660576,0.19955366){\color[rgb]{0.78823529,0.11372549,0.11372549}\makebox(0,0)[lt]{\lineheight{1.25}\smash{\begin{tabular}[t]{l}$\pf$\end{tabular}}}}%
    \put(0,0){\includegraphics[width=\unitlength,page=2]{cone-and-cocone.pdf}}%
    \put(0.12778919,0.10157555){\color[rgb]{0,0,0}\makebox(0,0)[lt]{\lineheight{1.25}\smash{\begin{tabular}[t]{l}$\cyl_\lhd$\end{tabular}}}}%
    \put(0.79965245,0.09841275){\color[rgb]{0,0,0}\makebox(0,0)[lt]{\lineheight{1.25}\smash{\begin{tabular}[t]{l}$\cyl_\rhd$\end{tabular}}}}%
    \put(0.47405891,0.2030471){\color[rgb]{0,0,0}\makebox(0,0)[lt]{\lineheight{1.25}\smash{\begin{tabular}[t]{l}$C_3$\end{tabular}}}}%
    \put(0.47405891,0.00157718){\color[rgb]{0,0,0}\makebox(0,0)[lt]{\lineheight{1.25}\smash{\begin{tabular}[t]{l}$C_4$\end{tabular}}}}%
  \end{picture}%
\endgroup%

%% file: sections/arcs.tex
\section{Proposition B: arc complexes}\label{sec: proofB}

The goal of this section is to prove \cref{proposition B}, which essentially says that
\[
    \int_{M} \Yo_\Ocal(D^1) \simeq \Yo_\Ocal(M) \in \PSh(\Ocal)
\]
for all $M\in\OC$.

\subsection{Reduction to \texorpdfstring{$M=S^1$}{M=S1}}
To prove \cref{proposition B}, we will work one circle at a time, as this allows us to rewrite the factorization homology over $S^1$ as a colimit over $\Dop$ of a cyclic bar construction.
\begin{prop}\label{proposition B'}
    For all $M \in \Mfd_1^\partial$ the map 
    \[ 
        \colim_{D \in \Disk_{/S^1}}\Yo_\Ocal(D \sqcup M) \too \Yo_\Ocal(S^1\sqcup M)
    \]
    in $\Psh(\Ocal)$ is an equivalence. 
\end{prop}
Assuming \cref{proposition B'}, the proof of \cref{proposition B} is straightforward:
\begin{proof}[Proof of \cref{proposition B}]
    The factorization homology $\int_{M} \Yo_\Ocal(D^1)$ is the colimit of the composite \[
    \psi_M\colon \Disk_{/M}\to\Mfd_1\xtoo{\col(-)}\OC\xtoo{\Yo_\Ocal}\Psh(\Ocal).
    \]
    Suppose that $M\in\OC$ is a disjoint union of $k$ circles and $l$ disks. 
    By \cref{cor:Disk1tensordisjunct} we can rewrite  $\Disk_{/M}\simeq (\Disk_{/S^1})^k\times(\Disk_{/D^1})^l$. 
    Since $(\Disk_1)_{/D^1}$ has a terminal object, the functor
    \[
    (\Disk_{/S^1})^k\simeq(\Disk_{/S^1})^k\times *^l\to(\Disk_{/S^1})^k\times(\Disk_{/D^1})^l \simeq \Disk_{/M}
    \] 
    sending each $*$ to the terminal object $\id_{D^1} \in \Disk_{/D^1}$ is final.
    Writing $N = \sqcup_l D^1$, it hence suffices to show that the map  
     \begin{equation}\label{eq:yobar}
         \colim_{D_1 \in \Disk_{/S^1}}\dots
         \colim_{D_n \in \Disk_{/S^1}} \Yo_\Ocal(D_1\sqcup \cdots \sqcup D_n \sqcup N) \too
         \Yo_\Ocal((\sqcup_k S^1)\sqcup N)
     \end{equation}
    in $\Psh(\Ocal)$ is an equivalence.
    This follows by inductively applying \cref{proposition B'}.
\end{proof}

In the rest of the section, we provide a proof of \cref{proposition B'}. First we reduce it to a statement about contractibility of certain colimits.
\begin{lem}\label{lem: propBreduction}
    To prove \cref{proposition B'} it suffices to show that 
    \[
        \colim_{C_n \in \Cyc^\op} \Map_{\boxMfd_{2,S^1}}(C_n, W)   
    \]
    is contractible for all $W \in \boxMfd_{2,S^1}$ such that every component of $W$ has free boundary.
\end{lem}
\begin{proof}
    Unravelling the definition of the Yoneda embedding, 
    \cref{proposition B'} says that for every $N \in \calO$ the map 
    \begin{equation}\label{eq:propB}
        \colim_{(D,D\xto{i}S^1)\in\Disk_{/S^1}} \Map_{\OC}(N, D \sqcup M)  
        \xtoo{(i\sqcup \id_M)\circ(-)} \Map_{\OC}(N, S^1 \sqcup M)
    \end{equation}
    is an equivalence.
    This is equivalent to saying that, for all $W \in \Map_{\OC}(N, S^1 \sqcup M)$, the fiber at $W$ is contractible.
    (Note that for $W$ to be a morphism in $\OC$ every connected component has to have free boundary or incoming boundary, but as every incoming boundary component is a disk, we get that every component of $W$ has free boundary.)
    This fiber at $W$ is given by
    \[
       \colim_{(D,D\xto{i} S^1)\in\Disk_{/S^1}} \Map_{\OC_{/S^1 \sqcup M}}\big((N, N\xto{W}S^1\sqcup M), (D\sqcup M,\col'(D\xto{i} S^1) \sqcup M \times [0,1])\big).
    \]
    Note that by \cref{prop: sliceunderff} we can equivalently take this mapping space in the bigger slice category $(\Bord_2^\partial)_{/S^1 \sqcup M}$, which is equivalent to $(\Mfd_{2,S^1 \sqcup M}^\partial)^\op$ by \cref{cor:Phi-and-Psi}. 
    This is because $C_n$ embeds in the path component $W_0$ of $W$ containing the fixed circle boundary so its complement bordism has nonempty free boundary; if there are more than one path component then  $W\in\OC$ forces $W \setminus W_0$ to still be in $\OC$ after further deleting a collar of the fixed boundary $M$.  
    Applying \cref{lem: comparediagram} to identify the colimit diagrams, we can hence rewrite this as
    \[
       \colim_{C_n \in \Cyc^\op} \Map_{\boxMfd_{2,S^1 \sqcup M}}(C_n \sqcup M \times [0,1], W) .
    \]
    This still differs from the diagram in the claim by $M \times [0,1]$, but forgetting this part of the embedding defines a map
    \[
        \Map_{\boxMfd_{2,S^1 \sqcup M}}(C_n \sqcup M \times [0,1], W)
        \too
        \Map_{\boxMfd_{2,S^1}}(C_n, W)
    \]
    that is natural in $C_n$, is a Serre fibration, and has contractible fibers as its fiber is the space of collars of $W \setminus i(C_n)$ \cite[3.4.2, Corollaire 1; 5.2.1, Corollaire 1]{cerf1961topologie}.
\end{proof}

It follows from \Cref{lem: map=arc systems} and \Cref{lem: Cyc=Lambda infty} that $\colim_{\Cyc^\op} \Map_{\boxMfd_{2,S^1}}(C_n, W)$ is a colimit indexed by  $\Cyc^\op\simeq\Lambda_\infty^\op$ of arc systems $\Arc_n(W)$. We will prove the contractibility of this colimit using the fact that certain arc complexes of surfaces are contractible.

\subsection{Comparison with arc complexes}
The goal of this subsection is to show that \cref{lem: propBreduction} is equivalent to the contractibility of a certain arc complex, which we define below.
\begin{defn}
    The \hldef{arc complex of $W$ rel.~$P$} to be the simplicial set with $\Arc(W;P)$ with $k$-simplices $\Arc_{k+1}(W;P)$ (\Cref{defn:arc}).
    The $i$th face map $\Arc_{k+1}(W;P)\rightarrow\Arc_{k}(W;P)$ is given by forgetting the $i$th arc in the system, and the $i$th degeneracy map is given by doubling the $i$th arc.  
\end{defn}

There is a more general definition of the arc complex, and it is contractible under mild conditions by \cite[Lemma 2.5]{wahl08} as a variation of \cite{Hatcher1991}. We state a version that is specialized to our situation, which can be found in \cite[Lemma 7.1]{hatcherwahl}.
\begin{thm}\label{thm: arccplxcontract}
    Let $W$ be a connected surface with a boundary circle $S^1 \subset \partial W$, $P \subset W \setminus S^1$ a finite \emph{non-empty} subset.
    Then the geometric realization of the arc complex $\Arc_\bullet(W; P)$ is contractible. 
\end{thm}
\begin{proof}[Note on the theorem.]
    To be precise, Hatcher and Wahl work with a simplicial complex whose vertices are $\Arc_1(W;P)$, i.e.~isotopy classes of arcs from the base-point to $P$, and where $k+1$ distinct isotopy classes of arcs form a $k$-simplex if and only if they can be made mutually disjoint.
    By \Cref{rem: orderedarcs} there is a canonical ordering on the vertices of such a $k$-simplex, and the $(k+1)$-tuples $[\gamma_1, \dots, \gamma_{k+1}] \in \Arc_{k+1}(W;P)$ are by definition ordered with respect to this ordering.
    Therefore the geometric realization of our simplicial set is homeomorphic to the geometric realization of their simplicial complex.
\end{proof}

We will now use this theorem to show that the arc complex with endpoints in the free boundary is also contractible whenever it is non-empty.

\begin{prop}\label{prop:arc-cpx-comparison}
    Let $W \in \boxMfd_{2,S^1}$ be such that the connected component of $W$ that contains $S^1$ has non-empty free boundary.
    Then the geometric realization of $\Arc(W; \pf)$ is contractible.
\end{prop}
\begin{proof}
    Without loss of generality, we may assume that $W$ is connected as paths can only lie in the connected component of $S^1$.
    Let $W' \in \boxMfd_{2,S^1}$ be the surface obtained from $W$ by gluing a $2$-disk to each circle in $\pf$.
    Let $y_i$ be the midpoints of these $2$-disks and $x_j$ the midpoints of the $1$-disks in $\pf W' \subset \pf W$.
    We let $P = \{x_1,\dots\} \cup \{y_1,\dots\} \subset W'$ be the resulting finite subset of $W'$.
    We define maps
    \[
        f\colon \Arc_n(W; \pf) \too \Arc_n(W'; P)
        \qquad\text{and}\qquad
        g\colon \Arc_n(W'; P) \too \Arc_n(W; \pf)
    \]
    as follows.
    For $f$ we take an arc system $(\gamma_1,\dots,\gamma_n)$ in $W$ rel.~$\pf$, isotop the arcs that end in a disk so that their endpoint is the midpoint of the disk, and extend the arcs that end in a circle by an arc to the midpoint of the newly glued in disk.
    For $g$ we take an arc system in $W'$ rel.~$P$ and remove a small disk around each point in $P$ that is in the interior of $W'$ to obtain a surface that we can identify with $W$.
    This is illustrated in \cref{fig:bijection}.
    These constructions are well-defined on isotopy classes and mutually inverse.

    Moreover, the maps $f$ and $g$ are compatible with forgetting or duplicating the $i$th arc, so they define isomorphisms of simplicial sets.
    We know that $P$ is non-empty because it is in bijection with $\pf W$, which we assumed to be non-empty.
    The simplicial set $\Arc(W'; P)$ thus has a contractible realization by \cref{thm: arccplxcontract}, and hence so does the isomorphic simplicial set $\Arc(W;\pf)$.
\end{proof}

  \begin{figure}[h]
\centering
\def\svgwidth{\linewidth}
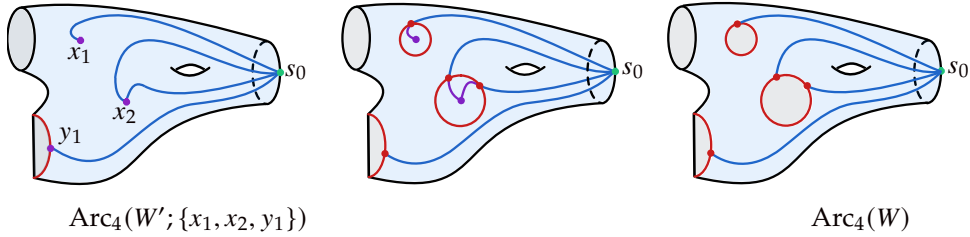
\caption{Bijection in the proof of \cref{prop:arc-cpx-comparison}.}
\label{fig:bijection}
\end{figure}

The geometric realization of a simplicial set can be computed as its colimit in the \category{} $\Scal$. 
Our next task is to show that we can rewrite $\colim_{\Cyc^\op} \Map_{\boxMfd_{2,S^1}}(C_n, W)$ of \Cref{lem: propBreduction} as a colimit indexed over a full subcategory of $\Cyc$ that is equivalent to $\Delta$, on which it agrees with the simplicial set $\Arc(W; \pf)$.
First we recall some facts about final subcategories of $\Lambda_\infty$, which will help us identify their counterparts inside $\Cyc$.

\begin{thm}[{\cite[Theorem B.3]{NS2018}}]\label{thm:Ns2018cofinal}
    Let $j_\infty\colon \Delta\rightarrow\Lambda_\infty$ be the functor sending $[n-1]=\{0,1,\ldots, n-1\}$ to $\frac{1}{n}\mathbb{Z}\cong \mathbb{Z}\times [n-1]$ equipped with lexicographical ordering and $\mathbb{Z}$-action given by addition on the first factor. 
    This functor is final.
\end{thm}


We also record the following well-known consequence, which can be deduced from \cref{cor:Disk1tensordisjunct}, \cref{lem: comparediagram}, and \cref{thm:Ns2018cofinal}.
\begin{cor}\label{cor:DiskM-weakly-contractible}
    For every $M \in \Mfd_1^\partial$ the category $\Disk_{/M}$ is weakly contractible.
\end{cor}


Under the equivalence $\Cyc \simeq \Lambda_\infty$ from \cref{lem: Cyc=Lambda infty}, the wide subcategory $\simp \subset \Lambda_\infty$ corresponds to a certain wide subcategory $\hldef{\Cyc_{\simp}} \subset \Cyc$.
Let us give a geometric description of the path components of the space of morphisms from $C_m$ to $C_n$ in $\Cyc_\Delta$.
Under the identification $\Cyc_{\Delta}\simeq \Delta$ as restriction of \Cref{lem: Cyc=Lambda infty} to full subcategories, 
The subspace $\Map_{\Cyc_\Delta}(C_m, C_n) \subset \Emb_{S^1}^\square(C_m, C_n)$ consists of those path components corresponding to systems of $m$ arcs in $C_n$ that do not wrap around $C_n$.
In terms of the standard arcs this means that we require that $\iota\colon C_m \hookrightarrow C_n$ sends each standard arc $[\beta_i]$ in $C_m$ to a standard arc $[\iota \circ \beta_i] = [\beta_{\lambda(i)}]$.
This defines a linear order preserving map $\lambda\colon \{1,\dots,n\} \to \{1,\dots,m\}$.
%

\begin{lem}\label{lem:Arc-commutative-diagram}
For every $W \in \boxMfd_{2,S^1}$ there are $2$-cells making the diagram
    \[\begin{tikzcd}
        {\Cyc_{\simp}^\op} \rar \dar & 
        {\Cyc^\op} \ar[rr,"{\Emb_{S^1}^\square(-,W)}"] && 
        \Scal \dar[equal] \\
        {\Dop} \ar[rr,"{\Arc(W; \pf)}"] &&
        \Sets \rar[hook] & \Scal
    \end{tikzcd}\]
    commute.
\end{lem}
\begin{proof}
    We already know from \cref{lem: map=arc systems} that evaluation at the standard arcs defines for all $n$ a homotopy equivalence 
    \[
        (\ev_{\beta_1},\dots, \ev_{\beta_n})\colon 
        \Emb_{S^1}^\square(C_n, W) \xtoo{\simeq} \Arc_n(W; \pf).
    \]
    The claim of the lemma is that this map is natural in $C_n \in \Cyc_{\simp}^\op$.
    Since both functors $\Cyc_{\simp}^\op \to \Scal$ land in the full subcategory $\Sets \subset \Scal$ of discrete spaces, it will suffice to prove that the bijection
    \[
        (\ev_{\beta_1},\dots, \ev_{\beta_n})\colon 
        \pi_0\Emb_{S^1}^\square(C_n, W) \xtoo{\cong} \Arc_n(W; \pf)
    \]
    is a natural isomorphism of $1$-functors $h(\Cyc_{\simp}^\op) \to \Sets$.
    That is, we need to show that for every $\iota\colon C_n \to C_m$ in $\Cyc_{\simp}$ that induces $\lambda\colon \{1,\dots,n\} \cong [n-1] \to \{1,\dots,m\} \cong [m-1]$ in $\simp$, the square of sets
    \[\begin{tikzcd}
        {\Emb_{S^1}^\square(C_m, W)} \ar[rr,"{(\ev_{\beta_1},\dots, \ev_{\beta_n})}"] \dar["\iota^*"'] &&
        {\Arc_n(W;\pf)} \dar["\lambda^*"] \\
        {\Emb_{S^1}^\square(C_m, W)} 
        \ar[rr, "{(\ev_{\beta_1},\dots, \ev_{\beta_m})}"]  &&
        {\Arc_m(W;\pf)}
    \end{tikzcd}\]
    commutes.
    This is indeed true by how we obtained $\lambda$ from $\iota$, namely it is defined so that there are isotopies $\iota \circ \beta_i \sim \beta_{\lambda(i)}$.
\end{proof}

Now we have all the ingredients to prove \cref{proposition B'}.
\begin{proof}[Proof of \cref{proposition B'}]\label{prop:map-into-S1}
    By \Cref{lem: propBreduction}, it suffices to show that
    $\colim_{\Cyc^\op} \Map_{\boxMfd_{2,S^1}}(C_n, W)$
    is contractible for all $W\in\boxMfd_{2,S^1}$ such that every component of $W$ has non-empty free boundary.
    It follows from \cref{thm:Ns2018cofinal} and \cref{lem:Arc-commutative-diagram} that we can rewrite this colimit as
    \[
      \colim_{\Cyc^\op} \Map_{\boxMfd_{2,S^1}}(C_n, W)
      \simeq \colim_{\Cyc_{\simp}^\op} \Map_{\boxMfd_{2,S^1}}(C_n, W)
      \simeq |\Arc(W;\pf)|,
    \]
    which is indeed contractible by \cref{prop:arc-cpx-comparison} as $\pf W$ intersects every component of $W$.
\end{proof}

%% file: figures/bijection.pdf_tex
\begingroup%
  \makeatletter%
  \providecommand\color[2][]{%
    \errmessage{(Inkscape) Color is used for the text in Inkscape, but the package 'color.sty' is not loaded}%
    \renewcommand\color[2][]{}%
  }%
  \providecommand\transparent[1]{%
    \errmessage{(Inkscape) Transparency is used (non-zero) for the text in Inkscape, but the package 'transparent.sty' is not loaded}%
    \renewcommand\transparent[1]{}%
  }%
  \providecommand\rotatebox[2]{#2}%
  \newcommand*\fsize{\dimexpr\f@size pt\relax}%
  \newcommand*\lineheight[1]{\fontsize{\fsize}{#1\fsize}\selectfont}%
  \ifx\svgwidth\undefined%
    \setlength{\unitlength}{1756.03825126bp}%
    \ifx\svgscale\undefined%
      \relax%
    \else%
      \setlength{\unitlength}{\unitlength * \real{\svgscale}}%
    \fi%
  \else%
    \setlength{\unitlength}{\svgwidth}%
  \fi%
  \global\let\svgwidth\undefined%
  \global\let\svgscale\undefined%
  \makeatother%
  \begin{picture}(1,0.23589058)%
    \lineheight{1}%
    \setlength\tabcolsep{0pt}%
    \put(0,0){\includegraphics[width=\unitlength,page=1]{bijection.pdf}}%
    \put(0.06477825,0.1762349){\color[rgb]{0.03921569,0,0}\makebox(0,0)[lt]{\lineheight{1.25}\smash{\begin{tabular}[t]{l}$x_1$\end{tabular}}}}%
    \put(0.11164757,0.11397663){\color[rgb]{0.03921569,0,0}\makebox(0,0)[lt]{\lineheight{1.25}\smash{\begin{tabular}[t]{l}$x_2$\end{tabular}}}}%
    \put(0.05405855,0.09263099){\color[rgb]{0.03921569,0,0}\makebox(0,0)[lt]{\lineheight{1.25}\smash{\begin{tabular}[t]{l}$y_1$\end{tabular}}}}%
    \put(0.2895557,0.15905504){\color[rgb]{0.03921569,0,0}\makebox(0,0)[lt]{\lineheight{1.25}\smash{\begin{tabular}[t]{l}$s_0$\end{tabular}}}}%
    \put(0.63984056,0.16068581){\color[rgb]{0.03921569,0,0}\makebox(0,0)[lt]{\lineheight{1.25}\smash{\begin{tabular}[t]{l}$s_0$\end{tabular}}}}%
    \put(0.97851611,0.16130704){\color[rgb]{0.03921569,0,0}\makebox(0,0)[lt]{\lineheight{1.25}\smash{\begin{tabular}[t]{l}$s_0$\end{tabular}}}}%
    \put(0.06501729,0.00162385){\color[rgb]{0.03921569,0,0}\makebox(0,0)[lt]{\lineheight{1.25}\smash{\begin{tabular}[t]{l}$\mathrm{Arc}_4(W';\{x_1,x_2,y_1\})$\end{tabular}}}}%
    \put(0.83715965,0.00162385){\color[rgb]{0.03921569,0,0}\makebox(0,0)[lt]{\lineheight{1.25}\smash{\begin{tabular}[t]{l}$\mathrm{Arc}_4(W)$\end{tabular}}}}%
    \put(0,0){\includegraphics[width=\unitlength,page=2]{bijection.pdf}}%
  \end{picture}%
\endgroup%

%% file: sections/main-theorem.tex
\section{Proof of the main theorems}
Now we are ready to prove the main results of this paper (\cref{thm:main} and \cref{thm:O-dense-in-OC-intro}) from \cref{proposition A} and \cref{proposition B}.
\subsection{Density of \texorpdfstring{$\Ocal$}{O} in \texorpdfstring{$\OC$}{OC}}
We begin by proving \cref{thm:O-dense-in-OC-intro} from the introduction, which says that $\Ocal$ is dense in $\OC$, i.e.~that the restricted Yoneda embedding
    \(
        \Yo_\Ocal\colon \OC \hookrightarrow \Psh(\Ocal) 
    \)
    is fully faithful.
\begin{proof}[Proof of \cref{thm:O-dense-in-OC-intro}]
    We would like to show that for all $M, N \in \OC$ the top map in the following map square is an equivalence
\[\begin{tikzcd}[column sep=large]
	{\Map_{\OC}(M, N)} & {\Map_{\PSh(\Ocal)}(\Yo_\Ocal(M), \Yo_\Ocal(N))} \\
	{\lim_{D \in \Disk_{/M}} \Map_{\OC}(D, N)} & {\lim_{D \in \Disk_{/M}} \Map_{\PSh(\Ocal)}(\Yo_\Ocal(D), \Yo_\Ocal(N)).}
	\arrow[from=1-1, to=1-2]
	\arrow["\simeq", from=1-1, to=2-1]
	\arrow["\simeq"', from=1-2, to=2-2]
	\arrow["\simeq", from=2-1, to=2-2]
\end{tikzcd}\]
    Here the horizontal maps are given by applying the restricted Yoneda embedding $\Yo_\Ocal$ and the vertical maps are induced by the diagram $\col'\colon \Disk_{/M} \to \OC_{/M}$ as in \cref{lem: comparediagram}.
The left vertical arrow is an equivalence by \cref{proposition A} and the right vertical arrow by \cref{proposition B}. 
The bottom arrow is an equivalence by the definition of the restricted Yoneda embedding $\Yo_\Ocal$ and the Yoneda lemma for $\Ocal$.
Therefore, the top arrow is an equivalence as desired.
\end{proof}

\subsection{The slice of \texorpdfstring{$\Ocal$}{O} over \texorpdfstring{$\OC$}{OC}}
Next we prove a consequence of \cref{proposition B}, which describes the slices one encounters when left Kan extending from $\Ocal$ to $\OC$.
For this we need the following lemma.

\begin{lem}\label{lem:Quillen-A-variant}
    Let $\Dcal$ be an \category{} and $X\colon I \to \Dcal$ a diagram. 
    Let $X = \colim_{i \in I} \Yo(X(i)) \in \Psh(\Dcal)$ be the colimit of the representable presheaves.
    Then the induced functor $I \to \Un(X)$ is final.
\end{lem}
\begin{proof}
    We can factor $X$ as
    \[
        I \xtoo{F} \Ecal \xtoo{p} \Dcal
    \]
    where $F$ is final and $p$ is the right fibration representing the presheaf $X_!(*)$.
    (The constant presheaf $* \in \PSh(I)$ unstraightens to the right fibration $\id_I$, and its left Kan extension along $X$ is represented by taking a fibrant replacement of $X\colon I \to \Dcal$ in the contravariant model structure, where the fibrant objects are right fibrations and every trivial cofibration is final \cite[Corollary 4.1.1.11]{HTT}.)
    It will therefore suffice to argue that $X_!(*)$ is the colimit of $\Yo_{\Dcal} \circ X$ over $I$.
    First note that $\colim_{i \in I} \Yo_I(i) \simeq *$ because its value at each $j \in I$ is $\colim_{i \in I} \Map_I(j,i) \simeq |I_{j/}| \simeq *$.
    Applying the colimit preserving functor $X_!$ yields
    \[
        X_!(*) \simeq X_!(\colim_{i \in I} \Yo_I(i)) 
        \simeq \colim_{i \in I} X_!\Yo_I(i)
        \simeq \colim_{i \in I} \Yo_\Dcal(X(i))
    \]
    where the last step uses that $X_! \circ \Yo_I \simeq \Yo_{\Dcal} \circ X$ \cite[Proposition 5.2.6.3]{HTT}.
    This shows $(p\colon \Ecal \to \Dcal)$ is indeed the unstraightening of $\colim_I \Yo_\Dcal \circ X$.
\end{proof}

Using this lemma, \cref{proposition B} allows us to describe the slices of $\Ocal \hookrightarrow \OC$ as follows.

\begin{prop}\label{prop:Disk-final-in-slice}
    For all $M \in \OC$ the functor 
    \[
        \col'_{/M}\colon \Disk_{/M} \too \Ocal \times_{\OC} \OC_{/M}
    \]
    (induced by $\col'\colon \Mfd_1^\partial \to \OC$ from \cref{def:col}) is final.
\end{prop}
\begin{proof}
    The right fibration 
    \(
        \Ocal \times_{\OC} \OC_{/M} \too \Ocal
    \)
    is by definition the unstraightening of $\Yo_\Ocal(M)$.
    By \cref{proposition B} the presheaf $\Yo_\Ocal(M)$ is the colimit of $\Yo_\Ocal(D)$ over $D \in \Disk_{/M}$, so \cref{lem:Quillen-A-variant} applied to $I = \Disk_{/M}$ yields the claim.
\end{proof}

\begin{cor}\label{cor:Dopfinalinslice}
    Suppose that $M =(\sqcup_k S^1)\sqcup (\sqcup_l D^1)\in \OC$ is a disjoint union of $k$ circles with $l$ disks where $k,l\geq 0$. 
    Then each of the functors in the composite
    \[
        \Dop \xtoo{\Delta} (\Dop)^k\times *^l 
        \to \Disk_{/\sqcup_k S^1} \times \Disk_{/\sqcup_l D^1} 
        \simeq \Disk_{/(\sqcup_k S^1)\sqcup (\sqcup_l D^1)} 
        \too \Ocal \times_{\OC} \OC_{/(\sqcup_k S^1)\sqcup (\sqcup_l D^1)}
    \]
    is final.
\end{cor}
\begin{proof}
    The first arrow is final because $\Dop$ is sifted. (This is true even for $k=0$.)
    The second one is final because $\Dop \to \Lambda_\infty$ is final by \Cref{thm:Ns2018cofinal} and $\Disk_{/D^1}$ has a terminal object.
    The third one is the content of \cref{prop:Disk-final-in-slice}.
\end{proof}

This in particular allows show that $\Ocal$ and $\OC$ have equivalent classifying spaces.

\begin{cor}\label{cor:O-OC-initial}
    The inclusion $\Ocal \hookrightarrow \OC$ is initial
    and hence induces an equivalence $|\Ocal|\simeq |\OC|$. 
\end{cor}
\begin{proof}
    \cref{cor:Dopfinalinslice} implies that $|\Dop| \simeq |\Ocal \times_{\OC} \OC_{/M}$, which is thus contractible.
    By Quillen's Theorem A (or rather the opposite of \cite[Theorem 4.1.3.1]{HTT}) this implies that $\Ocal \hookrightarrow \OC$ is initial and thus a weak equivalence.
\end{proof}

\subsection{Operadic Kan extensions}
Before we proceed to prove the main theorem (\cref{thm:main}), we first recall a few facts about (operadic) left Kan extensions from \cite[\S3.1]{HA} and give a criterion for when they preserve strong monoidality.

\begin{defn}
    Let $i\colon \Acal \to \Bcal$ and $F\colon \Acal \to \Ccal$ be lax symmetric monoidal functors.
    An \hldef{operadic left Kan extension} of $F$ along $i$ is a tuple $(G, \alpha)$ of a lax symmetric monoidal functor $G\colon \Bcal \to \Ccal$ and a symmetric monoidal transformation $\alpha\colon F \to G \circ i$ such that 
    for all $c \in \Ccal$ the transformation induced by $\alpha$ exhibits $G(-) \otimes c\colon \Bcal \to \Ccal$ as the pointwise left Kan extension of $F(-) \otimes c\colon \Acal \to \Ccal$ along $i$.
\end{defn}

In particular, setting $c=\unit$ we see that (as a non-monoidal functor) $G$ is the pointwise left Kan extension of $F$ along $i$.
The existence of operadic left Kan extension is guaranteed if the relevant colimits exist and are preserved by tensoring with every object.

\begin{thm}[{\cite[\S3.1]{HA}}]\label{prop:Kan-extension-existence}
    Let $i\colon \Acal \to \Bcal$ and $F\colon \Acal \to \Ccal$ be lax symmetric monoidal functors and suppose that for all $b \in \Bcal$ the diagram
    \[
        \Acal \times_\Bcal \Bcal_{/b}
        \too \Acal
        \xtoo{F} \Ccal 
    \]
    admits a colimit and that this colimit is preserved by the functor $-\otimes c\colon \Ccal \to \Ccal$ for all $c \in \Ccal$.
    Then the category
    \[
        \Fun^{\otimes\rm lax}(\Bcal, \Ccal) \times_{\Fun^{\otimes \rm lax}(\Acal, \Ccal)} \Fun^{\otimes\rm lax}(\Acal, \Ccal)_{/F}
    \]
    has an initial object, which is an operadic left Kan extension of $F$ along $i$.
\end{thm}
\begin{proof}
    We first need to check that our definition agrees with the one in \cite[Definition 3.1.2.2]{HA}.
    In \cite{HA}, $(G,\alpha)$ is an operadic left Kan extension of $F$, if and only if for all $b$ the transformation $\alpha$ exhibits $G(b)$ as the operadic colimit of the diagram
    \[
        D_b\colon \Acal_{\act}^\otimes \times_{\Bcal_\act^\otimes} (\Bcal_\act^\otimes)_{/b} \too 
        \Acal_{\act}^\otimes \xtoo{F} \Vcal_\act^\otimes.
    \]
    Consider the full inclusion
    \[
        J\colon \Acal \times_\Bcal \Bcal_{/b} \hookrightarrow 
        \Acal_{\act}^\otimes \times_{\Bcal_\act^\otimes} (\Bcal_\act^\otimes)_{/b} 
    \]
    where $\Acal = (\Acal^\otimes)_{\langle1\rangle}$ denotes the underlying (non-monoidal) category as usual.
    This functor admits a left-adjoint defined by using the cocartesian lifts of the unique active morphism $\langle n \rangle \to \langle 1 \rangle$.
    Therefore, $J$ is a right-adjoint and in particular final.
    It follows that $G(b)$ is the operadic colimit of $D_b$ if and only if it is the operadic colimit of $D_b \circ J$ (see \cite[Remark 3.1.1.4]{HA}).
    But, as pointed out in \cite[Example 3.1.1.17]{HA}, this is exactly the case if $G(b) \otimes c$ is the colimit of $(D_b\circ J(-))\otimes c$ for all $c \in \Ccal$.
    Quantifying over all $b \in \Bcal$ this is exactly our definition of operadic left Kan extension.

    By \cite[Proposition 3.1.3.3]{HA} the operadic left Kan extension of $F$ exists, if and only if for all $b \in \Bcal$ the diagram
    \[
        D_b\colon \Acal_{\act}^\otimes \times_{\Bcal_\act^\otimes} (\Bcal_\act^\otimes)_{/b} \too 
        \Acal_{\act}^\otimes \xtoo{F} \Vcal_\act^\otimes
    \]
    admits an operadic colimit lifting the relevant map to $\Fin_*^\act$.
    As above, we can rewrite this as a colimit over $\Acal \times_\Bcal \Bcal_{/b}$, which we assumed to exist.
\end{proof}

\begin{lem}\label{prop:disjunctive-Kan-extension}
    Let $i\colon \Acal \to \Bcal$ and $F\colon \Acal \to \Ccal$ be (strong) symmetric monoidal functors and suppose moreover that for all $b_1, b_2 \in \Bcal$ the functor
    \[
        \otimes\colon (\Acal \times_\Bcal \Bcal_{/b_1}) \times (\Acal \times_\Bcal \Bcal_{/b_1})
        \too \Acal \times_\Bcal \Bcal_{/b_1\otimes b_2} 
    \]
    is final. 
    If the operadic left Kan extension $(G,\alpha)$ of $F$ along $i$ exits, the lax symmetric monoidal functor $G$ is in fact strong symmetric monoidal.
\end{lem}
\begin{proof}
    Let $b_1, b_2 \in \Bcal$ and $(a_j,f_j\colon i(a_j) \to b_j) \in \Acal \times_\Bcal \Bcal_{/b_j}$.
    Then because $\alpha\colon F \to i^* G$ is a map of lax symmetric monoidal functors, we get a commutative diagram
\[\begin{tikzcd}
	{F(a_1) \otimes F(a_2)} && {i^*G(a_1)\otimes i^*G(a_2)} && {G(b_1) \otimes G(b_2)} \\
	{F(a_1\otimes a_2)} && {i^*G(a_1\otimes a_2)} && {G(b_1\otimes b_2)}
	\arrow["{\alpha_{a_1} \otimes \alpha_{a_2}}", from=1-1, to=1-3]
	\arrow[from=1-1, to=2-1]
	\arrow["{G(f_1) \otimes G(f_2)}", from=1-3, to=1-5]
	\arrow[from=1-3, to=2-3]
	\arrow[from=1-5, to=2-5]
	\arrow["{\alpha_{a_1 \otimes a_2}}", from=2-1, to=2-3]
	\arrow["{G(f_1 \otimes f_2)}", from=2-3, to=2-5]
\end{tikzcd}\]
    functorially in the $(a_j,f_j)$.
    We now consider the composite rectangle and by taking the colimit over both $a_1$ and $a_2$ (with $f_j$ implicit) we obtain the composite rectangle in the diagram
\[\begin{tikzcd}[column sep=small]
	{\colim\limits_{a_1 \in \Acal \times_\Bcal \Bcal_{/b_1}} \colim\limits_{a_2 \in \Acal \times_\Bcal \Bcal_{/b_2}} F(a_1) \otimes  F(a_2)} &
    {\colim\limits_{a_1 \in \Acal \times_\Bcal \Bcal_{/b_1}} F(a_1) \otimes G(b_2)} & 
    {G(b_1) \otimes G(b_2)} \\
	{\colim\limits_{a_1 \in \Acal \times_\Bcal \Bcal_{/b_1}} \colim\limits_{a_2 \in \Acal \times_\Bcal \Bcal_{/b_2}} F(a_1\otimes a_2)} & {\colim\limits_{a \in \Acal \times_\Bcal \Bcal_{/b_1 \otimes b_2}} F(a)} & {G(b_1\otimes b_2)}
	\arrow["\simeq", from=1-1, to=1-2]
	\arrow[from=1-1, to=2-1]
	\arrow["\simeq", from=1-2, to=1-3]
	\arrow[from=1-3, to=2-3]
	\arrow[from=2-1, to=2-2]
	\arrow["\simeq", from=2-2, to=2-3]
\end{tikzcd}\]
    where the bottom right horizontal map is an equivalence because we know that $\alpha\colon F \to i^*G$ exhibits $G$ as the pointwise left Kan extension of $F$ along $i$.
    The two maps in the top are equivalences because $\alpha$ exhibits $G \otimes c$ as the pointwise left Kan extension of $F$ along $i$.
    (Here we use the case of $c = F(a_1)$ for the left map and $c = G(b_2)$ for the right map.)
    The left arrow is an equivalence because we assumed $F$ to be strong symmetric monoidal, and the bottom left arrow is an equivalence because we assumed that the functor 
    \[
        \otimes\colon (\Acal \times_\Bcal \Bcal_{/b_1}) \times (\Acal \times_\Bcal \Bcal_{/b_1})
        \too \Acal \times_\Bcal \Bcal_{/b_1\otimes b_2} 
    \]
    is final. 
    It follows that the right vertical map is an equivalence, which shows that $G$ is strong symmetric monoidal.
\end{proof}

\subsection{Main theorem}
Now we are ready to prove the main theorem of this paper in its utmost generality, which in particular implies \cref{thm:main} from the introduction.

\begin{thm}\label{thm:extending-from-O-to-OC}
    Let $\Vcal \in \SM$ be a symmetric monoidal \category{} and $F\colon \Ocal \too \Vcal$ a symmetric monoidal functor such that the diagram
    \[
        \Disk_{/S^1} \xtoo{\col'} \Ocal \xtoo{F} \Vcal
    \]
    admits a colimit and this colimit is preserved by $-\otimes v\colon \Vcal \to \Vcal$ for all $v \in \Vcal$.
    Then the operadic Kan extension $i_!F$ of $F$ along $i$ exists and is strong symmetric monoidal.
    Moreover, $i_!F$ is the \emph{unique} symmetric monoidal functor $i_!F\colon \OC \to \Vcal$ with $i^*i_!F \simeq F$ such that the canonical map
    \begin{equation}\label{eq: goodTFT}
        \int_{S^1} F(D^1) \too i_!F(S^1) 
    \end{equation}
    induced by the diagram in \cref{def:col} is an equivalence.
\end{thm}
\begin{proof}
    For $M, N \in \OC$ the symmetric monoidal functor $\col'\colon \Mfd_1^\partial \to \OC$ induces a commutative square
    \[\begin{tikzcd}
        {\Disk_{/M} \times \Disk_{/N}} \ar[r] \ar[d,"\sqcup"] & 
        (\Ocal \times_{\OC} \OC_{/M}) \times (\Ocal \times_{\OC} \OC_{/N})\ar[d,"\sqcup"]\\
        {\Disk_{/M \sqcup N}} \ar[r] & 
        \Ocal \times_{\OC} \OC_{/M\sqcup N}
    \end{tikzcd}\]
    where the left functor is an equivalence because $\Mfd_1^\partial$ is $\otimes$-disjunctive by \cref{cor:Disk1tensordisjunct}.
    The horizontal functors are final by \cref{prop:Disk-final-in-slice} and hence the right functor must also be final.
    Therefore, \cref{prop:disjunctive-Kan-extension} shows that the left Kan extension is strong symmetric monoidal if it exists.

    To see that the operadic Kan extension exists it suffices, by \cref{prop:Kan-extension-existence} and \cref{prop:Disk-final-in-slice}, to check that for every $M \in \OC$ the diagram
    \[
        \Disk_{/M} \xtoo{\col'} \Ocal \times_{\OC} \OC_{/M} \too \Ocal \xtoo{F} \Vcal
    \]
    admits a colimit that is moreover preserved by tensoring with any object $v \in \Vcal$.
    This is clear when $M$ is in $\Ocal$ as then $\Disk_{/M}$ has $(M, \id_M)$ as a terminal object.
    We may thus induct over the number of circles in $M$ and assume that $M = N \sqcup S^1$ and that the claim holds for $N$.
    Then the colimit can be computed as
    \[
        \colim_{D \in \Disk_{/M}} F(M) 
        \simeq \colim_{D_1 \in \Disk_{/N}} \colim_{D_2 \in \Disk_{/S^1}} F(D_1 \sqcup D_2) 
        \simeq \colim_{D_1 \in \Disk_{/N}} \colim_{D_2 \in \Disk_{/S^1}} (F(D_1) \otimes F(D_2)).
    \]
    This colimit exists and is preserved by $-\otimes v$ because $\colim_{D_2 \in \Disk_{/S^1}} F(D_2)$ exists and is preserved by $F(D_1) \otimes - \otimes v$ (by assumption of the theorem), and because $\colim_{D_1 \in \Disk_{/N}} F(D_1)$ exists and is preserved by $- \otimes \colim_{D_2 \in \Disk_{/S^1}} F(D_2) \otimes v$ (by induction hypothesis).
    
    Finally, we need to argue that $i_!F$ is the unique extension for which \cref{eq: goodTFT} is an equivalence.
    Firstly, note that \cref{eq: goodTFT} is indeed an equivalence for $i_!F$ because the map
    \[
        \colim_{D \in \Ocal \times_{\OC} \OC_{/S^1}} F(D) \too i_!F(S^1)
    \]
    is an equivalence and \cref{prop:Disk-final-in-slice} allows us to rewrite the colimit as a colimit over $\Disk_{/S^1}$, which computes the factorization homology.
    Now suppose that $G$ is some other symmetric monoidal extension of $F$, i.e.~with $F \simeq i^*G$.
    By adjunction, we get a map $\alpha\colon i_!F \to G$ and this is an equivalence when restricted to $\Ocal$.
    If $G$ satisfies that \cref{eq: goodTFT} is an equivalence, then the commutative diagram induced by $\alpha$ shows that $\alpha_{S^1}\colon (i_!F)(S^1) \to G(S^1)$ is also an equivalence.
    Because $i_!F$ and $G$ are (strong) symmetric monoidal and every object in $\OC$ is a disjoint union of disks and circles, it follows that $\alpha$ is an equivalence.
\end{proof}

Often we can just assume that all the colimits of the desired shape exist and are preserved by the tensor product.

\begin{defn}[{\cite[Definition 3.1.1.18]{HA}}]
    Let $\Vcal$ be a symmetric monoidal \category{} and $\Kcal$ a collection of \categories{}.
    We say that $\Vcal$ is \hldef{compatible with $\Kcal$-indexed colimits} if $\Vcal$ admits colimits of shape $K$ for all $K \in \Kcal$ and the functor $x\otimes -\colon \Vcal \to \Vcal$ preserves colimits of shape $K$ for all $K \in \Kcal$ and $x \in \Vcal$.
    When $\Kcal = \{\Dop\}$ we say that $\Vcal$ is \hldef{compatible with geometric realizations}.
\end{defn}

Using this notation, \cref{prop:Kan-extension-existence} gives us the following.
If $i\colon \Acal \to \Bcal$ is a lax symmetric monoidal functor, and $\Vcal$ a symmetric monoidal \category{} compatible with $\{\Acal \times_\Bcal \Bcal_{/b}\}_{b \in \Bcal}$-indexed colimits,
then there is an adjunction
\[
    i_!\colon \Fun^{\otimes,\rm lax}(\Acal, \Vcal) \adj  \Fun^{\otimes,\rm lax}(\Bcal, \Vcal) \cocolon i^*
\]
such that for every lax $F\colon \Acal \to \Vcal$ the unit map $F \to i^*i_!F$ exhibits $i_!F$ as the operadic left Kan extension of $F$.

If we further assume that $i$ is strong symmetric monoidal and satisfies the finality conditions on slices as in \cref{prop:disjunctive-Kan-extension}, then the above adjunction restricts to an adjunction 
\[
    i_!\colon \Fun^{\otimes}(\Acal, \Vcal) \adj  \Fun^{\otimes}(\Bcal, \Vcal) \cocolon i^*
\]
on \categories{} of strong symmetric monoidal functors.

\begin{cor}\label{cor:extending-O-to-OCwithgeomreal}
    Let $\Vcal \in \SM$ be a symmetric monoidal \category{} compatible with geometric realizations.
    Then there is an adjunction
    \[
        i_!\colon \Fun^\otimes(\Ocal, \Vcal) \adj \Fun^\otimes(\OC, \Vcal) \cocolon i^\ast,
    \]
    where $i_!$ is fully faithful, i.e.~every symmetric monoidal functor $F\colon \Ocal \to \Vcal$ can be extended to $\OC$ and $i_!F$ is initial among such extensions.
    Moreover, a symmetric monoidal functor $G\colon \OC \to \Vcal$ with $i^*G \simeq F$ is the left Kan extension if and only if the canonical map
    \begin{equation}\label{eq: goodTFTcor}
        \int_{S^1} F(D^1) \too G(S^1) 
    \end{equation}
    induced by the diagram in \cref{def:col} is an equivalence.
\end{cor}

In other words, the left Kan extension $i_!F$ of a symmetric monoidal functor $F\colon \calO\calC \to \Vcal$ canonically admits a symmetric monoidal structure.
The functor
    \[
        i_!\colon \Fun^\otimes(\Ocal, \Vcal) \hookrightarrow \Fun^\otimes(\OC, \Vcal)
    \]
is fully faithful, and its essential image consists of those symmetric monoidal functors for which \cref{eq: goodTFTcor} is an equivalence.
In the next section, we will see that in general $i_!$ is far from being an equivalence.

%% file: sections/examples.tex
\section{Applications and examples}\label{sec:app}
In this section, we survey some applications of \cref{thm:extending-from-O-to-OC}. The first collection of examples arises from $\Erm_\infty$-Calabi--Yau algebras, including cochains on manifolds and finite Galois extensions. Then we look at some concrete examples of $\Erm_1$-Calabi--Yau algebras in the context of topological field theories with values in vector spaces and linear categories over a field $k$.  
Finally, we explain how our results are relevant to a variant of the oriented cobordism hypothesis in dimension $2$.

\subsection{\texorpdfstring{$\Erm_\infty$}{E-infinity}-Calabi--Yau algebras}
A rich source of $\Erm_1$-Calabi--Yau algebras are $\Erm_\infty$-Calabi--Yau algebras, which we recall below.
\begin{defn}
    An \hldef{$\Erm_\infty$-Calabi--Yau algebra} in a symmetric monoidal $\infty$-category $\Ccal$ is a pair $(A, \tau)$ of an $\Erm_\infty$-algebra $A$ in $\Ccal$ and a map $\tau\colon A \to \unit$ such that the composite
    \[
        A \otimes A \xtoo{\text{multiply}} A \xtoo{\tau} \unit
    \]
    is a non-degenerate pairing exhibiting $A$ as its own dual.
    We say that $\tau$ is a non-degenerate trace for $A$.
\end{defn}

\begin{obs}\label{obs: cyctr}
    For every $\Erm_\infty$-Calabi--Yau algebra $(A,\tau)$ we can construct an $\Erm_1$-Calabi--Yau algebra structure, as noted in \cite[Remark 4.6.5.10]{HA}.
    The argument in \cite{HA} does not explicitly deal with the $\SO(2)$-invariance, so we recall the argument and briefly explain how this can be achieved.
    Indeed, the factorization homology $\int_{S^1} A \simeq \mrm{THH}(A)$ can be computed as $\colim_{S^1} A$ where this (constant) colimit is taken in the \category{} of $\Erm_\infty$-algebras in $\calC$.
    The map $S^1 \to *$ thus induces an $\SO(2)$-invariant retraction $\int_{S^1} A \to A$.
    Using this we can factor the trace $\tau$ as
    \[
        \tau\colon A \too \int_{S^1} A \xtoo{} A \xtoo{\tau} \unit
    \]
    and then the composite of the latter two morphisms gives the desired $\SO(2)$-invariant trace.
\end{obs}

In particular, we get for every $\Erm_\infty$-Calabi--Yau algebra $(A,\tau)$ an action of the surface operad on $\int_{S^1} A$.
\begin{rem}
    Since our construction of the action of the surface operad only uses the $\Erm_1$-Calabi--Yau structure, it seems likely that in the case where $A$ is $\Erm_\infty$-Calabi--Yau this action of the surface operad on $\int_{S^1} A$ factors through another operad, possibly trivializing some of the structure and adding other structure.
    It would be interesting to determine the $\Erm_\infty$-analogue of \cref{thm:extending-from-O-to-OC}.
\end{rem}

\begin{example}\label{ex:cochainsM}
    Let $R$ be an even-periodic $\Erm_\infty$-ring spectrum. Suppose that $M$ is an $R$-oriented even-dimensional closed manifold. Since the cochain algebra $C^*(M;R)$ is an $\Erm_\infty$-object in $\mathrm{Mod}_{R}$, the Poincar\'{e} duality pairing factors through $\int_{s^1}C^*(M;R)$ via \cref{obs: cyctr}, thereby giving rise to a cyclic trace and endowing $C^*(M;R)$ with the structure of an $\Erm_1$-Calabi--Yau object. \cref{thm:extending-from-O-to-OC} states that the associated open TFT canonically extends to an open-closed TFT $F'\colon \OC\rightarrow \mathrm{Mod}_{R}$ with $F'(S^1)=\int_{S^1}C^*(M;R)$.
\end{example}
In order to give a concrete description of the value at $S^1$, we consider the following variant:
\begin{example}\label{ex:LS}
   Suppose that $R$ is the Eilenberg--MacLane spectrum of a commutative ring and $M$ is an $R$-oriented even-dimensional closed manifold.
   We can periodize $R[t^{\pm1}]$ by adjoining an invertible generator in degree 2 to obtain an $\Erm_\infty$ ring spectrum. 
   Then $C^*(M;R)[t^{\pm1}] \simeq C^*(M;R[t^{\pm1}])$ is an $\Erm_\infty$-object in $\mathrm{Mod}_{R[t^{\pm1}]}$ and thus has the structure of an $\Erm_1$-Calabi--Yau object. Hence, we obtain an open-closed TFT $F'\colon \OC\rightarrow \mathrm{Mod}_{R[t^{\pm1}]}$ with $F'(S^1)=\int_{S^1}C^*(M;R)[t^{\pm1}]$.

    Suppose further that $M$ is simply-connected. Then $\mathrm{Mod}_{R[t^{\pm1}]}$ is equivalent to the derived category of chain complexes over $R[t^{\pm1}]$, and we can further identify 
    \[
      F'(S^1)=\int_{S^1}C^*(M;R)[t^{\pm1}]=C^*(\mathcal{L}M;R)[t^{\pm1}]
     \]
     as the cochain algebra on the free loop space on $M$ \cite[Proposition 5.3]{ayala2015factorization}\cite{jones1987cyclic,ungheretti2017free}.
\end{example}

\begin{example}\label{ex:H*(LM)notcenter}
Continuing with the previous example, upon taking homology, which passes to the homotopy category $\Vcal=\mathrm{Vect}_{R[t^{\pm1}]}$, the symmetric monoidal functor $H^*(F'(-);R)[t^{\pm1}]$ defines an open-closed TFT valued in $\Vcal$. But its value on the circle $H^*(\mathcal{L}M;R)[t^{\pm1}]$ is infinite-dimensional as a module over $R[t^{\pm1}]$, and thus cannot be the centre of $H^*(M;R)[t^{\pm1}]$.
\end{example}

\begin{example}
    Let $G$ be a finite group. 
    A map of $\Erm_\infty$-ring spectra $A\rightarrow B$ is a \hldef{$G$-Galois extension} of $A$ in the sense of Rognes \cite{rognes2008galois} if there is an $A$-linear $G$-action on $B$ and the canonical maps $A\to B^{hG}$ and $B\otimes_A B\to \prod_G B$ are both equivalences. 
    The dualizing sphere $\mathbb{S}^{\mathrm{ad}G}\simeq\mathbb{S}$ is trivial, and thus the canonical map $\delta\colon  B\simeq B\otimes \mathbb{S}^{\mathrm{ad}G}\to D_A(B)$ from the $A$-module $B$ to its $A$-linear dual is an equivalence of $A$-modules. 
    Furthermore, there is a trace map $\lambda\colon B\to A$ such that the pairing $B\otimes_A B\xto{\mu}B\xto{\lambda}A$ is nondegenerate and exhibits $B$ as its $A$-linear dual in the sense that it is left adjoint to $\delta$  \cite[\S 6.4]{rognes2008galois}. 
    Since $\mathrm{CAlg}_{A/}(\mathrm{Sp})\simeq \mathrm{CAlg}(\mathrm{Mod}_A)$, this equips $B$ with the structure of an $\mrm{E}_\infty$-Calabi--Yau object in $\mathrm{Mod}_A$. 
    It then follows from \cref{obs: cyctr} that this pairing factors through $\int_{S^1}B=\mathrm{THH}(B/A)$, which is the relative THH of $B$ taken inside $\mathrm{Mod}_A$, thereby endowing $B$ with the structure of an $\Erm_1$-Calabi--Yau object in $\mathrm{Mod}_A$. 
    
    The ring $\pi_0(B\otimes_A B)\cong \mathrm{Map}(G, \pi_0(B))$ contains a function $\chi_e$ that evaluates to 1 at the identity element $e\in G$ and 0 everywhere else. The map $B\otimes_A B\simeq \prod_G B\to B$ picks out the copy of $B$ indexed by $e$, so $B=(\prod_G B)[\chi_e^{-1}]$. 
    Therefore, we have
    \[
     \mathrm{THH}(B/A)=B\otimes_{B\otimes_A B}B\simeq B\otimes_{\prod_G B}B=B\otimes_{\prod_G B}(\prod_G B)[\chi_e^{-1}]\simeq B[1^{-1}]=B,
    \]
    so \cref{thm:extending-from-O-to-OC} tells us in particular that the $A$-module $B$ is also an algebra over the surface operad.
\end{example}

\subsection{Examples from topological field theory}
\begin{defn}
    Let $\Vcal \in \SM$ be a symmetric monoidal \category{}. 
    We define a (2-dimensional) \hldef{open topological field theory (open TFT)} to be a symmetric monoidal functor $\Ocal\rightarrow\Vcal$, and a (2-dimensional) \hldef{open-closed topological field theory (open-closed TFT)} to be a symmetric monoidal functor $\OC\rightarrow\Vcal$.
\end{defn}

\begin{rem}
    Note the definition of open-closed TFTs is not agreed upon across the literature. 
    For instance, some define it to be a symmetric monoidal functor out of $\Bord_2^\partial$ (for instance \cite{lauda2007state,lauda2008open}). Our notion agrees with that of Costello \cite{costello2007topological} in the context of topological conformal field theories. 
\end{rem}
\cref{thm:extending-from-O-to-OC} then says that if $\Vcal$ admits colimits indexed by $\Disk_{/S^1}$  and the tensor product preserves such colimits in each variable, then every open TFT valued in $\Vcal$ has a canonical extension to an open-closed TFT whose value at $S^1$ is the Hochschild homology of its value at $D^1$. 

One might ask if all open-closed TFTs arise this way. 
The answer is no in general. 
Below we will first see two examples where \cref{thm:extending-from-O-to-OC} applies, and then record examples in those contexts of open-closed TFTs that are not canonically extended from open TFTs.
\begin{example}[Knowledgeable Frobenius algebras]\label{ex:knowledgeableFrob}
    In \cite{lauda2008open}, Lauda and Pfeiffer provided a classification of symmetric monoidal functors $\Bord_2^\partial\rightarrow\Vcal$ where $\Vcal$ is a 1-category in terms of \hldef{knowledgeable Frobenius algebras}. 
    A knowledgeable Frobenius algebra consists of a tuple $(A,C,\iota,\iota^*)$, where $A$ is a symmetric Frobenius algebra, $C$ is a commutative Frobenius algebra, as well as an algebra morphism $\iota\colon C\rightarrow A$ and a coalgebra morphism $\iota^*\colon A\rightarrow C$ that satisfy certain compatibility conditions. 
    Under the equivalence between knowledgeable Frobenius algebras in $\Vcal$ and symmetrical monoidal functors $F\colon \Bord_2^\partial\rightarrow\Vcal$, $F$ is sent to a knowledgeable Frobenius algebra with $A=F(D^1)$, $C=F(S^1)$, and $\iota\colon  C\rightarrow A$ is given by the value of $F$ on the bordism from $S^1$ to $D^1$ that is the reverse to $C_1$ (\Cref{notn: Cn}).

    Here we are working with the restriction to the non-full subcategory $\OC\subset\Bord_2^\partial$, in which case we no longer have the bordism $\emptyset\rightarrow S^1$ given by the 2-disk that records the unit of the algebra $C=F(S^1)$. 
    However, it is true that if two unital algebras are isomorphic as non-unital algebras, then they are also isomorphic as unital algebras. On the other hand, any commutative Frobenius algebra is dualizable by definition, whereas the bordism $\emptyset\rightarrow S^1\sqcup S^1$ corresponding to the coevaluation map of $C=F(S^1)$ is not in $\OC$, so there is no requirement that $C$ needs to be dualizable.
\end{example}

\begin{example}[Pivotal $k$-linear categories]
Consider the symmetric monoidal bicategory
$\mathrm{Lex}^f$ of finite $k$-linear 1-categories over an algebraically closed field $k$,  whose 1-morphisms are left exact functors and the
2-morphisms are linear natural transformations.
In \cite{muller2024categorified}, M\"{u}ller and Woike showed that the 2-groupoid of open TFTs valued in $\mathrm{Lex}^f$ is equivalent to the $2$-groupoid of pivotal Grothendieck-Verdier categories in
$\mathrm{Lex}^f$. In particular, if an open TFT $F\colon \calO\rightarrow \mathrm{Lex}^f$ sends $D^1$ to a  pivotal finite tensor category $\mathcal{C}$ in the sense of  \cite{etingof2004finite}(this implies in particular that the monoidal product on $\calC$ is rigid), then there is a canonical extension of $F$ to an open-closed TFT $\bar{F}$ that sends $S^1$ to the Drinfeld centre $Z(\mathcal{C})$ of $\mathcal{C}$ \cite[Theorem 4.3]{muller2024categorified}. Note that the Drinfeld centre is the Hochschild cohomology of $\mathcal{C}$, and in the case where $\mathcal{C}$ is a pivotal finite tensor category this is canonically isomorphic to its $k$-linear dual, which is the Hochschild homology of $\mathcal{C}$ by \cite[Theorem 5.9]{muller2023lyubashenko}. We thus expect this canonical extension to be exactly the left Kan extension along the inclusion along $\Ocal\hookrightarrow \OC$, although this was not explicitly proved in \cite{muller2024categorified}. 
\end{example}

Our goal now is to produce examples of open-closed TFTs valued in $\Vect_k$ such that the value at $S^1$ is not the centre of the value at $D^1$ in a suitable sense. \cref{ex:H*(LM)notcenter} provides a manifold-theoretic counterexample. More counterexamples can be found in the world of knowledgeable Frobenius algebra in $\Vcal=\mrm{Vect}_k$ (\cref{ex:knowledgeableFrob}), the 1-category of vector spaces over an algebraically closed field $k$, such that $C$ is not the centre of $A$ in the classical sense.

\begin{example}
    Lauda and Pfeiffer showed in \cite{lauda2007state} that for $A$ a strongly separable algebra over $k$, there exists a commutative Frobenius algebra structure on its centre $Z(A)$ and the inclusion $\iota\colon Z(A)\rightarrow A$ makes $(A,Z(A),\iota,\iota^*)$ a knowledgeable Frobenius algebra in $\mathrm{Vect}_k$.
    However, not all knowledgeable Frobenius algebra are of this form. 
    In \cite[Example 2.19]{lauda2007state} they construct a knowledgeable Frobenius algebra $(A,C, \iota,\iota^*)$ over $k$ (under certain assumptions, which are satisfied for $k=\Cbb$) where $A$ is the algebra of $n$-by-$n$-matrices over $k$ and $C = k[X]/(X^2-1)$, which is not equivalent to the centre $Z(A) = k$.
\end{example}

We can also construct more trivial counter-examples by applying the following lemma to any closed TFT, given by some (non-zero) commutative Frobenius algebra $C$.
This yields an open-closed TFT $\bar{F}\colon \OC \subset \Bord_2^\partial \to \Vect_k$ with $\bar{F}(D^1) = 0$ but $\bar{F}(S^1) = C$, which is not $Z(0)=0$.

\begin{lem}\label{lem:extending-by-0}
    Every closed TFT $F\colon \Bord_2 \to \Vect_k$ can be extended (uniquely) to a functor $\bar{F}\colon \Bord_2^\partial \to \Vect_k$ satisfying $\bar{F}(D^1) = 0$.
\end{lem}
\begin{proof}
    Because $\bar{F}$ is symmetric monoidal we must have $\bar{F}(M) = 0$ whenever $M$ is a $1$-manifold that is not closed.
    For any bordism $W \colon M \to N$ in $\Bord_2^\partial$ such that $W$ has free boundary we can find a factorization as $M \to M \sqcup D^1 \to N$.
    Evaluating $\bar{F}$ on this factorization we find that 
    \[
        \bar{F}(W)\colon \bar{F}(M) \too \bar{F}(M\sqcup D^1) = 0 \too \bar{F}(N)
    \]
    must be the $0$-map.
    We have therefore shown that the extension $\bar{F}$ is uniquely determined on objects and morphisms as
    \begin{align*}
        \bar{F}(M) &= \begin{cases}
            F(M) & \text{ if $M$ is closed,} \\
            0 & \text{ if $M$ has boundary,}
        \end{cases} 
        \quad \text{and} \\
        \bar{F}(W\colon M \to N) &= \begin{cases}
            F(W) & \text{ if $\partial W = M \sqcup N$,} \\
            0 & \text{ if $W$ has free boundary.}
        \end{cases} 
    \end{align*}
    It remains to check that this always is a well-defined symmetric monoidal functor.
    This is indeed the case because ``manifolds with free boundary'' behaves like an ideal:
    composing a bordism that has free boundary with an arbitrary bordism results in a bordism that has free boundary, and the same holds for disjoint union.
\end{proof}

\subsection{Relation to the non-compact cobordism hypothesis}\label{subsec:noncompact}
In this section, we explain how \cref{thm:extending-from-O-to-OC} serves as input to a proof of a variant of the (oriented) cobordism hypothesis in dimension 2, following some of the ideas briefly sketched in \cite[Section 4.2]{lurie2008classification}.
The (framed) cobordism hypothesis (in dimension $n$) was first proposed in \cite{baez1995higher}. Subsequently, there have been many works on the cobordism hypothesis in various generality (in particular variations of the tangential structure), including \cite{lurie2008classification,schommer2009classification,Harpaz-1d-cobhyp,ayala2017cobordism,grady2021geometric}. 

While the variant of the cobordism hypothesis relevant to our work is not an instance of a tangential structure, to provide context we start by quickly summarizing the (oriented) cobordism hypothesis in dimension 2 following \cite{lurie2008classification}. 
Let \hldef{$\Bord^{\mrm{or}}_{012}$} be the following symmetric monoidal $(\infty,2)$-category: 
\begin{enumerate}
    \item Objects are oriented $0$-manifolds;
    \item $1$-morphisms from $A$ to $B$ are oriented $1$-bordisms from $A$ to $B$;
    \item $2$-morphisms between $1$-morphisms $M,N\colon A\to B$ are oriented $2$-bordisms with corners $W\colon M\to N$ that restricts to trivial $1$-bordisms along $A$ and $B$;
    \item Higher morphisms are given by orientation-preserving diffeomorphisms, isotopies, etc., encoding the homotopy type of moduli spaces $B\Diff_{M\sqcup N}(W)$ of $2$-morphisms.
\end{enumerate}
The symmetric monoidal product on $\Bord^{\mrm{or}}_{012}$ is given by disjoint union.

In \cite{lurie2008classification}, Lurie provided a detailed sketch of the proof of the following thesis \cite[Theorem 4.2.26]{lurie2008classification}:
For $(\Ccal,\otimes,\mathbf{1})$ a symmetric monoidal $(\infty,2)$-category, there is an equivalence of $\infty$-groupoids 
\[
\mathrm{Fun}^\otimes(\Bord^{\mrm{or}}_{012},\Ccal)\simeq ((\Ccal^{\mathrm{fd}})^\simeq)^{h\mrm{SO}(2)}.
\] 
That is, symmetric monoidal functors $\Bord^{\mrm{or}}_{012}\to\Ccal$ are classified as the homotopy fixed points of a certain $\mrm{SO}(2)$-action on the $\infty$-groupoid of fully dualisable objects in $\Ccal$.  

In practice, it is often hard to understand concretely this $\mrm{SO}(2)$-action. 
Without it, one gets a classification of symmetric monoidal functors out of the extended \emph{framed} $2$-bordism category as fully dualisable objects of $\Ccal$. 

Alternatively, Lurie proposed that this can be achieved by relaxing the condition of fully dualisability and restricting to a wide but non-full subcategory of $\Bord^{\mrm{or}}_{012}$ on the left-hand side. 
More precisely, let $\hldef{\Bord_{012}^\nc}\subset\Bord_{012}^{\rm or}$ be the subcategory with the same objects and $1$-morphisms, and that contains precisely those $2$-morphism $W$ between $M,N\colon A\to B$ such that every path component of $W$ has nonempty intersection with $M$.

\begin{defn}[{\cite[Definition 4.2.6]{lurie2008classification}}]
For $(\Ccal,\otimes,\mathbf{1})$ a symmetric monoidal $(\infty,2)$-category, a \hldef{Calabi--Yau object} is a dualisable object $X\in\Ccal$ together with a morphism $\eta\colon  \mrm{ev}_X\circ\mrm{coev}_X\to\mrm{id}$ in $\Map_\Ccal(\mathbf{1},\mathbf{1})$ that is $\mrm{SO}(2)$-equivariant (for the canonical action on $\mrm{ev}_X\circ\mrm{coev}_X$ and the trivial action on $\id$) and is the counit for an adjunction between $\mrm{ev}_X$ and $\mrm{coev}_X$.
\end{defn}
Without loss of generality, we will assume that $\Ccal$ has duals from here on. The following \hldef{non-compact cobordism hypothesis} is proposed in \cite[Theorem 4.2.11]{lurie2008classification}:
\begin{proto}\label{protothm}
Let $(\Ccal,\otimes,\mathbf{1})$ be a symmetric monoidal $(\infty,2)$-category.
Then the $(\infty,2)$-category of symmetric monoidal functors $\Bord_{012}^\nc\to\Ccal$ is equivalent to the $\infty$-groupoid of Calabi--Yau objects of $\Ccal$. The equivalence is implemented by evaluating at the object $*$.
\end{proto}

The first observation is that $\OC$ is the lax slice $\big(\Bord_{012}^\nc\big)_{\emptyset\sslash}$.
Hence, the $(\infty,2)$-category $\Bord_{012}^\nc$ is equivalent to the weak categorical chain complex of length two given by the symmetric monoidal cocartesian fibration $\pf\colon \OC\to \Bord_1^{\mrm{or}}$.
Similarly, one can unfold $\Ccal$ to a symmetric monoidal cocartesian fibration $\Ccal_{\mathbf{1}\sslash}\to\Ccal_0$, where $\Ccal_0$ is the $(\infty,1)$-category obtained by discarding the non-invertible $2$-morphisms in $\Ccal$.%
\footnote{The equivalence between $(\infty,2)$-categories with duals and weak categorical chain complex of length two is stated in \cite[Proposition 3.3.30]{lurie2008classification}, together with a sketch of a proof.
There is ongoing work of Haugseng and Nikolaus that aims to provide a complete proof of this statement.} 
Then \cref{protothm} is equivalent to a classification of symmetric monoidal functors $\OC\to \Ccal_{\mathbf{1}\sslash}$ that preserve cocartesian fibrations in terms of  Calabi--Yau objects in $\Ccal$.
A brief sketch of the proposed proof strategy to this statement can be found in \cite[96]{lurie2008classification}.
The input to the strategy are relative versions of  the space-level refinements of \cite[Theorem A]{costello2007topological}, which can be deduced from the argument for \cref{thm:extending-from-O-to-OC}.

%% file: sections/appendix.tex
\section{Slices of the bordism category}\label{subsec:Bord-pullback-proof}
The goal of this appendix is to prove the following result that we have been using to interact with the bordism \category:
\begin{thm}\label{prop:ArBord-pullback}
    There is a pullback square of symmetric monoidal \categories{}
    \[\begin{tikzcd}
        \boxMfd_d \ar[r] \ar[d,"\partial_+"'] \ar[dr, phantom, very near start, "\lrcorner"] & \Ar(\Bord_d^\partial) \ar[d, "\ev_0"] \\
        \Mfd_{d-1}^{\partial,\cong} \ar[r] & \Bord_d^\partial.
    \end{tikzcd}\]
\end{thm}

As a direct consequence we get the following description of slice categories.
\begin{cor}\label{thm:Phi-and-Psi}
    For every compact oriented $(d-1)$-manifold $M$ there are equivalences
    \[
        \hldef{\widetilde{\Psi}_S} \colon \boxMfd_{d,S} \xtoo{\simeq} (\Bord_d^\partial)_{S/}
        \qquad \text{and} \qquad
        \hldef{\widetilde{\Phi}_S} \colon (\boxMfd_{d,S})^\op \xtoo{\simeq} (\Bord_d^\partial)_{/S}
    \] 
    and for $S = \emptyset$ these equivalences are symmetric monoidal.
\end{cor}
\begin{proof}
We start with the following observation:
    As morphisms in $\boxMfd_d$ by definition restrict to diffeomorphisms on the $\partial_+$-boundary, sending $W \mapsto \partial_+ W$ is a well-defined functor into the topologically enriched groupoid $\Mfd_{d-1}^{\partial, \cong}$ of compact oriented $(d-1)$-manifolds and diffeomorphisms between them.
    We thus have a pullback square of topologically enriched categories
    \[\begin{tikzcd}
        \boxMfd_{d,M} \ar[r] \dar \ar[dr, phantom, very near start, "\lrcorner"] & 
        \boxMfd_d \ar[d,"\partial_+"]  \\
        \{M\} \rar & \Mfd_{d-1}^{\partial,\cong} .
    \end{tikzcd}\]
    On mapping spaces this corresponds to the fiber sequence
    \[
        \Emb_M^\square(W, V) \too \Emb^\square(W, V) \too \Diff(M).
    \]
    The right map is a fibration (see \cite{cerf1961topologie}) and therefore we still have a fiber sequence in $\Scal$.
    This shows that the above square is also a pullback square of \categories{}.

    Pasting the above pullback square and \cref{prop:ArBord-pullback} yields a pullback square
    \[\begin{tikzcd}
        \boxMfd_{d,M} \ar[r] \ar[d] \ar[dr, phantom, very near start, "\lrcorner"] & \Ar(\Bord_d^\partial) \ar[d, "\ev_+"] \\
        \{M\} \ar[r] & \Bord_d^\partial
    \end{tikzcd}\]
    of \categories{}.
    By definition the slice $(\Bord_d^\partial)_{M/}$ is the pullback, so we get $\widetilde{\Psi}_S$ by comparing pullbacks.
    When $S = \emptyset$, this is a square of symmetric monoidal categories as in this case the functors in the first square preserve disjoint union.
    To get $\widetilde{\Phi}_S$ we combine $\widetilde{\Psi}_S$ with the anti-equivalence $(\Bord_d^\partial)^\op \cong \Bord_d^\partial$ that reverses bordisms.
\end{proof}

Before considering the pullback square in \cref{prop:ArBord-pullback}, which is specific to $\Bord_d^\partial$, we show that for a general Segal space $X_\bullet$, the simplicial nerve of a pullback as in \cref{prop:ArBord-pullback} is always given by the d\'ecalage of $X_\bullet$.
Recall that the d\'ecalage $X_{1+\bullet}$ is defined by precomposing $X\colon \Dop \to \Scal$ with the functor $(1+\bullet)\colon \Dop \to \Dop$ that adjoins a new initial object.
Restricting to this new initial object induces a map $X_{1+\bullet} \to X_0$ from the d\'ecalage to the constant simplicial space on $X_0$.

\begin{lem}
    For every Segal space $X_\bullet$ there is a natural pullback square
    \begin{equation}\label{eqn:ac-pullback}
    \begin{tikzcd}
        \ac(X_{1+\bullet}) \ar[r] \ar[d,""] \ar[dr, phantom, very near start, "\lrcorner"] & \Ar(\ac(X_\bullet)) \ar[d, "\ev_0"] \\
        X_0 \ar[r] & \ac(X_\bullet).
    \end{tikzcd}\end{equation}
\end{lem}
\begin{proof}
    By definition, $(X_{1+\bullet})_n = \Map(\Delta^0 \ast \Delta^n, X)$.
    We have a natural map $\Delta^1 \times \Delta^n \to \Delta^0 \ast \Delta^n = \Delta^{1+n}$ that sends $\{0\} \times \Delta^n$ to $\Delta^0$ and $\{1\} \times \Delta^n$ identically to $\Delta^n$.
    Mapping the square below on the left (which is natural in $\Delta^\bullet$) into $X$ and then applying $\ac(-)$ yields the square on the right
    \[\begin{tikzcd}
        \Delta^0 \ast \Delta^\bullet & \Delta^1 \times \Delta^\bullet \lar \\
        \Delta^0 \uar & \Delta^\bullet \lar \uar
    \end{tikzcd}
    \xtoo{\ac(\Map(-,X))}
    \begin{tikzcd}
        \ac(X_{1+\bullet}) \dar \rar & \ac(X^{\Delta^1}_\bullet)\dar \\
        X_0 \rar & \ac(X_\bullet).
    \end{tikzcd} \]
    The square in \cref{eqn:ac-pullback} is then obtained by using the map $\ac(X^{\Delta^1}_\bullet) \to \Ar(\ac(X_\bullet))$ that is adjoint to the map $(X_\bullet)^{\Delta^1} \to (\xN^r_\bullet \ac(X))^{\Delta^1} \simeq \xN^r_\bullet \Ar(\ac(X_\bullet))$.

    We can factor the square from \cref{eqn:ac-pullback} as the commutative diagram of \categories{}
\[\begin{tikzcd}
	{\ac(X_{1+\bullet})} & {\ac(X)^\simeq \times_{\ac(X)} \Ar(\ac(X))} & {\Ar(\ac(X))} \\
	{X_0\times \ac(X)} & {\ac(X)^\simeq\times \ac(X)} & {\ac(X) \times \ac(X)}\\
	{X_0} & {\ac(X)^\simeq} & {\ac(X)}
	\arrow[from=1-1, to=1-2]
	\arrow[from=1-1, to=2-1]
	\arrow[hook, from=1-2, to=1-3]
	\arrow[from=1-2, to=2-2]
	\arrow["\lrcorner"{anchor=center, pos=0.125}, draw=none, from=1-2, to=2-3]
	\arrow["\lrcorner"{anchor=center, pos=0.125}, draw=none, from=2-2, to=3-3]
	\arrow["{(\ev_0, \ev_1)}", from=1-3, to=2-3]
	\arrow["\pr_{\rm left}", from=2-3, to=3-3]
	\arrow[from=2-2, to=2-3]
	\arrow[from=2-1, to=2-2]
	\arrow[from=2-1, to=3-1]
	\arrow["\lrcorner"{anchor=center, pos=0.125}, draw=none, from=2-1, to=3-2]
	\arrow[from=2-2, to=3-2]
	\arrow[from=3-1, to=3-2]
	\arrow[hook, from=3-2, to=3-3]
\end{tikzcd}\]
    where the three squares labelled by ``$\lrcorner$'' are cartesian by construction.
    By pullback pasting it will thus suffice to show that the top left square is cartesian.
    This top left square can be obtained by applying $\ac(-)$ to the square
\[\begin{tikzcd}
	{X_{1+\bullet}} & {\xN^r(\ac(X)^\simeq \times_{\ac(X)} \Ar(\ac(X)))} \\
	{X_0\times X} & {\xN^r(\ac(X)^\simeq\times \ac(X))}
	\arrow[from=1-1, to=1-2]
	\arrow[from=1-1, to=2-1]
	\arrow[from=1-2, to=2-2]
	\arrow[from=2-1, to=2-2]
\end{tikzcd}\]
    The functor 
    \[
        (\ev_0,\ev_1)\colon \ac(X)^\simeq \times_{\ac(X)} \Ar(\ac(X)) \too \ac(X)^\simeq \times \ac(X)
    \]
    is a left fibration and thus applying the Rezk nerve, it results in a left fibration of Segal spaces.
    (See \cite[\S2.2]{HK22} for a definition.)
    The left map in the square $X_{1+\bullet} \to X_0 \times X_\bullet$ is a left fibration of Segal spaces because $X$ is Segal \cite[Lemma 2.10 (3)]{GCKT18}.
    As both vertical maps in the square are left fibrations it suffices to show that the square induces an equivalence on vertical fibers on $0$-simplices:
    this is indeed the case as both vertical fibers at some $(x,y) \in X_0 \times X_0$ compute the mapping space $\Map_{\ac(X)}(x,y)$ \cite[Corollary 3.15]{Rezk-nerve-25}.
    Moreover, this cartesian square is preserved by $\ac(-)$ by \cite[Proposition A.14]{HK22} as the two Segal spaces on the right are complete.
\end{proof}

To prove \cref{prop:ArBord-pullback} we thus have to show that the associated category of the d\'ecalage of the Segal space $\Bord_d^\partial[\bullet]$ agrees with the \category{} obtained from the (quasi-unital) topological category $\boxMfd_d$.

\begin{lem}
    There is an equivalence of \categories{}
    \[
        \boxMfd_d \simeq \ac(\Bord_d^\partial[1+\bullet]).
    \]
\end{lem}
\begin{proof}
    Let $\xN(\boxMfd_d)$ be the topologically enriched nerve of $\boxMfd_d$, which is a semi-simplicial Segal space.
    As $\boxMfd_d$ is quasi-unital, this is quasi-unital as a semi-simplicial Segal space and thus uniquely extends to a simplicial space.
    The \category{} associated to $\boxMfd_d$ (which we, by abuse of notation, also denote $\boxMfd_d$) is $\ac(\xN(\boxMfd_d))$.
    
    Recall that a map of Segal spaces $f\colon Y_\bullet \to Z_\bullet$ is a Dwyer--Kan equivalence if it is essentially surjective (i.e.~$Y_0 \to Z_0$ is surjective up to isomorphism in the homotopy category of $Z_\bullet$) and the map $((d_1,d_0), f_1) \colon Y_1 \to Y_0^{\times 2} \times_{Z_0^{\times 2}} Z_1$ is an equivalence.
    Note that a sufficient condition for essential surjectivity is that $f_0\colon Y_0 \to Z_0$ hits all path components.

    We will now construct a zigzag of Dwyer--Kan equivalences of semi-simplicial Segal spaces
    \[
        \xN(\boxMfd_d) \xtoo{I} X_\bullet \xleftarrow[\quad]{C} \Bord_d^\partial[1+\bullet].
    \]
    In fact, $C$ will be a (level-wise) equivalence of semi-simplicial spaces.
    Each of these semi-simplicial Segal space has quasi-units in the sense of \cite{Haugseng-semicategories} and the two functors preserve them.
    Thus, this uniquely extends to a zigzag of Dwyer--Kan equivalences of (simplicial) Segal spaces.
    The functor $\ac(-)$ inverts Dwyer--Kan equivalence, so applying it yields the desired equivalence.

    To construct the zigzag we work with semi-simplicial objects in the $1$-category $\mrm{TopGpd}$ of topological groupoids (where both objects and morphisms are topologised).
    We can then obtain the semi-simplicial spaces by composing with the realization functor $|-|\colon \mrm{TopGpd} \too \Scal$.

    The functor $X_\bullet \colon \Dop_\inj \too \mrm{TopGpd}$ will essentially be a topological version of the Rezk nerve.
    The objects of $X_n$ are $n$-tuples of embeddings $M_0 \hookrightarrow \dots \hookrightarrow M_n$ in $\boxMfd_d$.
    The morphisms of $X_n$ are $(n+1)$-tuples of diffeomorphisms making the diagram
\[\begin{tikzcd}
	{M_0} & {M_1} & \dots & {M_n} \\
	{N_0} & {N_1} & \dots & {N_n}
	\arrow["{i_1}", hook, from=1-1, to=1-2]
	\arrow["{\varphi_0}"', from=1-1, to=2-1]
	\arrow["{i_2}", hook, from=1-2, to=1-3]
	\arrow["{\varphi_1}"', from=1-2, to=2-2]
	\arrow["{i_n}", hook, from=1-3, to=1-4]
	\arrow["{\varphi_n}"', from=1-4, to=2-4]
	\arrow["{j_2}"', hook, from=2-1, to=2-2]
	\arrow["{j_2}"', hook, from=2-2, to=2-3]
	\arrow["{j_n}"', hook, from=2-3, to=2-4]
\end{tikzcd}\]
    commute.
    Both the space of objects and the space of morphisms is topologised by letting the embeddings and diffeomorphisms vary in the Whitney $\Ccal^\infty$-topology.
    The space of objects thus is the topological space $\xN_n(\boxMfd_d)$.
    Let $I\colon \xN_n(\boxMfd_d) \too X_n$ denote the inclusion, 
    where we think of $\xN_n(\boxMfd_d)$ as a topological groupoid where the only morphisms $(\varphi_i)$ are the identity morphisms.

    Next, we define a functor of topological groupoids
    \begin{align*}
        C\colon \Bord_d^\partial[1+n] & \too X_n \\
        (W, \mu) & \longmapsto \left( W_{[\mu(0),\mu(1)]} \hookrightarrow W_{[\mu(0),\mu(2)]} \hookrightarrow \dots \hookrightarrow W_{[\mu(0),\mu(n+1)]} \right),
    \end{align*}
    i.e.~a $[1+n]$-walled bordism $(W,t)$ in $\Rbb^\infty$ is mapped to the sequence of embeddings $M_0 \hookrightarrow \dots \hookrightarrow M_n$ where $M_i = W_{[\mu(0), \mu(i+1)]}$ is the part of $W$ between the $0$th and $(i+1)$st walls, and all the embeddings are the identity.
    Here we think of $W_{[\mu(0), \mu(i+1)]}$ as an object of $\boxMfd_d$ by setting $\partial_+ = W_{\mu(0)}$ and $\partial_- = W_{\mu(i+1)}$.
    On morphisms this functor is defined by restricting a diffeomorphism $W_{[\mu(0), \mu(n+1)]} \cong V_{[\mu'(0), \mu'(n+1)]}$ to each of the pieces.
    \\
    
    \textbf{Claim 1:}
    For all $n$ the map $C\colon \Bord_d^\partial[1+n] \to X_n$ of topological groupoids induces an equivalence of spaces after realization, and thus $C$ is an equivalence of semi-simplicial spaces. 

    We will deal with the case of $n=1$; the general case is similar.
    We will show that the topological functor $C_1\colon \Bord_d^\partial[1+1] \to X_1$ is a Dwyer--Kan equivalence and thus applying $|-|$ yields an equivalence of spaces.
    The functor is essentially surjective because for any object $(i\colon M_0 \hookrightarrow M_1)$ in $X_1$ we can find an embedding $j\colon M_1 \hookrightarrow [0,2]\times \Rbb^N$ such that $j(i(M_0)) = j(M_1) \cap [0,1] \times \Rbb^N$.
    Then $(j(M_1), \mu = (0,1,2)) \in \Bord_2^\partial[2]$ is a well-defined $[2]$-walled manifold and $C_1$ sends it to an object that is isomorphic (via $j$) to the one we started with.
    To establish fully faithfulness we need to show that the square
    \[\begin{tikzcd}
        {\mrm{Mor}(\Bord_d^\partial[2])} \dar \rar & 
        {\mrm{Mor}(X_1)} \dar["{(s,t)}"] \\
        {\mrm{Obj}(\Bord_d^\partial[2])^{\times 2}} \rar &
        {\mrm{Obj}(X_1)^{\times 2}}
    \end{tikzcd}\]
    is a homotopy pullback square.
    The bottom left space is discrete because $\Bord_d^\partial[n]$ was defined as a topologically enriched groupoid.
    The right vertical map is a Serre fibration.
    To see this we can decompose the right side of the diagram as a disjoint union and write
    \[
        \mrm{Obj}(X_1)^{\times 2} 
        \cong \left(\coprod_{M_0, M_1} \Emb^\square(M_0, M_1)\right)^{\times 2}
        \cong \coprod_{M_0, M_1, N_0, N_1} \Emb^\square(M_0, M_1) \times \Emb^\square(N_0, N_1)
    \]
    If we fix a choice of $M_i$ and $N_i$, this space has a locally retractile action of $\Diff(M_1) \times \Diff(N_1)$, see \cite[\S2]{CRW-retractile} for a discussion of locally retractile actions and how to use them to prove that maps are Serre fibrations.
    This group also acts on $\mrm{Mor}(X_1)$ by acting on the embeddings $M_0 \hookrightarrow M_1$ and $N_0 \hookrightarrow N_1$ and conjugating the diffeomorphism $M_1 \cong N_1$.
    Because the map $(s,t)$ is equivariant for the action and the base is locally retractile, it follows that $(s,t)$ is a Serre fibration.
    Thus, to show that the square is a homotopy pullback square it will suffice to compare the vertical fibers.
    Given two $[2]$-walled bordisms $(W,\mu)$ and $(V,\mu')$ the fiber on the left is the space of diffeomorphisms $W_{[\mu(0),\mu(2)]} \cong V_{[\mu'(0), \mu'(2)]}$ that are compatible with the walls.
    This maps to the fiber on the right, which is the space of dashed diffeomorphisms making the following diagram commute
\[\begin{tikzcd}
	{W_{[\mu(0),\mu(1)]}} \dar[dashed] \rar[hook] & {W_{[\mu(0),\mu(2)]}} \dar[dashed] \\
	{V_{[\mu(0),\mu(1)]}} \rar[hook] & {V_{[\mu'(0),\mu'(2)]}} .
\end{tikzcd}\]
    Here the diffeomorphism on the right determines a unique diffeomorphism on the left if and only if it is compatible with the walls.
    Thus, the square is a homotopy pullback, and we conclude that $C_n \colon \Bord_d^\partial[1+n] \to X_n$ is an equivalence for $n=1$ and thus for all $n$, proving claim 1.
    \\ 
    
    \textbf{Claim 2:}
    The semi-simplicial map $I\colon \xN_\bullet(\boxMfd_d) \to X_\bullet$ is a Dwyer--Kan  equivalence.
    
    Level-wise this map includes $\xN_n(\boxMfd_d)$ as the space of objects of the topological groupoid $X_n$.
    In particular, it is $\pi_0$ surjective and thus $I$ is essentially surjective.
    It hence remains to check that the square of topological groupoids
    \[\begin{tikzcd}
        {\xN_1(\boxMfd_d)} \ar[d] \ar[r] & {X_1} \ar[d] \\
        {\xN_0(\boxMfd_d)}^{\times 2} \ar[r] & {X_0}^{\times 2}
    \end{tikzcd}\]
    yields a pullback square of spaces after applying $|-|$.
    The spaces in the left column are exactly the object spaces of the topological groupoids to their right, so in order to use \cref{lem:eqf-pullback} to prove this, we have to check that the square
    \[\begin{tikzcd}
        {\mrm{Mor}(X_1)} \ar[d] \ar[r, "t"] & {\mrm{Obj}(X_1)} \ar[d] \\
        {\mrm{Mor}(X_0^{\times 2})} \ar[r, "t"] & {\mrm{Obj}(X_0^{\times 2})}
    \end{tikzcd}\]
    satisfies the assumptions therein. The top map is a Serre fibration as was shown in the proof of Claim 1. (In fact, we even know that $(s,t)$ is a Serre fibration.)
    The bottom map is a Serre fibration because its target is discrete.
    To see that the square is a pullback we spell out definitions:
    a point in the top left is a morphism $(\varphi_0, \varphi_1) \colon (i_1\colon M_0 \hookrightarrow M_1) \to (j_1\colon N_0 \hookrightarrow N_1)$ and a point in the pullback has the same data, except that the embedding $j_1$ is not specified.
    But $j_1$ is uniquely determined as $\varphi_1 \circ i_1 \circ \varphi_0^{-1}$, so this is indeed a pullback.
    Hence, \cref{lem:eqf-pullback} applies, completing the proof of Claim 2.
    
    Claim 1 and 2 imply that after applying $\ac(-)$ the maps $I$ and $C$ induce the desired equivalences
    \[
        \ac(\xN_\bullet(\boxMfd_d)) \xrightarrow[\simeq]{\ac(I)} X_\bullet \xleftarrow[\simeq]{\ac(C)} \ac(\Bord_d^\partial[1+\bullet]).
        \qedhere
    \]
\end{proof}

We still need to prove the lemma that we used in the above proof.
\begin{lem}\label{lem:eqf-pullback}
    Let $F\colon \Ccal \to \Dcal$ be a functor of topological groupoids such that
    \[\begin{tikzcd}
        {\mrm{Mor}(\Ccal)} \rar["t"] \dar["F"'] & {\mrm{Obj}(\Ccal)} \dar["F"] \\
        {\mrm{Mor}(\Dcal)} \rar["t"] & {\mrm{Obj}(\Dcal)}
    \end{tikzcd}\]
    is a pullback square of topological spaces and the horizontal maps (defined by sending a morphism to its target) are Serre fibrations.
    Then
    \[\begin{tikzcd}
        {\mrm{Obj}(\Ccal)} \rar \dar["F"'] & {|\Ccal|} \dar["{|F|}"] \\
        {\mrm{Obj}(\Dcal)} \rar & {|\Dcal|}
    \end{tikzcd}\]
    is a pullback square in $\Scal$.
\end{lem}
\begin{proof}
    We want to apply \cite[Proposition 1.6]{Segal1974-Cat-and-coh} to the map of simplicial spaces $\xN(\Ccal) \to \xN(\Dcal)$.
    This gives the desired conclusion, once we have checked that for all $d\colon [n] \leftarrow [m]$ the square
    \[\begin{tikzcd}
        {\xN_n(\Ccal)} \rar["d^*"] \dar & {\xN_m(\Ccal)} \dar \\
        {\xN_n(\Dcal)} \rar["d^*"] & {\xN_m(\Dcal)}
    \end{tikzcd}\]
    is homotopy cartesian. 
    Using pullback-pasting we can reduce this to the case of $m=0$ and using the Segal condition on $\xN_\bullet \Ccal$ and $\xN_\bullet \Dcal$ (which they satisfy in $\Scal$ by the assumption that the target maps are fibrations) we can further reduce to $n=1$.
    There are two morphisms $[1] \leftarrow [0]$:
    for one of them the square is homotopy cartesian by assumption, and for the other one the square is isomorphic to the former via the map that sends every morphism to its inverse.
\end{proof}